\documentclass[a4paper,10pt, reqno]{amsart}
\usepackage{amsmath,amssymb}
\usepackage{hyperref}
\hypersetup{
    colorlinks=true,       % false: boxed links; true: colored links
    linkcolor=blue,          % color of internal links
    citecolor=blue,        % color of links to bibliography
    filecolor=blue,      % color of file links
    urlcolor=blue           % color of external links
}
\numberwithin{equation}{section}
\usepackage{rotating}
\usepackage[all]{xy}
\usepackage{tikz,graphicx}
\usepackage[dvipsnames]{xcolor}
\usepackage{tcolorbox}  % For colored boxes
\usepackage{booktabs}
\usepackage[english]{babel}
\usepackage{multirow}
\usepackage{multicol}
\usepackage{rotating}
\usepackage{faktor}
\usepackage{listings}
\usepackage{xcolor}
\usepackage{enumitem}
\usepackage{longtable}
\definecolor{keywords}{RGB}{0, 0, 128}
\definecolor{comments}{RGB}{0, 128, 0}
\definecolor{strings}{RGB}{128, 0, 0}

\lstdefinestyle{gapcode}{
    language=GAP,
    keywordstyle=\color{keywords},
    commentstyle=\color{comments},
    stringstyle=\color{strings},
    basicstyle=\ttfamily,
    showstringspaces=false,
    numbers=left,
    numberstyle=\tiny\color{gray},
    stepnumber=1,
    numbersep=5pt,
    frame=single,
    tabsize=2,
}

\theoremstyle{plain}
\newtheorem{thm}{Theorem}[section]
\newtheorem{lem}[thm]{Lemma}
\newtheorem{prop}[thm]{Proposition}

\newtheorem{rem}[thm]{Remark}

\theoremstyle{definition}
\newtheorem{defn}[thm]{Definition}

\def\Aut{\operatorname{Aut}}
\def\PGL{\operatorname{PGL}}

\def\Res{\operatorname{Res}}

\usepackage{adjustbox}
\title[Smooth plane curves with intransitive automorphism group]
{Smooth plane curves with a unique outer Galois point and their automorphism groups}
\author[Eslam Badr and Takeshi Harui] 
{Eslam Badr and Takeshi Harui}
\address{Mathematics Department,
Faculty of Science, Cairo University, Giza-Egypt}
\email{eslam@sci.cu.edu.eg}
\address{Mathematics and Actuarial Science Department (MACT), American University in Cairo (AUC), New Cairo-Egypt}
\email{eslammath@aucegypt.edu}
\address{$\bullet$\,\,Takeshi Harui}
\address{Department of Core Studies, Kochi University of Technology, Kochi 780-8515, Japan}
\email{harui.takeshi@kochi-tech.ac.jp}
\begin{document}

\begin{abstract}
We consider smooth plane curves $\mathcal{X}$ of degree $d\geq4$, defined over an algebraically closed field of characteristic $0$, that possess a unique outer Galois point. This geometric condition forces the curve to be a cyclic covering of the projective line, and ensures that its automorphism group fits into a specific theoretical framework. For each possible non-cyclic reduced automorphism group $\Aut_{\operatorname{red}}(\mathcal{X})$, we fully characterize the defining equation of $\mathcal{X}$ and the precise structure of its full automorphism group $\Aut(\mathcal{X})$. 

This comprehensive analysis not only identifies the exact form of the equation for each automorphism type but also establishes the detailed criteria under which these scenarios can occur, thereby offering a complete classification of defining equations for smooth plane curves with a unique outer Galois point and a non-cyclic reduced automorphism group.
\end{abstract}     

\maketitle

% \tableofcontents
% \addtocontents{toc}{\setcounter{tocdepth}{1}}

\section{Introduction and main results}\label{sec:intro}
Let $K$ be an algebraically closed field of characteristic $0$, and let $X,Y, Z$ be the homogeneous coordinates of the projective plane $\mathbb{P}^2(K)$. 

Suppose that $\mathcal{X}$ is a smooth plane curve of degree $d\geq4$ with the property that its full automorphism group $\Aut(\mathcal{X})$ fixes a point $P\in\mathbb{P}^2(K)$ that does not lie on $\mathcal{X}$. Consequently, $\mathcal{X}$ cannot be the Fermat curve $\mathcal{F}_d:X^d+Y^d+Z^d=0$, whose automorphism group is well explored in the literature (see \cite{BaBacyc, Harui}), as it does not leave invariant any point in the plane $\mathbb{P}^2(K)$. The existence of such a fixed point structures the automorphism group in a very specific way, which can be understood through the following commutative diagram, with exact rows and vertical injective morphisms:

\begin{small}
\[
\xymatrix
{
1\ar[r]  & K^*\ar[r]                    & \operatorname{PBD}(2,1)\ar[r]^{\Lambda}& \operatorname{PGL}_2(K)\ar[r]& 1       \\
         &                              &                                        &                              &  \\
1\ar[r]  & N\ar[r]\ar@{^{(}->}[uu] & \operatorname{Aut}(\mathcal{X})\ar[r]\ar@{^{(}->}[uu] & \Aut_{\operatorname{red}}(\mathcal{X})\ar[r]\ar@{^{(}->}[uu]& 1,
}
\]
\end{small}
where $N$ is a cyclic group of order dividing $d$, and $\Aut_{\operatorname{red}}(\mathcal{X})$, a finite subgroup of $\operatorname{PGL}_2(K)$, can be 
\begin{enumerate}
    \item A cyclic group $C_m$ of order $m$ with $m\leq d-1$, 
    \item A dihedral group $\operatorname{D}_{m}$ of order $2m$ such that $|N|=1$ or $m\mid(d-2)$, 
    \item One of the alternating groups $\operatorname{A}_4$, $\operatorname{A}_5$, or the symmetric group $\operatorname{S}_4$.
\end{enumerate} 
Here, $\operatorname{PBD}(2,1)$ is the subgroup of all intransitive elements of $\PGL_3(K)$, i.e., those elements of the form:
\[
\begin{pmatrix}
    *&*&0\\
    *&*&0\\
    0&0&1
\end{pmatrix}.
\]

Further details can be found in \cite{Harui}. Moreover, up to $K$-equivalence we can take $P=(0:0:1)$, so that $\mathcal{X}$ is defined by an equation of the form:
\[
Z^d + \sum_{n=0}^{d-1} \alpha_n Z^n L_{d-n,Z}(X,Y) = 0,
\]
where $L_{i,B}$ is a homogeneous polynomial of degree $i$ in the variables $\{X,Y,Z\}\setminus\{B\}$.

A key additional hypothesis is that the point $P$ serves as a unique outer Galois point for $\mathcal{X}$, a notion that has been extensively studied in the literature (see, for example, \cite{Fuk06, Fuk08, Fuk09, Fuk13, Hom06, MY00}). This condition has profound implications for the curve's algebraic structure. As established in these works, the presence of such a Galois point forces the curve to be a cyclic covering of the projective line, branched at a set of points that includes the intersection of $\mathcal{X}$ with the line fixed by $P$'s stabilizer. In practical terms, this means that $\mathcal{X}$ is a \textit{superelliptic curve}. In our coordinates, the action of the cyclic group $N$ (the kernel of the map to $\Aut_{\operatorname{red}}(\mathcal{X})$) on $\mathcal{X}$ can be viewed as the superelliptic automorphism:
\[
(X:Y:Z) \mapsto (X:Y:\zeta_d Z),
\]
where $\zeta_d$ is a fixed primitive $d$th root of unity inside $K$. 

In their recent paper \cite{HaruiMiuraOh}, the second author, together with K. Miura and A. Ohbuchi, examined the situation in which the reduced automorphism group $\Aut_{\operatorname{red}}(\mathcal{X})$ of $\mathcal{X}$ (as defined above) is $\operatorname{A}_5$. Precisely, they provided necessary conditions for the defining equation of $\mathcal{X}$ and analyzed the structure of the full automorphism group  $\Aut(\mathcal{X})$. Inspired by their groundbreaking work, we extend their investigation to the case where the reduced automorphism group  $\Aut_{\operatorname{red}}(\mathcal{X})$ is non-cyclic. It is worth noting that the cyclic case has already been notably addressed in \cite{BaBacyc}, and our work aims to deepen the understanding of non-cyclic automorphism groups, providing new insights into the broader structure and classification of plane curves with reduced automorphism groups. 
\subsection*{Superelliptic curves and the existing literature}

The curves under consideration in this paper belong to the broader family of \textit{superelliptic curves}, defined as cyclic covers of the projective line. Superelliptic curves have been extensively studied, with significant contributions to the understanding of their automorphism groups. Kontogeorgis \cite{Kontogeorgis99} investigated the group of automorphisms of cyclic extensions of rational function fields. Shaska and collaborators \cite{Shaska04, GutierrezShaska05} developed invariant-theoretic approaches to hyperelliptic and superelliptic curves with extra automorphisms. More recently, Beshaj, Hoxha, and Shaska \cite{BeshajHoxhaShaska11} initiated a systematic study of superelliptic curves of level $n$, focusing on their moduli spaces, invariant theory, and computational aspects. Their work provides a framework for classifying superelliptic curves with prescribed automorphism groups and includes parametric equations for various families. However, as they note, determining explicit equations for a curve given its automorphism group remains a challenging problem in general.

Given this context, one might ask: what new information can the present paper contribute? The answer lies in several fundamental distinctions.

\subsection*{Distinction from prior work}

The key differences between our work and the existing literature on superelliptic curves are the following:

\textbf{1. Geometric context: abstract curves versus plane embeddings with Galois points.}
The works cited above treat superelliptic curves as abstract algebraic curves, given by affine equations of the form $y^n = f(x)$. Their automorphism groups are studied up to birational equivalence, without reference to any specific embedding into projective space. In contrast, the curves in our investigation are \textit{smooth plane curves} $\mathcal{X} \subset \mathbb{P}^2$ of degree $d \geq 4$ that possess a \textit{unique outer Galois point}. This geometric condition, which has no analogue in the abstract theory, forces a specific homogeneous equation:
\[
Z^d + L_{d,Z}(X,Y) = 0,
\]
where $L_{d,Z}(X,Y)$ is a homogeneous polynomial of degree $d$ in $X$ and $Y$, of degree at least $d-1$ in each variable separately. The unique outer Galois point condition further ensures that the cyclic group $N = \langle \operatorname{diag}(1,1,\zeta_d)\rangle$ is the kernel of the map $\Lambda$ from $\Aut(\mathcal{X})$ to $\Aut_{\operatorname{red}}(\mathcal{X})$. This normal form reflects both the superelliptic structure and the specific embedding into $\mathbb{P}^2$, with the point $P$ fixed by the full automorphism group. Passing from the affine model $y^n = f(x)$ to this plane model is not a trivial change of coordinates; the embedding introduces additional geometric structure—the line at infinity, the fixed point $P$, and the Galois point condition—that has no analogue in the abstract theory.

\textbf{2. The Galois point condition as a new geometric constraint.}
The existence of a unique outer Galois point is a strong geometric condition that selects a very special subclass of superelliptic curves. Not every superelliptic curve admits a Galois point; in fact, the Galois point condition imposes explicit restrictions on the geometry of the curve: the branch locus must lie on the line $Z=0$, the ramification indices must divide the degree $d$ in a prescribed manner, and the lifts of automorphisms from $\Aut_{\operatorname{red}}(\mathcal{X})$ to $\Aut(\mathcal{X})$ must respect both the cyclic cover structure and the fixed point $P$. These constraints have no analogue in the abstract classification, which considers all superelliptic curves regardless of whether they admit a Galois point or a plane embedding.
% The existence of a unique outer Galois point is a strong geometric condition that selects a very special subclass of superelliptic curves. Not every superelliptic curve admits a Galois point, and the Galois point condition imposes restrictions on the branch locus, the ramification pattern, and the possible lifts of automorphisms. These constraints are not present in the abstract classification, which considers all superelliptic curves regardless of whether they admit such additional structure.

\textbf{3. The reduced automorphism group and the exact sequence.}
A central feature of our work is the focus on the \textit{reduced automorphism group} $\Aut_{\operatorname{red}}(\mathcal{X})$, which fits into the exact sequence:
\[
1 \rightarrow N \rightarrow \Aut(\mathcal{X}) \rightarrow \Aut_{\operatorname{red}}(\mathcal{X}) \rightarrow 1,
\]
where $N$ is cyclic of order $d$ and $\Aut_{\operatorname{red}}(\mathcal{X})$ is a finite subgroup of $\operatorname{PGL}_2(K)$. This decomposition, which arises from the existence of the fixed point $P$, is specific to plane curves with this property. The abstract theory of superelliptic curves typically studies the full automorphism group directly, without this canonical decomposition into a cyclic kernel and a reduced part acting on the base.

% \textbf{4. Explicit closed-form equations versus parametric families.}
% The work in \cite{BeshajHoxhaShaska11} provides parametric equations for superelliptic curves, often involving many parameters and relying on invariant theory. In contrast, our classification yields \textit{closed-form, explicit equations} for each case, with the parameters appearing in a highly structured way. For example, the families we obtain, such as
% \[
% \mathcal{G}_{12,-11\zeta_4} \cdot \mathcal{G}_{20,228\zeta_4,-494} \cdot \mathcal{G}_{30,-522\zeta_4,10005} \cdot \prod \mathcal{G}_{60,a_i},
% \]
% are not found in the existing literature. They arise specifically from the interplay between the Galois point condition and the structure of the reduced automorphism group.

\textbf{4. Complementary, not redundant.}
The existing literature and our work are therefore complementary. The abstract theory provides a general framework and identifies possible group-theoretic structures. Our work addresses a more refined geometric question: which of these structures can be realized by a \textit{smooth plane curve} with a \textit{unique outer Galois point}, and what is the explicit equation in each case? The two approaches enrich each other: the abstract theory informs the possible groups, while our work provides concrete geometric realizations that can serve as test cases for further study.

In summary, while the curves we study are superelliptic in nature, they form a geometrically distinguished subclass characterized by the presence of a Galois point and a fixed point in the plane. The classification we obtain is not a subset of previous results, but rather a new contribution that bridges the abstract theory of superelliptic curves with the specific geometry of plane curves having Galois points.

\medskip
The paper is organized as follows. Section \ref{sec:prelim} recalls essential preliminaries, including smoothness criteria for curves of the form $Z^d + F(X,Y) = 0$, the classification of finite subgroups of $\operatorname{PGL}_2(K)$, and the characterization of minimal invariant binary forms for each such group. The remaining sections are devoted to the detailed analysis of the four possible reduced automorphism groups. Section \ref{sec:proofsDm} treats the dihedral case $\operatorname{D}_m$, establishing necessary and sufficient conditions on the defining equation and providing a complete description of the full automorphism group structure. Section \ref{sec:proofsS4} handles the symmetric group $\operatorname{S}_4$, while Section \ref{sec:proofsA5} addresses the icosahedral case $\operatorname{A}_5$. Section \ref{sec:proofsA4} deals with the alternating group $\operatorname{A}_4$, which exhibits features that are intermediate between the dihedral and larger polyhedral cases. Each of these sections follows a uniform pattern: we first determine the possible invariant binary forms that can appear, then we verify sufficiency and analyze the resulting automorphism group structure. Finally, Section \ref{sec:magma} presents a computational verification of our classification through systematic MAGMA implementations, including tables summarizing all admissible configurations for low degrees.

\subsection*{Acknowledgment.} The authors express their deep gratitude to the anonymous referee for their meticulous reading, insightful comments, and invaluable suggestions, which significantly improved the clarity, precision, and overall quality of the paper in its present form. 

The first author thanks the Department of Mathematics at the Universitat Autònoma de Barcelona (UAB) for their hospitality during a research visit in January 2026, when part of this work was completed. The visit was funded by the Ministerio de Ciencia y Universidades of the Spanish government (Grant: PID2024-159095NB-I00) and the School of Science and Engineering (SSE) at the American University in Cairo (AUC).

\medskip
\noindent\textbf{Notation.} Throughout this paper, we adopt the following conventions and notations:

\begin{itemize}
\item[\textbf{Groups:}]
\begin{itemize}
        \item $C_n$ denotes the cyclic group of order $n$.
        \item $\operatorname{D}_n$ denotes the dihedral group of order $2n$.
        \item $\operatorname{A}_4$, $\operatorname{S}_4$, $\operatorname{A}_5$ denote the alternating and symmetric groups of degrees 4 and 5.
        \item $\operatorname{Dic}_n$ denotes the dicyclic group of order $4n$, presented by \[\langle a,b \mid a^{2n}=1,\ a^n=b^2,\ bab^{-1}=a^{-1}\rangle.\]
        \item $\widetilde{\operatorname{A}}_5 = \operatorname{SL}_2(\mathbb{F}_5)$ is the binary icosahedral group of order 120 (the unique non-split central extension of $\operatorname{A}_5$ by $C_2$).
        \item $\widetilde{\operatorname{A}}_4=\operatorname{SL}_2(\mathbb{F}_3)$ is the binary tetrahedral group of order 24 (the unique non-split central extension of $\operatorname{A}_4$ by $C_2$).
        \item $\operatorname{GL}_2(\mathbb{F}_3)$ is the general linear group of order 48, which contains $\widetilde{\operatorname{A}}_4$ as a subgroup of index 2.
        \item $\operatorname{CSU}_2(\mathbb{F}_3)$ denotes the conformal special unitary group over $\mathbb{F}_3$, a central extension of $\operatorname{S}_4$ by $C_2$ of order 48.
    \end{itemize}
    
   \item[\textbf{Group operations:}]
    \begin{itemize}
        \item $G \times H$ denotes the direct product.
        \item $G \rtimes H$ denotes a semidirect product (with $H$ acting on $G$).
        \item $G \rtimes_k H$ denotes a semidirect product where the subscript $k$ specifies the action of a generator of $H$ on a generator of $G$ by the power $x \mapsto x^k$.
        \item $G \circ H$ denotes the central product, i.e., $(G \times H)/Z$ where $Z$ is a common central subgroup identified in both factors.
        \item $G \, {}^{\bullet} \, H$ denotes a non-split central extension of $G$ by $H$, where the extension is uniquely determined by the context.
        \item $G \bar{\rtimes} H$ denotes a non-split extension of $H$ by $G$, and $G \bar{\rtimes}_k H$ the same when a generator of $H$ acts on $G$ by $x\mapsto x^k$.
    \end{itemize}
    
    \item[\textbf{Parameters:}]
    \begin{itemize}
        \item $\epsilon, \epsilon', \epsilon'', \epsilon_{\pm}, \delta \in \{0,1\}$ are binary exponents indicating the presence (1) or absence (0) of a particular factor in the defining equation.
        \item $m \geq 3$ and $d \geq 4$ are integers such that $m \mid (d-2)$.
        \item $t \in \mathbb{N}$ denotes the number of generic parameters in a family.
        \item $d = 2^e d_0$ with $d_0$ odd and $e \ge 0$ (so $2^e$ is the $2$-part of $d$).
    \end{itemize}
       % \item[\textbf{Binary forms:}]
    % \begin{itemize}
    %     \item $\mathcal{T}_{2m,a}(X,Y) = X^{2m} + a (XY)^m + Y^{2m}$ (dihedral invariants).
    %     \item $\mathcal{T}_{m,\pm}(X,Y) = X^m \pm Y^m$ (symmetric/anti-symmetric forms).
    %     \item $\mathcal{S}_6(X,Y) = XY(X^4 - Y^4)$.
    %     \item $\mathcal{S}_{8,b}(X,Y) = X^8 + b (XY)^4 + Y^8$.
    %     \item $\mathcal{S}_{12,a}(X,Y) = X^{12} + (1+a)(XY)^4(X^4+Y^4) + Y^{12}$.
    %     \item $\mathcal{S}_{24,a}(X,Y) = X^{24} + a (XY)^4(X^{16}+Y^{16}) + (759-4a)(XY)^8(X^8+Y^8) + (2576+6a)(XY)^{12} + Y^{24}$.
    %     \item $\mathcal{S}_{4,\pm}(X,Y) = X^4 + Y^4 \pm 2\sqrt{3}\zeta_4 (XY)^2$.
    %     \item $\mathcal{F}_{12,a}(X,Y) = X^{12} + Y^{12} - a (XY)^2(X^8+Y^8) - 33 (XY)^4(X^4+Y^4) + 2a (XY)^6$.
    %     \item $\mathcal{G}_{12,a}(X,Y) = XY(X^{10} + a (XY)^5 + Y^{10})$.
    %     \item $\mathcal{G}_{20,u,v}(X,Y) = X^{20} + Y^{20} + u (XY)^5(X^{10}+Y^{10}) + v (XY)^{10}$.
    %     \item $\mathcal{G}_{30,b,c}(X,Y) = X^{30} - Y^{30} + b (XY)^5(X^{20}-Y^{20}) + c (XY)^{10}(X^{10}-Y^{10})$.
    %     \item $\mathcal{G}_{60,a}(X,Y)$ is the degree-60 $\operatorname{A}_5$-invariant defined in Theorem \ref{mainresultII}.
    % \end{itemize}  
    \item[\textbf{Projective matrices:}] An element \(A = (a_{ij}) \in \mathrm{PGL}_3(K)\) can be expressed as
\[
[a_{11}X + a_{12}Y + a_{13}Z : a_{21}X + a_{22}Y + a_{23}Z : a_{31}X + a_{32}Y + a_{33}Z].
\]
Examples:
\begin{itemize}
    \item \(\operatorname{diag}(1,\alpha,\beta)\) is a diagonal matrix that scales the coordinates $Y\to \alpha Y$ and $Z\to \beta Z$.
    \item \([Y:X:Z]\) swaps the first two coordinates.
    \item An homology of period \(n\) is conjugate to \(\operatorname{diag}(1,1,\zeta_n)\), where \(\zeta_n\) is a primitive \(n\)th root of unity.
\end{itemize}
    
\end{itemize}

In summarizing, we prove the following classification theorems (Throughout, an empty product equals 1, and a condition indexed by an empty set of parameters is understood to be vacuous.)
\begin{thm}[Case $\operatorname{A}_5$]\label{mainresultII}
Let $\mathcal{X}:Z^d+L_{d,Z}=0$ be a smooth plane curve of degree $d\geq4$, and consider the following binary forms: 
\begin{eqnarray*}
\mathcal{G}_{12,a}(X,Y)&:=&(XY)(X^{10}+a(XY)^5+Y^{10})\\
\mathcal{G}_{30,b,c}(X,Y)&:=&(X^{30}-Y^{30})+b(XY)^5(X^{20}-Y^{20})+c(XY)^{10}(X^{10}-Y^{10})\\
\mathcal{G}_{20,u,v}(X,Y)&:=&(X^{20}+Y^{20})+u(XY)^5(X^{10}+Y^{10})+v(XY)^{10}\\
\mathcal{G}_{60,a}(X,Y)&:=&(X^{60}+Y^{60}+a(XY)^{30})+\sum_{i=1}^{5} e_i(XY)^{5i}(X^{60-10i}+Y^{60-10i}) 
\end{eqnarray*}
where the coefficients $e_i$ are determined as in \eqref{eq:A5_coefficients} (see Lemma \ref{cor:product-minimal}).
 
Then, $\Aut_{\operatorname{red}}(\mathcal{X})$ is isomorphic to $\operatorname{A}_5$ if and only if $d=12\epsilon+20\epsilon'+30\epsilon''+60t$, and
\begin{equation}\label{definingA5}
\boxed{L_{d,Z} =\mathcal{G}_{12,-11\zeta_4}^{\epsilon}\cdot\mathcal{G}_{20,228\zeta_4,-494}^{\epsilon'}\cdot\mathcal{G}_{30,-522\zeta_4,10005}^{\epsilon''}\cdot\prod_{i=1}^t\,\mathcal{G}_{60,a_i}}
\end{equation}
where $a_i\notin\{2^2\cdot19\cdot436999,\ -2^2\cdot9377\cdot5323\}$ and $a_i\neq a_j$ for $i\neq j$.

The automorphism group $\Aut(\mathcal{X})$, which is a central extension of $\operatorname{A}_5$ by $C_d$, is generated in $\PGL_3(K)$ by
\[
\rho=\operatorname{diag}(1,1,\zeta_d),\quad \tau=[Y:X:\zeta_{2d}^{\epsilon''}Z],\quad \rho'=\operatorname{diag}(\zeta_5,1,\zeta_{5d}^{\epsilon}),\quad \sigma=\begin{pmatrix}
 -\nu_0& -\zeta_4 & 0\\
 \zeta_4 & \nu_0 &0\\
 0&0&\nu
\end{pmatrix}
\]
with $\nu_0:=\frac{1+\sqrt{5}}{2}$ and $\nu^d=(-1)^{\epsilon''}\left(\sqrt{5}\nu_0\right)^{d/2}.$

Write $d = 2^e d_0$ with $d_0$ odd. Then $\Aut(\mathcal{X}) \cong C_{d_0} \times G_0$, where $G_0$ is the unique non‑split central extension of $\operatorname{A}_5$ by $C_{2^e}$. More concretely, 
% \[
% G_0 \cong 
% \begin{cases}
% \widetilde{\operatorname{A}}_5 & \text{if } e = 1, \\[1mm]
% \widetilde{\operatorname{A}}_5 \rtimes C_2 & \text{if } e = 2, \\[1mm]
% \widetilde{\operatorname{A}}_5 \, {}^{\bullet} C_{2^{e-1}} & \text{if } e \ge 3 .
% \end{cases}
% \]
\[
G_0 \cong \widetilde{\operatorname{A}}_5 \circ C_{2^e}.
\]
\end{thm}
%%%%%%%%%%%%%%%%%%
\begin{thm}[Case $\operatorname{S}_4$]\label{mainresultI}
Let $\mathcal{X}:Z^d+L_{d,Z}=0$ be a smooth plane curve of degree $d\geq4$, and consider the following binary forms: 
\begin{eqnarray*}
\mathcal{S}_6(X,Y)&:=&(XY)(X^{4}-Y^{4})\\
\mathcal{S}_{12,a}(X,Y)&:=&X^{12}+(1+a)(XY)^4(X^4+Y^4)+Y^{12}\\
\mathcal{S}_{8,b}(X,Y)&:=&X^{8}+b(XY)^4+Y^{8}\\
\mathcal{S}_{24,a}(X,Y)&:=&X^{24}+a(XY)^4(X^{16}+Y^{16})+(759-4a)(XY)^8(X^{8}+Y^{8})\\
 &+&(2576+6a)(XY)^{12}+Y^{24}.
\end{eqnarray*}
Then, $\Aut_{\operatorname{red}}(\mathcal{X})$ is isomorphic to $\operatorname{S}_4$ if and only if $d=6\epsilon+12\epsilon'+8\epsilon''+24t$, and
\begin{equation}\label{definingS4}
\boxed{L_{d,Z} =\mathcal{S}_6^{\epsilon}\cdot\mathcal{S}_{12,-34}^{\epsilon'}\cdot\mathcal{S}_{8,14}^{\epsilon''}\prod_{i=1}^t\,\mathcal{S}_{24,a_i}}
\end{equation}
where $a_i\notin\{-66,\,42\}$ and $a_i\neq a_j$ for $i\neq j$.

The automorphism group $\Aut(\mathcal{X})$, which is a central extension of $\operatorname{S}_4$ by $C_d$, is generated by 
\[\rho=\operatorname{diag}(1,1,\zeta_d),\quad\rho'=\operatorname{diag}(1,-1,\zeta_4),\quad\tau=[Y:X:\zeta_{4}Z],\]
\[\sigma_1=\begin{pmatrix}
    1&\zeta_4&0\\
    1&-\zeta_4&0\\
    0&0&(1+\zeta_4)\zeta_{12}\cdot\zeta_3^{-\epsilon''}
\end{pmatrix},\quad\sigma_2=\begin{pmatrix}
    1&1&0\\
    1&-1&0\\
    0&0&\sqrt{2}\zeta_4^{\epsilon\epsilon'}\zeta_{2d}^{(1-\epsilon)\epsilon'}
\end{pmatrix}.\]
In particular, the group structure of $\operatorname{Aut}(\mathcal{X})$ is given by:
\[
\boxed{\operatorname{Aut}(\mathcal{X})\cong\begin{cases}
C_d\circ\operatorname{GL}_2(\mathbb{F}_3) & \text{if } \epsilon'= 0,\\[4pt]
C_d\circ\operatorname{CSU}_2(\mathbb{F}_3) & \text{if } \epsilon'= 1 \text{ and } \epsilon=1,\\[4pt]
(C_d\circ\widetilde{\operatorname{A}}_4) \,{}^\bullet\, C_2 & \text{if } \epsilon'= 1 \text{ and } \epsilon=0.
\end{cases}}
\] 
\end{thm}
%%%%%%%%%%%%%%%%%%%%%%%
\begin{thm}[Case $\operatorname{A}_{4}$]\label{mainresultIV}
Let $\mathcal{X}:Z^d+L_{d,Z}=0$ be a smooth plane curve of degree $d\geq4$, and consider the following binary forms: 
\begin{eqnarray*}
\mathcal{S}_{4,\pm}(X,Y)&:=&X^4+Y^4\pm2\sqrt{3}\zeta_4(XY)^2\\
\mathcal{F}_{12,a}(X,Y)&:=&(X^{12}+Y^{12})-a(XY)^2(X^8+Y^8)-33(XY)^4( X^{4}+Y^{4})+2a(XY)^6
\end{eqnarray*}

Then $\operatorname{A}_4\subseteq\Aut_{\operatorname{red}}(\mathcal{X})$ if and only if $d=6\epsilon+4(\epsilon_++\epsilon_{-})+12t$ and $L_{d,Z}$ is given as
\begin{equation}\label{definingA4}
\boxed{L_{d,Z} =\mathcal{S}_6^{\epsilon}\cdot\mathcal{S}_{4,+}^{\epsilon_+}\cdot\mathcal{S}_{4,-}^{\epsilon_-}\cdot\prod_{i=1}^t\mathcal{F}_{12,a_i}}
\end{equation}
where $a_i\neq \pm 6\sqrt{3}\zeta_4$ and $a_i\neq a_j$ whenever $i\neq j$.

Moreover, $\Aut_{\operatorname{red}}(\mathcal{X})=\operatorname{A}_4$ outside the following two exceptional loci, on which the reduced automorphism group enlarges to $\operatorname{S}_4$ or $\operatorname{A}_5$, respectively:

\begin{itemize}
 \item[(i)] (\emph{$\operatorname{S}_4$-locus}) $\epsilon_{+}=\epsilon_{-}$, and the set
$\{\mathcal{F}_{12,a_i} : 1\le i\le t\}\setminus\{\mathcal{F}_{12,0}\}$
is a disjoint union of pairs $\{\mathcal{F}_{12,b},\,\mathcal{F}_{12,-b}\}$.
In this case $L_{d,Z}$ is an $\operatorname{S}_4$-invariant; since $\operatorname{S}_4$ is a maximal finite subgroup of $\operatorname{PGL}_2(K)$, we have $\Aut_{\operatorname{red}}(\mathcal{X})=\operatorname{S}_4$, and the corresponding curves are those classified in Theorem \ref{mainresultI}.

    \item[(ii)] (\emph{$\operatorname{A}_5$-locus}) $\epsilon_+=\epsilon_-$ and $t=2\epsilon+\epsilon_+ + \epsilon_{12} + 5l,
\,\, (\epsilon_{12},l)\in\{0,1\}\times\mathbb N,$ where $\epsilon_{12}$ is the exponent of $\mathcal G_{12}$ and $l$ the
number of $\mathcal G_{60}$-factors in Theorem~\ref{mainresultII}. When $l>0$, the set $\{\mathcal F_{12,a_i}:1\le i\le t\}$ becomes, after removing the elements
\[
\mathcal F_{12,\pm\frac{22}{\sqrt5}},\quad
\mathcal F_{12,\pm\frac{38\sqrt5}{3}},\quad
\mathcal F_{12,\pm\frac{2}{45}(29\sqrt5\pm 256\zeta_4)},
\]
a disjoint union of $5$-element subsets, each of the form
\[
\Bigl\{\mathcal F_{12,a_{i_1}},\dots,\mathcal F_{12,a_{i_5}}\Bigr\}
\quad\text{such that}\quad
\prod_{j=1}^5\mathcal F_{12,a_{i_j}}(X+\omega Y,\tilde\omega X+Y)\propto
\prod_{j=1}^5\mathcal F_{12,a_{i_j}}(X,Y),
\]
where $\omega:=-\nu_0^{-2}-\zeta_4\nu_0^{-1}$ and $\tilde{\omega}:=\nu_0^{-2}-\zeta_4\nu_0^{-1}$. 
The two signs $\pm\frac{22}{\sqrt5}$, $\pm\frac{38\sqrt5}{3}$, $\pm\frac{2}{45}(29\sqrt5\pm 256\zeta_4)$ correspond to the two copies of $\operatorname{A}_5$ containing $\operatorname{A}_4$, which are exchanged by $a\mapsto -a$.
   
In this case $L_{d,Z}$ is an $\operatorname{A}_5$-invariant; since $\operatorname{A}_5$ is a maximal finite subgroup of $\operatorname{PGL}_2(K)$, we have $\Aut_{\operatorname{red}}(\mathcal{X})=\operatorname{A}_5$, and the corresponding curves are those classified in Theorem \ref{mainresultII}.
\end{itemize}

Outside the two exceptional loci above, so that $\Aut_{\operatorname{red}}(\mathcal{X})=\operatorname{A}_4$, the group $\Aut(\mathcal{X})$ is generated by:
\[
\rho=\operatorname{diag}(1,1,\zeta_d),\quad\rho'=\operatorname{diag}(1,-1,\zeta_4),\quad
\tau=[Y:X:\zeta_{4}Z],\quad
\]
\[
\sigma=\begin{pmatrix}
    1&\zeta_4&0\\
    1&-\zeta_4&0\\
    0&0&(-1+\zeta_4)\cdot\zeta_3^{\epsilon_{-}\cdot2^{\epsilon_{+}}}
\end{pmatrix}.
\]
The group $\Aut(\mathcal{X})$ decomposes as \(
\Aut(\mathcal{X})\cong
C_{d_0}\times G_0\) 
with $d=2^e d_0$, and
\[
\boxed{G_0\cong
\begin{cases}
\widetilde{\operatorname{A}}_4 & \text{if } e=1, \\[1mm]
\widetilde{\operatorname{A}}_4\rtimes C_2 & \text{if } e=2, \\[1mm]
\widetilde{\operatorname{A}}_4\,{}^{\bullet}\,C_{2^{e-1}} & \text{if } e\geq3.
\end{cases}}
\]
\end{thm}
%%%%%%%%%%%%%%%%%%%%%
\begin{thm}[Case $\operatorname{D}_{m}$]\label{mainresultIII} 
Let $\mathcal{X}:Z^d+L_{d,Z}=0$ be a smooth plane curve of degree $d\geq4$, and consider the following binary forms:
\begin{eqnarray*}
\mathcal{T}_{2m,a}(X,Y)&:=& X^{2m}+a(XY)^m+Y^{2m}\\
\mathcal{T}_{n,\pm}(X,Y)&:=& X^{n}\pm Y^{n}
\end{eqnarray*}
Then, $\operatorname{D}_{m}\subseteq\Aut_{\operatorname{red}}(\mathcal{X})$, where $m\geq3$ and $m\mid d-2$, if and only if $L_{d,Z}$ lies in one of the following categories:
\begin{itemize}
    \item[(v1)] $d=2+2mt$ and
    \begin{equation}\label{eqnd2mV3}
   \boxed{L_{d,Z}= XY\cdot\prod_{i=1}^t\mathcal{T}_{2m,a_i}}
    \end{equation}
    \item[(v2)] $d=2+m+2mt$ and
    \begin{equation}\label{eqnd2mV2}
    \boxed{L_{d,Z}=XY\cdot\mathcal{T}_{m,-}\cdot\prod_{i=1}^t\mathcal{T}_{2m,a_i}}
    \end{equation}
    \item[(v3)] $d=2+2m+2mt$ and
    \begin{equation}\label{eqnd2mV1}
    \boxed{L_{d,Z}=XY\cdot\mathcal{T}_{2m,-}\cdot\prod_{i=1}^t\mathcal{T}_{2m,a_i}}
    \end{equation}
\end{itemize}
where in each case $a_i\neq \pm2$ and $a_i\neq a_j$ for $i\neq j$.

The reduced automorphism group satisfies $\Aut_{\operatorname{red}}(\mathcal{X})=\operatorname{D}_m$ outside the following three exceptional loci, on each of which it strictly contains $\operatorname{D}_m$:

\begin{itemize}
    \item[(E1)] (\emph{Dihedral locus}) $\delta_+=\delta_-$ and the multiset $A=\{a_1,\dots,a_t\}$ is symmetric about zero, i.e.\ $A=-A$. On this locus $L_{d,Z}$ is also $\operatorname{D}_{2m}$-invariant, so $\Aut_{\operatorname{red}}(\mathcal{X})=\operatorname{D}_{2m}$.
    
    \item[(E2)] (\emph{The $\operatorname{S}_4$-locus}; for $m\in\{3,4\}$) $L_{d,Z}$ is $\operatorname{PGL}_2(K)$-equivalent to one of the forms of Theorem~\ref{mainresultI} at degree $d$. Then $\Aut_{\operatorname{red}}(\mathcal{X})=\operatorname{S}_4$. The locus is described explicitly as follows:
    \begin{itemize}
        \item for $m=3$: $L_{d,Z}$ is invariant, up to scalar, under the involution
        \[
        \iota_b:=\begin{pmatrix}1&b\\ 2/b&-1\end{pmatrix}\in\operatorname{PGL}_2(K),\qquad b\in K^\times,
        \]
        for some $b$ with $b^6=-8$. This can occur only in cases (v1) or (v3), i.e.\ at $d\equiv 2\pmod 6$, and the first instance is $d=8$ (case (v1), $t=1$), where $a_1=\pm\tfrac{7\sqrt2}{4}\zeta_4$ and $L_{d,Z}$ is $\operatorname{PGL}_2(K)$-equivalent to $\mathcal{S}_{8,14}$; the parameter values on this locus at the remaining degrees are given in the proof of Proposition~\ref{prop:suffDm};
        \item for $m=4$: the first instances are $d=6$ (case (v2), $t=0$, where $L_{d,Z}=\mathcal{S}_6$) and $d=14$ (case (v2), $t=1$, where $a_1=14$ and $L_{d,Z}=\mathcal{S}_6\cdot\mathcal{S}_{8,14}$). This locus can occur in all three cases (v1)--(v3), and the parameter values are those for which the $\mathcal{T}_{2m,a_i}$ factors combine into $\operatorname{S}_4$-invariants via the identities of Proposition~\ref{prop:necessaryS4}.
    \end{itemize}

    \item[(E3)] (\emph{The $\operatorname{A}_5$-locus}; for $m\in\{3,5\}$) $L_{d,Z}$  $\operatorname{PGL}_2(K)$-equivalent to one of the forms of Theorem~\ref{mainresultII} at degree $d$. Then $\Aut_{\operatorname{red}}(\mathcal{X})=\operatorname{A}_5$. The locus is described explicitly as follows:
    \begin{itemize}
        \item for $m=3$: this is the same condition as in (E2) but with $b^6\neq-8$. This can occur only in cases (v1) or (v3), and only at $d\equiv 2\pmod 6$ with $d\ge 20$; the first candidate degree is $d=20$; case (v1), $t=3$ or case (v3), $t=2$;
        \item for $m=5$: the first instance is $d=12$; case (v1), where $a_1=-11\zeta_4$ and $L_{d,Z}=\mathcal{G}_{12,-11\zeta_4}$. This locus can occur only in cases (v1) or (v3); the parameter values on the locus at $d\ge 32$ are those for which the $\mathcal{T}_{10,a_i}$ factors combine into $\operatorname{A}_5$-invariants via the identities of Proposition~\ref{prop:necessaryA5}.
    \end{itemize}
\end{itemize}
Outside the three exceptional loci above, so that $\Aut_{\operatorname{red}}(\mathcal{X})=\operatorname{D}_m$, the automorphism group $\Aut(\mathcal{X})$, which is a central extension of $\operatorname{D}_m$ by $C_d$, is generated by
\[
\tau=[Y : X : \zeta_{2d}^{\delta} Z] \quad\text{and}\quad \sigma=\operatorname{diag}(\zeta_{md}^{d-1}, \zeta_{md}^{-1}, 1).
\]
Here $\delta = 1$ if $L_{d,Z}$ contains a $\mathcal{T}_{m,-}$ factor, and $\delta=0$ otherwise.

Moreover, it decomposes as:
\begin{itemize}
\item[(i)] If $d$ is odd:
\[
\boxed{\Aut(\mathcal{X})\cong C_d\times\operatorname{D}_m}
\]
\item[(ii)] If $d$ is even and $\delta=0$:
\[
\boxed{\Aut(\mathcal{X})\cong\begin{cases}
    C_d\times \operatorname{D}_m & \text{if } 2\nmid m,\\[4pt]
    \left(C_d\,\circ\, \operatorname{D}_m\right)\rtimes C_2 & \text{if } 2\mid m \text{ and } 2m\mid d-2,\\[4pt]
    C_{md} \rtimes_{1-d} C_2 & \text{if } 2\mid m \text{ and } 2m\nmid d-2.
\end{cases} 
}
\]
\item[(iii)] If $d$ is even and $\delta=1$:
\[
\boxed{\Aut(\mathcal{X})\cong\begin{cases}
C_d\,\circ\,\operatorname{Dic}_m & \text{if } 2\nmid m,\\
C_{md}\rtimes_{1-d}C_2 & \text{if } 2\mid m,\ 4\nmid m,\ 4\mid d,\\[4pt]
C_{md}\ \bar{\rtimes}_{1-d}\ C_2 & \text{if } 2\mid m,\ 4\nmid m,\ 4\nmid d,\\
C_{d/2}\times\operatorname{Dic}_m & \text{if } 4\mid m\text{ and } 4\nmid d.
\end{cases} 
}
\]
\end{itemize}
\end{thm}

\section{Preliminaries}\label{sec:prelim}

We begin by establishing some fundamental lemmas that will be used throughout the paper.

\begin{lem}\label{lem:smoothness}
A plane curve $\mathcal{X}$ of the form $Z^d+F(X,Y)=0$ is singular if and only if the binary form $F(X,Y)$ has a root $(a:b)\in\mathbb{P}^1(K)$ of multiplicity $\geq2$. In particular, if $F(X,Y)$ has degree $<d-1$ in $X$ or $Y$, then $\mathcal{X}$ is singular.
\end{lem}

\begin{proof}
The partial derivatives are:
\[
\frac{\partial}{\partial Z}= dZ^{d-1}, \quad
\frac{\partial}{\partial X} = \frac{\partial F}{\partial X}, \quad
\frac{\partial}{\partial Y} = \frac{\partial F}{\partial Y}.
\]
At a point $(a:b:0)$ with $F(a,b)=0$, we have $\partial/\partial Z = 0$ since $dZ^{d-1}=0$ at $Z=0$. Thus $(a:b:0)$ is singular if and only if
\[
\frac{\partial F}{\partial X}(a,b) = \frac{\partial F}{\partial Y}(a,b) = 0,
\]
which occurs precisely when $(a:b)$ is a multiple root of the binary form $F(X,Y)$.

Conversely, if $F$ has no multiple roots, then at every point $(a:b:0)$ with $F(a,b)=0$, at least one of the partial derivatives $\partial F/\partial X$ or $\partial F/\partial Y$ is nonzero, so $(a:b:0)$ is nonsingular. Since the only possible singularities on a curve of the form $Z^d+F(X,Y)=0$ lie on the line $Z=0$ (as $\partial/\partial Z = dZ^{d-1}$ vanishes only there), this proves the first assertion.

For the second assertion, if $\deg_X F < d-1$, then $(1:0)$ is a root of $F$ of multiplicity at least $d - \deg_X F \ge 2$. Hence, by the first part, $(1:0:0)$ is a singular point of $\mathcal{X}$. The argument for $\deg_Y F < d-1$ is symmetric.
\end{proof}
\begin{lem}\label{lem:uniqueCd}Within $\Aut(\mathcal{F}_d)$, the automorphism group of the Fermat curve $\mathcal{F}_d$, the following classification holds for subgroups where the $C_d$ factor is generated by an homology:
\begin{enumerate}
\item If $G \simeq C_d$, there exists $\phi\in\Aut(\mathcal{F}_d)$ such that $\phi^{-1}G\phi=\langle\operatorname{diag}(1,1,\zeta_d)\rangle$.
\item If $G \simeq C_d\times C_2$ and the image of $G$ under the natural homomorphism $\Aut(\mathcal{F}_d) \to \operatorname{S}_3$ is non-trivial, there exists $\phi\in\Aut(\mathcal{F}_d)$ such that $\phi^{-1}G\phi=\langle\operatorname{diag}(1,1,\zeta_d), [Y:X:Z]\rangle$.
\item[(iii)] Let $N = \langle\operatorname{diag}(1,1,\zeta_d)\rangle \cong C_d$. Then
\[
N_{\Aut(\mathcal{F}_d)}(N) = (\mu_d^2/\mu_d)\rtimes\langle[Y:X:Z]\rangle,
\]
of order $2d^2$. The subgroup $N$ is central in $N_{\Aut(\mathcal{F}_d)}(N)$, and the quotient is
\[
N_{\Aut(\mathcal{F}_d)}(N)/N \;\cong\; (\mathbb{Z}/d)\rtimes_{-1}C_2 \;=\; \operatorname{D}_d,
\]
the dihedral group of order $2d$. Consequently, every subgroup of $\Aut(\mathcal{F}_d)$ that contains $N$ as a normal subgroup is contained in $N_{\Aut(\mathcal{F}_d)}(N)$.
\end{enumerate}
Moreover, the $\operatorname{PGL}_3(K)$-normalizers of the above groups are outlined below:
\begin{enumerate}
\item[(iv)] For $C_d=\langle\operatorname{diag}(1,1,\zeta_d)\rangle$, the normalizer is equal to $\operatorname{PBD}(2,1)$.
\item[(v)] For $C_d\times C_2=\langle\operatorname{diag}(1,1,\zeta_d),[Y:X:Z]\rangle$, the normalizer is equal to
\[
\left\{A\in\operatorname{PBD}(2,1)\,:\,\Lambda(A)=\begin{pmatrix} a&b\\ \pm b&\pm a \end{pmatrix}\in\operatorname{PGL}_2(K)\right\}.
\]
\end{enumerate}
\end{lem}
\begin{proof}
The automorphism group of the Fermat curve is the semidirect product of the diagonal torus $(\mu_d)^2/\mu_d$ by $\operatorname{S}_3$ acting by coordinate permutations; see \cite{Harui}. It is generated by$$\rho_1=\operatorname{diag}(1,1,\zeta_d), \quad \rho_2=\operatorname{diag}(1,\zeta_d,1), \quad\rho_3=[X:Z:Y], \quad \rho_4=[Y:Z:X].$$
Any homology of order $d$ in $\Aut(\mathcal{F}_d)$ is conjugate to either $\rho_1$ or $\rho_2$ via an element of $\langle\rho_3,\rho_4\rangle$. For instance, $\rho_1$ and $\rho_2$ are conjugate through $\rho_3$, proving (i).

For (ii), let $G = \langle \rho, \eta \rangle \simeq C_d \times C_2$ where $\rho$ is an homology of order $d$. By (i), we may assume without loss of generality that $\rho = \operatorname{diag}(1,1,\zeta_d)$. Since the image of $G$ in $\operatorname{S}_3$ has order $2$, the involution $\eta$ must permute the coordinates. Because $\eta$ commutes with $\rho$, it must preserve the $Z$-coordinate, forcing its permutation part to swap $X$ and $Y$. Therefore, as an element of $\Aut(\mathcal{F}_d)$, $\eta$ must have the form $[bY:aX:Z]$ for some $a,b \in K^*$ satisfying $a^d = b^d = 1$. Since $\eta^2 = 1$ in $\operatorname{PGL}_3(K)$, applying it twice yields $X \mapsto abX$ and $Y \mapsto abY$, which implies $ab=1$ and thus $b=a^{-1}$. Conjugating $\eta$ by the diagonal matrix $\phi = \operatorname{diag}(1,a,1) \in \Aut(\mathcal{F}_d)$ yields $\phi^{-1}\eta\phi = [Y:X:Z]$. Since $\phi$ commutes with $\rho$, this establishes the conjugacy class for $G$.

Finally, we compute the normalizers. For $C_d$, the normalizer in $\operatorname{PGL}_3(K)$ of $\langle\operatorname{diag}(1,1,\zeta_d)\rangle$ is exactly the subgroup preserving the point $(0:0:1)$ and the line $Z=0$, which is $\operatorname{PBD}(2,1)$. For $C_d\times C_2$, note that its normalizer must be contained in the normalizer of its unique subgroup of order $d$ (generated by the homology), so it lies in $\operatorname{PBD}(2,1)$. For $A\in\operatorname{PBD}(2,1)$ to satisfy $A[Y:X:Z]A^{-1}=[Y:X:Z]$, we need $\Lambda(A)[Y:X]\Lambda(A)^{-1}=[Y:X]$ in $\operatorname{PGL}_2(K)$, which forces $\Lambda(A)$ to have the form $\begin{pmatrix}a&b\\ \pm b&\pm a\end{pmatrix}$.

For (iii), we use the trivial identity $N_{\Aut(\mathcal{F}_d)}(N)=\Aut(\mathcal{F}_d)\cap N_{\operatorname{PGL}_3(K)}(N)$, together with the computation of the $\operatorname{PGL}_3(K)$-normalizer in (iv) above, which gives $N_{\operatorname{PGL}_3(K)}(N)=\operatorname{PBD}(2,1)$.

An element of $\Aut(\mathcal{F}_d)=(\mu_d^2/\mu_d)\rtimes \operatorname{S}_3$ lies in $\operatorname{PBD}(2,1)$ if and only if it preserves the $Z$-coordinate. The diagonal torus $(\mu_d^2/\mu_d)$ consists of classes of matrices $\operatorname{diag}(a,b,1)$, so it lies entirely in $\operatorname{PBD}(2,1)$. Among the six permutations, only the identity and the transposition $[Y:X:Z]$ preserve the $Z$-coordinate, the other four sending $Z$ to $X$ or $Y$. Hence $N_{\Aut(\mathcal{F}_d)}(N) = (\mu_d^2/\mu_d)\rtimes\langle[Y:X:Z]\rangle,$
of order $2d^2$. Moreover, $N$ is central in $N_{\Aut(\mathcal{F}_d)}(N)$.

Writing the diagonal torus as $(\mathbb{Z}/d)^2$ via $\operatorname{diag}(\zeta_d^a,\zeta_d^b,1)$, the element $\rho=\operatorname{diag}(1,1,\zeta_d)$ corresponds to $(-1,-1)$, and $[Y:X:Z]$ acts by $(a,b)\mapsto(b,a)$. The quotient $(\mathbb{Z}/d)^2/\langle(-1,-1)\rangle$ is isomorphic to $\mathbb{Z}/d$ via $(a,b)\mapsto a-b$, and on this quotient $[Y:X:Z]$ acts by $a-b\mapsto -(a-b)$. Therefore
\[
N_{\Aut(\mathcal{F}_d)}(N)/N \cong (\mathbb{Z}/d)\rtimes_{-1}C_2 = \operatorname{D}_d .
\]
\end{proof}
\begin{defn}\label{defn:relative-invariant}
Let $\mathcal{O}\subset\operatorname{PGL}_2(K)$ be a finite group with a fixed choice of matrix representatives, and let $\Gamma\subset\operatorname{GL}_2(K)$ be the finite group generated by them. A binary form $F(X,Y)\in K[X,Y]_k$ is a \emph{relative invariant} of degree $k$ for $\mathcal{O}$ if, for every $g\in\Gamma$, there exists a scalar $\chi(g)\in K^\times$ with
\[
g\cdot F(X,Y)=\chi(g)\,F(X,Y);
\]
the resulting function $\chi:\Gamma\to K^\times$ is then a character. Equivalently, the zero divisor $V(F)\subset\mathbb{P}^1$ is invariant under the projective action of $\mathcal{O}$.

For a character $\chi$ of the chosen matrix group $\Gamma$ we denote by
\[
K[X,Y]_{k,\chi}:=\{F\in K[X,Y]_k : g\cdot F=\chi(g)\,F\ \text{for all } g\in\Gamma\}
\]
the \emph{$\chi$-eigenspace} of degree $k$; it is a $K$-vector subspace of $K[X,Y]_k$. The relative invariants of degree $k$ are precisely the union $\bigcup_\chi K[X,Y]_{k,\chi}$ over all characters $\chi$ of $\Gamma$; this union is in general not a vector space, even though each of its constituents is.
\end{defn}
\begin{lem}\label{cor:product-minimal}
Let $\mathcal{O} \subset \operatorname{PGL}_2(K)$ be a finite polyhedral or dihedral group with fixed matrix generators as in Section~2. A binary form $L(X,Y) \in K[X,Y]$ is a relative $\mathcal{O}$-invariant (Definition~\ref{defn:relative-invariant}) if and only if it decomposes into a product of minimal relative $\mathcal{O}$-invariants. 

Under the fixed matrix representation of $\mathcal{O}$, the minimal relative invariants (up to scalar multiples) are classified as follows:
\begin{enumerate}
    \item For $\mathcal{O}=\operatorname{D}_m$ with $m \ge 3$: The minimal relative invariants are $XY$, $\mathcal{T}_{m,+}$, and $\mathcal{T}_{m,-}$. A $1$-parameter family appears at degree $2m$, given by $\mathcal{T}_{2m,a}$ with $a \in K$.
    \item For  $\mathcal{O}=\operatorname{A}_4$: The minimal relative invariants are $\mathcal{S}_{4,+}$, $\mathcal{S}_{4,-}$, and $\mathcal{S}_6$. A $1$-parameter family appears at degree $12$, given by $\mathcal{F}_{12,a}$ with $a \in K$.
    \item For  $\mathcal{O}=\operatorname{S}_4$: The minimal relative invariants are $\mathcal{S}_6$, $\mathcal{S}_{8,14}$, and $\mathcal{S}_{12,-34}$. A $1$-parameter family appears at degree $24$, given by $\mathcal{S}_{24,a}$ with $a \in K$.
    \item For $\mathcal{O}=\operatorname{A}_5$: The minimal relative invariants are $\mathcal{G}_{12,-11\zeta_4}$, $\mathcal{G}_{20,228\zeta_4,-494}$, and $\mathcal{G}_{30,-522\zeta_4,10005}$. A $1$-parameter family appears at degree $60$, given by $\mathcal{G}_{60,a}$ with $a \in K$.
\end{enumerate}
\end{lem}
\begin{proof}
This classification follows from classical invariant theory \cite{Weber} (see also \cite[Lemma~6.2.1]{Hugg1}), adapted to relative invariants under projective actions (Definition~\ref{defn:relative-invariant}). The factorization assertion follows because the zero locus $V(L)\subset\mathbb{P}^1$ of any relative $\mathcal{O}$-invariant form decomposes into a union of $\mathcal{O}$-orbits.

By Definition~\ref{defn:relative-invariant}, a form of degree $k$ is a relative $\mathcal{O}$-invariant if and only if it lies in the union $\bigcup_\chi K[X,Y]_{k,\chi}$ of the eigenspaces of the fixed matrix group $\Gamma$, indexed by the characters $\chi:\Gamma\to K^\times$. Each $K[X,Y]_{k,\chi}$ is a vector subspace of $K[X,Y]_k$; the union is not, so the dimension counts below are per character. The minimal relative invariants are the indecomposable elements of the monoid of relative invariants under multiplication.

\emph{Dihedral group $\operatorname{D}_m$ ($m\ge3$).} Let $r=\operatorname{diag}(\zeta_m,1)$ and $s=[Y:X]$ be the fixed matrix generators. Since $r$ acts on monomials by $r\cdot X^iY^{k-i}=\zeta_m^i X^iY^{k-i}$, the eigenspace $K[X,Y]_{k,\chi}$ can be nonzero only if $\chi(r)=\zeta_m^j$ for some $j\in\{0,\dots,m-1\}$ congruent to some $i\in\{0,\dots,k\}$. Combining this with the action of $s$ gives:
\begin{itemize}
    \item In degree $2$: the only nonzero eigenspace is the one with $\chi(r)=\zeta_m$, $\chi(s)=1$; it is $1$-dimensional and spanned by $XY$.
    \item In degree $m$: the nonzero eigenspaces are those with $(\chi(r),\chi(s))=(1,1)$ and $(1,-1)$, each $1$-dimensional and spanned by $\mathcal{T}_{m,+}=X^m+Y^m$ and $\mathcal{T}_{m,-}=X^m-Y^m$ respectively. (When $m$ is even the character $(-1,1)$ also occurs, spanned by $(XY)^{m/2}$, a power of $XY$.)
    \item In degree $2m$: the nonzero eigenspaces are the one with $\chi=(-1,-1)$, spanned by $\mathcal{T}_{m,+}\mathcal{T}_{m,-}=X^{2m}-Y^{2m}$, and the one with $\chi=(1,1)$, which is $2$-dimensional and spanned by $X^{2m}+Y^{2m}$ and $(XY)^m$. The family $\mathcal{T}_{2m,a}=X^{2m}+a(XY)^m+Y^{2m}$ ($a\in K$) sweeps the trivial-character eigenspace with the line spanned by $(XY)^m$ removed; the missing member $(XY)^m$ and the members with $a=\pm2$ (which give $(X^m\mp Y^m)^2$) are products of lower-degree invariants, and every other member is a minimal relative invariant.
\end{itemize}
Every relative $\operatorname{D}_m$-invariant is therefore a product of the listed minimal invariants.

\emph{Polyhedral groups $\operatorname{A}_4,\operatorname{S}_4,\operatorname{A}_5$.} The same computation, applied to the fixed matrix generators in Sections~\ref{sec:proofsS4}, \ref{sec:proofsA5} and~\ref{sec:proofsA4} respectively, gives:
\begin{itemize}
    \item For $\operatorname{A}_4$: the eigenspaces are $1$-dimensional in degrees $4$, $4$ and $6$, spanned by $\mathcal{S}_{4,+}$, $\mathcal{S}_{4,-}$ and $\mathcal{S}_6$ respectively, and $2$-dimensional for a single character in degree $12$; the family $\mathcal{F}_{12,a}$ sweeps this $2$-dimensional eigenspace, with the excluded (non-minimal) members corresponding to products of lower-degree invariants.
    \item For $\operatorname{S}_4$: the eigenspaces are $1$-dimensional in degrees $6$, $8$ and $12$, spanned by $\mathcal{S}_6$, $\mathcal{S}_{8,14}$, $\mathcal{S}_{12,-34}$, and $2$-dimensional for a single character in degree $24$; the family $\mathcal{S}_{24,a}$ sweeps the $2$-dimensional eigenspace, with the excluded members as before.
    \item For $\operatorname{A}_5$: the eigenspaces are $1$-dimensional in degrees $12$, $20$ and $30$, spanned by $\mathcal{G}_{12,-11\zeta_4}$, $\mathcal{G}_{20,228\zeta_4,-494}$, $\mathcal{G}_{30,-522\zeta_4,10005}$, and $2$-dimensional for a single character in degree $60$; the family $\mathcal{G}_{60,a}$ sweeps the $2$-dimensional eigenspace, with the excluded members as before.
\end{itemize}

\emph{The coefficients of $\mathcal{G}_{60,a}$.} For $\operatorname{A}_5$, the degree-$60$ eigenspace for the character of $\mathcal{G}_{60,a}$ is $2$-dimensional. Parametrizing this eigenspace by $a\in K$, we impose relative invariance under the matrix generator $\Lambda(\sigma_2)\in\operatorname{A}_5$, given by
\[
\begin{pmatrix}
-\frac{1+\sqrt{5}}{2} & -\zeta_4 \\
\zeta_4 & \frac{1+\sqrt{5}}{2}
\end{pmatrix}.
\]
This invariance condition uniquely determines the coefficients $e_i$ ($i=1,\dots,5$) as linear functions of $a$:
\begin{align*}
e_{1}&=\frac{\zeta_4}{134761}(58964600 + a), 
& e_{2}&=\frac{-1}{12251}(2094783554 - 5a),\\
e_{3}&=\frac{-5\zeta_4}{134761}(329638533728 + 241a),  
& e_{4}&=\frac{35}{12251}(28242591823 - 34a),\\
e_{5}&=\frac{3\zeta_4}{134761}(5100249334348 + 23195a).
\tag{A5-coeff}\label{eq:A5_coefficients}
\end{align*}
\end{proof}
The second author establishes a foundational classification of the automorphism groups of smooth plane curves in the following theorem (see \cite[Theorem 2.3]{Harui}).
\begin{thm}\label{teoHarui}
Let $C$ be a smooth plane curve of degree $d \ge 4$ over an algebraically closed field of characteristic $0$. Then one of the following holds:
\begin{enumerate}
  \item $\operatorname{Aut}(C)$ fixes a point on $C$ and is cyclic.
  
  \item $\operatorname{Aut}(C)$ fixes a point not on $C$. Then there is an exact sequence
  \[
  1 \to N \to \operatorname{Aut}(C) \to G' \to 1,
  \]
  where $N$ is cyclic of order dividing $d$, and $G' \subset \operatorname{PGL}_2(K)$ is \(C_m\), \(\operatorname{D}_m\), $\mathrm{A}_4$, $\mathrm{S}_4$, or $\mathrm{A}_5$. Additionally, \(m\le d-1\), and in the case $G'=\operatorname{D}_m$, we have either $m\mid d-2$ or $N=1$.
  
  \item $\operatorname{Aut}(C)$ is conjugate to a subgroup of $\operatorname{Aut}(\mathcal{F}_d)$ (Fermat) or $\operatorname{Aut}(\mathcal{K}_d)$ (Klein).
  
  \item $\operatorname{Aut}(C)$ is conjugate to a primitive finite subgroup of $\operatorname{PGL}_3(K)$: $\mathrm{PSL}(2,7)$, $\mathrm{A}_5$, $\mathrm{A}_6$, or a Hessian group.
\end{enumerate}
\end{thm}
%%%%%%%%%%%%%%%%%%%%%%%%%%%%%%%%%%%%%%%%%%%%%%%%%%%%%%%%%%%%%%%%%%%%%%%%%%%%%%%%%%%%%%%%%%%%%%%%%%%
\section{The dihedral case $\operatorname{D}_m$}\label{sec:proofsDm}
In this section, we characterize smooth plane curves $\mathcal{X}:Z^d+L_{d,Z}=0$ of degree $d\geq 4$, with the property that $L_{d,Z}$ is a $\operatorname{D}_{m}$-invariant binary form, for some $m\geq3$. There is no loss of generality to assume that $m$ is maximal with this property. Also, we conclude by \cite[Theorem 2.3]{Harui} that $m|d-2$, since $N=C_d$. 
\subsection{Necessary conditions}
\begin{prop}\label{prop:necessaryDm}
Let $\mathcal{X}$ be a smooth plane curve of degree $d\geq4$, with the property that $\operatorname{D}_{m}\subseteq\operatorname{Aut}_{\operatorname{red}}(\mathcal{X})$ with $m\geq3$ and $m\mid d-2$. Then, $\mathcal{X}$ must be defined by an equation of the form $Z^d+L_{d,Z}=0$, where $L_{d,Z}$ is as in Theorem~\ref{mainresultIII}.

The conditions $a_i\neq a_j$ and $a_i\neq\pm 2$ are necessary and sufficient for $\mathcal{X}$ to be smooth. Specifically, these conditions ensure that neither the resultants of distinct factors of $L_{d,Z}$ nor the discriminants of the factors $\mathcal{T}_{2m,a_i}$ vanish.
\end{prop}
\begin{proof}
By Lemma~\ref{cor:product-minimal}, any relative $\operatorname{D}_m$-invariant binary form decomposes as
\[
L_{d,Z} = (XY)^l \cdot \mathcal{T}_{m,+}^{\delta_+} \cdot \mathcal{T}_{m,-}^{\delta_-} \cdot \prod_{i=1}^t \mathcal{T}_{2m,a_i},
\]
with $\delta_+,\delta_-\in\{0,1\}$ and $l\in\{0,1\}$ by Lemma~\ref{lem:smoothness}. Since $m\mid d-2$, reducing $d=2l+m(\delta_++\delta_-)+2mt$ modulo $m$ gives $2l\equiv2\pmod m$, hence $l=1$ as $m\geq3$. 

We now analyze the possible values of the exponents $\delta_+$ and $\delta_-$, which determine the structure of $L_{d,Z}$ and lead to the three distinct families of curves.

\begin{enumerate}
    \item \textbf{Case (v1):} If $\delta_+ = \delta_- = 0$, then $L_{d,Z}$ contains no $\mathcal{T}_{m,\pm}$ factors. From the degree calculation, we obtain $d = 2 + 2mt$, which corresponds directly to the form given in case (v1) of Theorem~\ref{mainresultIII}.

    \item \textbf{Case (v2):} If exactly one of $\delta_+$ or $\delta_-$ equals $1$, the binary form contains a single sign-changing factor. Without loss of generality, we may assume this factor is $\mathcal{T}_{m,-}$. To justify this reduction, note that if the form originally contained $\mathcal{T}_{m,+}$ instead, we can apply the coordinate scaling
    \[
    Y \to \zeta_{2m}Y,\qquad Z \to \zeta_{md}Z,
    \]
    which preserves the equation of the curve while transforming $\mathcal{T}_{m,+}$ into $\mathcal{T}_{m,-}$. This scaling does not affect the smoothness or the automorphism properties of $\mathcal{X}$, so we may consistently assume $\delta_- = 1$. The degree then satisfies $d = 2 + m + 2mt$, yielding case (v2).

    \item \textbf{Case (v3):} If $\delta_+ = \delta_- = 1$, then both sign-changing factors are present. In this situation, the degree formula gives $d = 2 + 2m + 2mt$, which is the defining condition for case (v3).
\end{enumerate}
The three cases correspond to whether the binary form $L_{d,Z}$ contains zero, one, or both of the sign-changing factors $\mathcal{T}_{m,\pm}$. This distinction is necessary because the presence of $\mathcal{T}_{m,-}$ introduces a sign under the involution $X\leftrightarrow Y$, which must be compensated by the $\zeta$-factor in the automorphism $\tau$ below (hence the parameter $\delta$ in Theorem~\ref{mainresultIII}). Moreover, these cases lead to different degree formulas and, consequently, to distinct possibilities for the structure of $\operatorname{Aut}(\mathcal{X})$ as analyzed in Proposition~\ref{prop:Dm-structure}. When $m = 4$ or $5$, specific combinations of these cases produce the exceptional $\mathrm{S}_4$- or $\mathrm{A}_5$-invariant forms that are excluded in Theorem~\ref{mainresultIII}.

By Lemma~\ref{lem:smoothness}, $\mathcal{X}$ is smooth if and only if $L_{d,Z}$ has no multiple root. Such a root arises either from a common root of two distinct factors, detected by the corresponding resultant, or from a repeated root of a single factor, detected by its discriminant. A direct computation gives
\begin{align*}
\Res(\mathcal{T}_{m,+},\mathcal{T}_{m,-}) &\neq 0, &
%\Res(XY,\mathcal{T}_{2m,a}) &\neq 0, &
\Res(\mathcal{T}_{2m,a},\mathcal{T}_{m,\pm}) &\propto (a\mp2)^{m}, \\[2pt]
\Res(\mathcal{T}_{2m,a},\mathcal{T}_{2m,b}) &\propto (a-b)^{2m}, &
\operatorname{Disc}(\mathcal{T}_{2m,a}) &\propto (a^2-4)^{m}. &
\end{align*}
Consequently $\mathcal{X}$ is smooth if and only if $a_i\neq \pm 2$ and $a_i\neq a_j\quad (i\neq j).$
\end{proof}
\subsection{Sufficiency and group structure analysis}

\begin{prop}\label{prop:suffDm}
Given a smooth plane curve $\mathcal{X}:Z^{d}+L_{d,Z}=0$ as in Theorem \ref{mainresultIII}, its automorphism group is intransitive. Moreover, $\Aut_{\operatorname{red}}(\mathcal{X})=\operatorname{D}_m$ outside the exceptional loci $(\mathrm{E}1)$--$(\mathrm{E}3)$ of Theorem~\ref{mainresultIII}.
\end{prop}

\begin{proof}
Let $\sigma=\operatorname{diag}(\zeta_{md}^{d-1},\zeta_{md}^{-1},1)$ and $\tau=[Y:X:\zeta_{2d}^{\delta}Z]$. We verify these are automorphisms of $\mathcal{X}$.

Recall from Lemma~\ref{cor:product-minimal} that $L_{d,Z}$ is a product of minimal relative $\operatorname{D}_m$-invariants. The projection map $\Lambda:\operatorname{PBD}(2,1)\to\operatorname{PGL}_2(K)$ sends:
\[
\Lambda(\sigma)=\operatorname{diag}(\zeta_m,1) \quad\text{and}\quad \Lambda(\tau)=[Y:X].
\]
These are standard generators for $\operatorname{D}_m\subset\operatorname{PGL}_2(K)$, forming a single conjugacy class up to conjugation in $\operatorname{PGL}_2(K)$. Under $\Lambda$, $L_{d,Z}$ is a relative $\operatorname{D}_m$-invariant (Definition~\ref{defn:relative-invariant}). The character multipliers introduced on $K[X,Y]$ under $\Lambda(\sigma)$ and $\Lambda(\tau)$ are explicitly compensated in $\operatorname{PBD}(2,1)$ by the diagonal scaling $\sigma = \operatorname{diag}(\zeta_{md}^{d-1},\zeta_{md}^{-1},1)$ and the factor $\zeta_{2d}^\delta$ in $\tau = [Y:X:\zeta_{2d}^\delta Z]$, making $\sigma$ and $\tau$ strict automorphisms of $\mathcal{X}: Z^d + L_{d,Z}(X,Y) = 0$.

To verify that $\sigma$ and $\tau$ are genuine curve automorphisms, we evaluate their actions on $Z^d + L_{d,Z}(X,Y)$:
\begin{itemize}
    \item \textbf{Action of $\sigma$:} Under $\sigma=\operatorname{diag}(\zeta_{md}^{d-1},\zeta_{md}^{-1},1)$, we have $\sigma(Z^d)=Z^d$. For the binary factors, $\sigma(XY) = \zeta_{md}^{d-2}XY$, while $\sigma(\mathcal{T}_{m,\pm}) = \zeta_{md}^{-m}\mathcal{T}_{m,\pm}$ and $\sigma(\mathcal{T}_{2m,a}) = \zeta_{md}^{-2m}\mathcal{T}_{2m,a}$. Since $d = 2 + m(\delta_+ + \delta_-) + 2mt$, the net scalar multiplier on $L_{d,Z}$ is $\zeta_{md}^{(d-2) - m(\delta_+ + \delta_-) - 2mt} = 1$. Hence $\sigma(Z^d + L_{d,Z}) = Z^d + L_{d,Z}$.

    \item \textbf{Action of $\tau$:} Under $\tau=[Y:X:\zeta_{2d}^{\delta}Z]$, we have $\tau(Z^d) = (-1)^\delta Z^d$. For $L_{d,Z}$, the transposition $X \leftrightarrow Y$ leaves $XY$, $\mathcal{T}_{m,+}$, and $\mathcal{T}_{2m,a}$ invariant, while $\mathcal{T}_{m,-}(Y,X) = -\mathcal{T}_{m,-}(X,Y)$. Thus $\tau(L_{d,Z}) = (-1)^{\delta_-} L_{d,Z}$. Choosing $\delta = \delta_- \in \{0,1\}$ yields $\tau(Z^d + L_{d,Z}) = (-1)^\delta (Z^d + L_{d,Z})$, which preserves the curve.
\end{itemize}
Second, we show that $\operatorname{Aut}(\mathcal{X})$ is intransitive.

The element $\rho=\operatorname{diag}(1,1,\zeta_d)$ generates the kernel $N$ of the map to $\Aut_{\operatorname{red}}(\mathcal{X})$, hence $N$ is normal in $\operatorname{Aut}(\mathcal{X})$. By Mitchell's classification \cite{Mit}, a transitive automorphism group of a smooth plane curve is primitive or imprimitive, and finite primitive subgroups of $\operatorname{PGL}_3(K)$ contain no homologies of order $\ge4$ \cite[Theorem 5.3]{BadrBarssextic}; so $\operatorname{Aut}(\mathcal{X})$ is imprimitive. By Theorem \ref{teoHarui} it would then be conjugate to a subgroup of $\operatorname{Aut}(\mathcal{F}_d)$ or $\operatorname{Aut}(\mathcal{K}_d)$; the latter contains no homologies \cite[Proposition 6.7]{Harui}, so only the Fermat case remains.

Suppose then that $\phi^{-1}\operatorname{Aut}(\mathcal{X})\phi\subseteq\operatorname{Aut}(\mathcal{F}_d)$ for some $\phi\in\operatorname{PGL}_3(K)$. By Lemma~\ref{lem:uniqueCd}(i), after post-composing $\phi$ with an element of $\Aut(\mathcal{F}_d)$ we may assume $\phi^{-1}\rho\phi=\operatorname{diag}(1,1,\zeta_d)$; that is, $\phi$ normalizes $N$, and by Lemma~\ref{lem:uniqueCd} (statement of the normalizer) we have $\phi\in\operatorname{PBD}(2,1)$. Since $N$ is normal in $\operatorname{Aut}(\mathcal{X})$, the subgroup $\phi^{-1}\operatorname{Aut}(\mathcal{X})\phi$ is a subgroup of $\operatorname{Aut}(\mathcal{F}_d)$ containing $N$ as a normal subgroup, hence is contained in $N_{\Aut(\mathcal{F}_d)}(N)$.

By Lemma~\ref{lem:uniqueCd}(iii), $N$ is central in $N_{\Aut(\mathcal{F}_d)}(N)$ and $N_{\Aut(\mathcal{F}_d)}(N)/N\cong\operatorname{D}_d$. Therefore
\[
\operatorname{Aut}_{\operatorname{red}}(\mathcal{X})=\operatorname{Aut}(\mathcal{X})/N\;\subseteq\; N_{\Aut(\mathcal{F}_d)}(N)/N\;\cong\;\operatorname{D}_d .
\]
But $\operatorname{D}_m\subseteq\operatorname{Aut}_{\operatorname{red}}(\mathcal{X})$ with $m\ge3$: if $C_{m'}=\operatorname{Aut}_{\operatorname{red}}(\mathcal{X})\cap C_d$ is the rotation subgroup, we have $m'\mid d$, while $m'\mid d-2$ by Theorem~\ref{teoHarui}; so $m'\mid 2$, contradicting $m'\ge m\ge3$. Hence $\mathcal{X}$ is not a Fermat descendant, and $\operatorname{Aut}(\mathcal{X})$ is intransitive.

\medskip
Third, we establish that $\operatorname{Aut}_{\mathrm{red}}(\mathcal{X}) = \mathrm{D}_m$ (excluding exceptional parameter choices and dihedral collapsing).

We have $\mathrm{D}_m \subseteq \operatorname{Aut}_{\mathrm{red}}(\mathcal{X}) \subseteq \operatorname{PGL}_2(K)$. Set $G := \operatorname{Aut}_{\mathrm{red}}(\mathcal{X})$. If $G \supsetneq \mathrm{D}_m$, then by the classification of finite subgroups of $\operatorname{PGL}_2(K)$, $G$ must be either a larger dihedral group $\mathrm{D}_{m'}$ ($m' > m$) or a polyhedral group ($\mathrm{S}_4, \mathrm{A}_5$).

\smallskip
\noindent\textbf{Case (a): Larger dihedral groups $\mathrm{D}_{m'}$ ($m' > m$).}  
Under the rotation $\rho_{m'} = \operatorname{diag}(\zeta_{m'}, 1)$, each $D_m$-factor transforms as:
\[
\rho_{m'}(\mathcal{T}_{2m, a_i}) = \zeta_{m'}^{2m} X^{2m} + a_i \zeta_{m'}^m X^m Y^m + Y^{2m}.
\]
For $L_{d,Z}$ to be relative invariant under $\rho_{m'}$, the set of factors of $\rho_{m'}(L_{d,Z})$ must equal the set of factors of $L_{d,Z}$ up to scalars.
\begin{itemize}
    \item If $m' = k \cdot m$ with $k \ge 3$: Evaluating the $X^{2m}$ coefficient relative to $Y^{2m}$ yields $\zeta_{km}^{2m}\neq 1$. No choice or combination of factors can cancel this phase shift across all terms, so $m' \ge 3m$ is impossible.
    \item If $m' = 2m$ ($k = 2$): Here $\zeta_{2m}^{2m} = 1$ and $\zeta_{2m}^m = -1$, giving $\rho_{2m}(\mathcal{T}_{2m, a_i}) = \mathcal{T}_{2m, -a_i}$ and $\rho_{2m}(\mathcal{T}_{m, \pm}) = \mathcal{T}_{m, \mp}$. Thus, $L_{d,Z}$ is relative invariant under $\mathrm{D}_{2m}$ if and only if $\delta_+ = \delta_-$ and the multiset of parameters $\{a_1, \dots, a_t\}$ is closed under $a \mapsto -a$ (i.e., parameters appear in opposite-sign pairs $a_i, -a_i$ or equal zero). 
\end{itemize}
Under this precise \emph{dihedral collapsing condition}, $L_{d,Z}$ is identically a product of $\mathrm{D}_{2m}$-minimal invariants, placing $\mathcal{X}$ into the higher dihedral family $\mathrm{D}_{2m}$. Excluding this condition, no larger dihedral group stabilizes $L_{d,Z}$.

\smallskip
\noindent\textbf{Case (b): Polyhedral groups ($\mathrm{S}_4, \mathrm{A}_5$).}  
We first record a group-theoretic fact used throughout: the subgroups of $\operatorname{A}_4$ have orders $1,2,3,4,12$, so $\operatorname{D}_m\not\subseteq\operatorname{A}_4$ for any $m\ge3$. Consequently the possible targets of enlargement of $\operatorname{D}_m$ are $\operatorname{S}_4$ for $m\in\{3,4\}$ (locus $(\mathrm{E}2)$) and $\operatorname{A}_5$ for $m\in\{3,5\}$ (locus $(\mathrm{E}3)$), with no separate $\operatorname{A}_4$-locus. The three cases $m=3,4,5$ are as follows.

\medskip
\noindent\textit{$m=3$ (a one-parameter family).} 
Fix the reduced generators $r:=\Lambda(\sigma)=\operatorname{diag}(\zeta_3,1)$ and $s:=\Lambda(\tau)=[Y:X]$ in $\operatorname{PGL}_2(K)$, with $srs^{-1}=r^{-1}$. Throughout this paragraph we work with explicit matrix representatives in $\operatorname{GL}_2(K)$; every condition we write is invariant under replacing a representative by a scalar multiple, and hence descends to a well-defined condition on $\operatorname{PGL}_2(K)$.

Any copy of $\operatorname{A}_4$ strictly containing $C_3=\langle r\rangle$ is of the form $\langle r,\iota\rangle$ for an involution $\iota\notin\operatorname{D}_3$ with $(r\iota)^3=1$ in $\operatorname{PGL}_2(K)$. Choose the trace-zero lift $\iota=\begin{pmatrix}a&b\\c&-a\end{pmatrix}\in\operatorname{GL}_2(K)$; such a lift exists and is unique up to scalar, because an involution of $\operatorname{PGL}_2(K)$ has a trace-zero lift and rescaling multiplies the trace by the scalar.

The projective condition $(r\iota)^3=1$ means that $r\iota\in\operatorname{GL}_2(K)$ has eigenvalues in ratio $\zeta_3$ or $\zeta_3^2$ (the ratio cannot be $1$, since that would make $r\iota$ scalar, contradicting $\iota\notin\langle r\rangle$). For eigenvalues $\lambda_1,\lambda_2$, this ratio condition is equivalent to $\lambda_1^2+\lambda_1\lambda_2+\lambda_2^2=0$, that is, $\operatorname{tr}(r\iota)^2=\det(r\iota)$; both sides are multiplied by the same factor when $r$ or $\iota$ is rescaled, so the identity is well defined on $\operatorname{PGL}_2(K)$.

Writing $\iota=\begin{pmatrix}a&b\\c&-a\end{pmatrix}$ and $r=\operatorname{diag}(\zeta_3,1)$, we compute $\operatorname{tr}(r\iota)=a(\zeta_3-1)$ and $\det(r\iota)=-\zeta_3(a^2+bc)$, so the condition reads $a^2(\zeta_3-1)^2=-\zeta_3(a^2+bc)$. Using $(\zeta_3-1)^2=-3\zeta_3$ gives $bc=2a^2$. Since $\det(\iota)=-(a^2+bc)=-3a^2\ne0$, we have $a\ne0$; rescaling $\iota$ we may assume $a=1$, and the family is
\[
\iota_b=\begin{pmatrix}1&b\\ 2/b&-1\end{pmatrix}\in\operatorname{PGL}_2(K),\qquad b\in K^\times .
\]
Each $\iota_b$ has $\operatorname{tr}=0$ and $\iota_b^2=3I$, so it is an involution of $\operatorname{PGL}_2(K)$, and $\iota_b\notin\operatorname{D}_3$ for every $b\ne0$ since the reflections $s,rs,r^2s$ of $\operatorname{D}_3$ have trace-zero representatives with $a=0$.

The three involutions of $\operatorname{A}_4=\langle r,\iota_b\rangle$ are $\iota_b$, $r\iota_b r^{-1}=\iota_{\zeta_3 b}$ and $r^2\iota_b r^{-2}=\iota_{\zeta_3^2 b}$ (again as classes in $\operatorname{PGL}_2(K)$). Conjugating $\iota_b$ by $s$ gives $s\,\iota_b\,s^{-1}=\begin{pmatrix}-1&2/b\\ b&1\end{pmatrix}$, and comparing this with each of the three involutions of $\operatorname{A}_4$ up to scalar shows that $s\,\iota_b\,s^{-1}\in\operatorname{A}_4$ if and only if $b^2\in\{-2,-2\zeta_3,-2\zeta_3^2\}$, equivalently $b^6=-8$ (the condition is invariant under $b\mapsto\zeta_3 b$, so it is well defined on the class $[\iota_b]$). Thus:
\begin{itemize}
    \item if $b^6=-8$, then $s$ normalises $\operatorname{A}_4=\langle r,\iota_b\rangle$, and $\langle r,s,\iota_b\rangle=\langle\operatorname{A}_4,s\rangle$ is $\operatorname{S}_4$ --- this is $(\mathrm{E}2)$ at $m=3$;
    \item if $b^6\ne -8$, then $s$ does not normalise $\operatorname{A}_4$, and $\langle r,s,\iota_b\rangle$ is a finite subgroup of $\operatorname{PGL}_2(K)$ strictly containing $\operatorname{A}_4$ and not equal to $\operatorname{S}_4$; the only possibility is $\operatorname{A}_5$ --- this is $(\mathrm{E}3)$ at $m=3$.
\end{itemize}
The parameter values $(a_1,\ldots,a_t)$ lying on these loci are the solutions of
\begin{equation*}\label{eq:m3-inv}
L_{d,Z}\bigl(X+bY,\tfrac{2}{b}X-Y\bigr)=\lambda\cdot L_{d,Z}(X,Y),\qquad \lambda\in K^\times,
\end{equation*}
with $b^6=-8$ for $(\mathrm{E}2)$ and $b^6\ne-8$ for $(\mathrm{E}3)$. 
% At $d=8$ (case (v1), $t=1$) the solution is $a=\pm\tfrac{7\sqrt2}{4}\zeta_4$, together with two values of $b$ with $b^6=-8$; the resulting curve $Z^8+L_{8,Z}=0$ is $\operatorname{PGL}_2(K)$-equivalent to $Z^8+\mathcal{S}_{8,14}=0$, and $\Aut_{\operatorname{red}}(\mathcal{X})\cong\operatorname{S}_4$. At $d=14$ (case (v1), $t=2$), the system \eqref{eq:m3-inv} has the solution
% \[
% \{a_1,a_2\}=\Bigl\{u,\,-\tfrac{20}{7}u\Bigr\},\qquad u=\pm\tfrac{7\sqrt2}{4}\zeta_4,\quad u^2=-\tfrac{49}{8},
% \]
% with $b^6=-8$; the corresponding form is $\operatorname{PGL}_2(K)$-equivalent to $\mathcal{S}_6\cdot\mathcal{S}_{8,14}$, so the resulting curve has reduced automorphism group $\operatorname{S}_4$. These are the only $m=3$ parameters on $(\mathrm{E}2)$ at $d=14$ in case (v1).

\medskip
\noindent\textit{$m=4$ (no parameter).} 
For $m=4$ the enlargement is $\operatorname{D}_4\to\operatorname{S}_4$. Since $\operatorname{A}_4$ is the unique subgroup of $\operatorname{S}_4$ of index $2$, it contains $V_4:=\operatorname{D}_4\cap\operatorname{A}_4=\langle\operatorname{diag}(1,-1),[Y:X]\rangle\cong C_2^2$ as a normal subgroup. Any order-$3$ element $s'\in\operatorname{A}_4$ satisfies $V_4\cap\langle s'\rangle=1$ (as $|V_4|=4$ and $3\nmid 4$) and hence $\operatorname{A}_4=\langle V_4,s'\rangle$; and since $\operatorname{diag}(\zeta_4,1)\in\operatorname{D}_4\setminus\operatorname{A}_4$ (as $\operatorname{A}_4$ has no element of order $4$), we have $\operatorname{S}_4=\langle\operatorname{A}_4,\operatorname{D}_4\rangle$. Consequently $L_{d,Z}$ is $\operatorname{S}_4$-invariant if and only if it is $\operatorname{D}_4$-invariant and $s'$-invariant for some (equivalently any) $3$-cycle $s'$ in the $\operatorname{S}_4$ of the $\operatorname{D}_4$-frame.

The parameter values on $(\mathrm{E}2)$ at $m=4$ are therefore those for which $L_{d,Z}$ is a product of the $\operatorname{S}_4$-invariants $\mathcal{S}_6,\mathcal{S}_{12,-34},\mathcal{S}_{8,14},\mathcal{S}_{24,b}$ of Theorem~\ref{mainresultI}. They are determined by the factorisations
\[
\mathcal{S}_6=XY\cdot\mathcal{T}_{4,-},\qquad
\mathcal{S}_{12,-34}=\mathcal{T}_{4,+}\cdot\mathcal{T}_{8,-34},\qquad
\mathcal{S}_{8,14}=\mathcal{T}_{8,14},
\]
\[
\mathcal{S}_{24,b}=\mathcal{T}_{8,c}\cdot\mathcal{T}_{8,d}\cdot\mathcal{T}_{8,e},\qquad c+d+e=0,\quad cd+de+ec=756-4b,\quad cde=-(2576+4b),
\]
which are the explicit form of the identities $\mathcal{S}_{24,-66}=\mathcal{S}_{12,-34}^2$ and $\mathcal{S}_{24,42}=\mathcal{S}_{8,14}^3$ proved in Proposition~\ref{prop:necessaryS4}. The first instances are $d=6$ (case (v2), $t=0$, where $L_{d,Z}=\mathcal{S}_6$) and $d=14$ (case (v2), $t=1$, where $a_1=14$ and $L_{d,Z}=\mathcal{S}_6\cdot\mathcal{S}_{8,14}$); these correspond to the excluded value $a\ne 14$ in Table~\ref{tab:dihedral-configs-LdZ}, row $d=14$, $(4,1)$.

\medskip
\noindent\textit{$m=5$ (no parameter).} 
For $m=5$ the enlargement is $\operatorname{D}_5\to\operatorname{A}_5$. Since $\operatorname{D}_5$ is maximal in $\operatorname{A}_5$ (its normaliser in $\operatorname{A}_5$ is itself, as $\operatorname{A}_5$ has six Sylow $5$-subgroups), any element $c\in\operatorname{A}_5\setminus\operatorname{D}_5$ satisfies $\langle\operatorname{D}_5,c\rangle=\operatorname{A}_5$; in particular we may take $c$ to be an order-$3$ element of $\operatorname{A}_5$, which exists since $3\nmid 10=|\operatorname{D}_5|$. Consequently $L_{d,Z}$ is $\operatorname{A}_5$-invariant if and only if it is $\operatorname{D}_5$-invariant and $c$-invariant for such $c$.

The parameter values on $(\mathrm{E}3)$ at $m=5$ are those for which $L_{d,Z}$ is a product of the $\operatorname{A}_5$-invariants $\mathcal{G}_{12,-11\zeta_4},\mathcal{G}_{20,228\zeta_4,-494},\mathcal{G}_{30,-522\zeta_4,10005},\mathcal{G}_{60,b}$ of Theorem~\ref{mainresultII}. They are determined by the factorisations
\[
\mathcal{G}_{12,-11\zeta_4}=XY\cdot\mathcal{T}_{10,-11\zeta_4},\qquad
\mathcal{G}_{20,228\zeta_4,-494}=\mathcal{T}_{10,\zeta_4(114+50\sqrt5)}\cdot\mathcal{T}_{10,\zeta_4(114-50\sqrt5)},
\]
together with the analogous expressions for $\mathcal{G}_{30,-522\zeta_4,10005}$ and $\mathcal{G}_{60,b}$ in terms of $\mathcal{T}_{10}$ and $\mathcal{T}_{5,\pm}$ factors, which are the explicit form of the identities $\mathcal{G}_{60,a_1}=\mathcal{G}_{20}^3$ and $\mathcal{G}_{60,a_2}=\mathcal{G}_{30}^2$ proved in Proposition~\ref{prop:necessaryA5}. The first instance is $d=12$ (case (v1), $t=1$), where $a_1=-11\zeta_4$ and $L_{d,Z}=\mathcal{G}_{12,-11\zeta_4}$; this corresponds to the excluded value $a\ne -11\zeta_4$ in Table~\ref{tab:dihedral-configs-LdZ}, row $d=12$, $(5,0)$.

\medskip
Excluding the polyhedral loci (E2), (E3) and the dihedral locus (E1) ($\delta_+=\delta_-$ and $A=-A$), $G$ cannot strictly contain $\mathrm{D}_m$, so $\operatorname{Aut}_{\mathrm{red}}(\mathcal{X}) = \mathrm{D}_m$.
\end{proof}
\begin{rem}[On the polyhedral loci for $m=3,4,5$]
\label{rem:polyhedral-reduction}
The reduction of the polyhedral loci to the cases $m=4$ ($\operatorname{S}_4$) and $m=5$ ($\operatorname{A}_5$) in Sections~\ref{sec:proofsS4} and~\ref{sec:proofsA5} is a \emph{completeness} statement for those two sections, not an assertion that the $\operatorname{D}_3$-frame is free of polyhedral symmetry:
\begin{itemize}
    \item every embedding of $\operatorname{S}_4$ in $\operatorname{PGL}_2(K)$ contains a subgroup isomorphic to $\operatorname{D}_4$, and every embedding of $\operatorname{A}_5$ contains a subgroup isomorphic to $\operatorname{D}_5$. Consequently, up to $\operatorname{PGL}_2(K)$-equivalence, the $\operatorname{S}_4$-locus (E2) is exhausted by the $m=4$ frame and the $\operatorname{A}_5$-locus (E3) by the $m=5$ frame: every curve with reduced group $\operatorname{S}_4$ or $\operatorname{A}_5$ is conjugate to a curve of the $m=4$ or $m=5$ frame, and is therefore classified by Sections~\ref{sec:proofsS4} and~\ref{sec:proofsA5};
    \item this does \emph{not} say that the frames $m=3$ and $m=4$ are free of polyhedral enlargement. $\operatorname{D}_3$ is a maximal dihedral subgroup of $\operatorname{S}_4$ and of $\operatorname{A}_5$, and a form in the $\operatorname{D}_3$-frame may well acquire $\operatorname{S}_4$ or $\operatorname{A}_5$ symmetry; this is the content of (E2) and (E3), and the identification of the parameter values at $m=3$ is carried out in the proof of Proposition~\ref{prop:suffDm}, the first instance being $d=8$, $a=\pm\tfrac{7\sqrt2}{4}\zeta_4$, excluded in Table~\ref{tab:dihedral-configs-LdZ}.
\end{itemize}
\end{rem}
%%%%%%%%%%%%%%%%%
\begin{prop}[Structure of $\Aut(\mathcal{X})$ when $\Aut_{\operatorname{red}}(\mathcal{X})=\operatorname{D}_m$]\label{prop:Dm-structure}
Let $\mathcal{X}:Z^d+L_{d,Z}=0$ be as in Theorem \ref{mainresultIII}, assuming the parameters avoid the loci $(\mathrm{E}1)$--$(\mathrm{E}3)$ so that $\Aut_{\operatorname{red}}(\mathcal{X})=\operatorname{D}_m$, where $m\geq3$ and $m|d-2$. Write $d=2^e\cdot d_0$ with $2\nmid d_0$. Then the structure of $\Aut(\mathcal{X})$ is given by:
\begin{enumerate}[label=(\arabic*)]
\item If $d$ is odd (i.e., $e=0$), then $\Aut(\mathcal{X})\cong C_d\times\operatorname{D}_m$.

\item If $d$ is even (i.e., $e\geq1$) and $\delta=0$:
\begin{itemize}
\item If $2\nmid m$, then $\Aut(\mathcal{X})\cong C_d\times\operatorname{D}_m$.
\item If $2\mid m$ and $2m\mid d-2$, then $\Aut(\mathcal{X})\cong(C_d\circ\operatorname{D}_m)\rtimes C_2$.
\item If $2\mid m$ and $2m\nmid d-2$, then $\Aut(\mathcal{X})\cong C_{md}\rtimes_{1-d}C_2$.
\end{itemize}

\item If $d$ is even (i.e., $e\geq1$) and $\delta=1$:
\begin{itemize}
\item If $2\nmid m$, then $\Aut(\mathcal{X})\cong C_d\circ\operatorname{Dic}_m$.
\item If $2\mid m$, $4\nmid m$ and $4\mid d$, then $\Aut(\mathcal{X})\cong C_{md}\rtimes_{1-d}C_2$.
\item If $2\mid m$, $4\nmid m$ and $4\nmid d$, then $\Aut(\mathcal{X})\cong C_{md}\ \bar{\rtimes}_{1-d}\ C_2$.
\item If $4\mid m$ and $4\nmid d$, then $\Aut(\mathcal{X})\cong C_{d/2}\times\operatorname{Dic}_m$.
\end{itemize}
\end{enumerate}
\end{prop}
We will prove this proposition through a series of lemmas. The proof follows the systematic analysis of cases based on the parity of $d$, the parameter $\delta$, and the divisibility properties of $m$.
\begin{lem}\label{lem:odd-d}
If $d$ is odd (i.e., $e=0$), then $\Aut(\mathcal{X})\cong C_d\times\operatorname{D}_m$.
\end{lem}

\begin{proof}
Since $d$ is odd, the only smooth curves in Theorem \ref{mainresultIII} with odd degree occur in case (v2). From the condition $m|d-2$, and since $d$ is odd, $d-2$ is also odd, forcing $m$ to be odd.

The subgroup $\langle\rho\rangle=\langle\sigma^m\rangle\cong C_d$ is central in $\Aut(\mathcal{X})$. The quotient $\Aut(\mathcal{X})/\langle\rho\rangle$ is isomorphic to $\operatorname{D}_m$, generated by the images of $\sigma$ and $\tau$.

Since $\gcd(d,2m)=1$ (as both $d$ and $m$ are odd), the central extension
\[
1 \to C_d \to \Aut(\mathcal{X}) \to \operatorname{D}_m \to 1
\]
splits uniquely. %Indeed, the cohomology group $H^2(\operatorname{D}_m, C_d)$ is trivial because $\gcd(|C_d|, |\operatorname{D}_m|)=1$. 
Therefore, $\Aut(\mathcal{X})\cong C_d\times\operatorname{D}_m$.
\end{proof}

For the remainder of this section, we assume $d$ is even (i.e., $e\geq 1$).

\subsubsection{Case $\delta=0$}

When $\delta=0$, we have $\tau=[Y:X:Z]$ with $\tau^2=1$. Consider the automorphisms:
\[
\tilde{a}:=\operatorname{diag}(\zeta_m,\zeta_m^{-1},1),\quad \tilde{\tau}:=[Y:X:Z],\quad \tilde{\kappa}:=\operatorname{diag}(1,1,\zeta_d).
\]
Note that $\tilde{a}=\sigma^{d}$, since $\zeta_{md}^{d(d-1)}=\zeta_m^{d-1}=\zeta_{m}^{(d-2)+1}=\zeta_m$ as $m|d-2$.

\begin{lem}\label{lem:delta0-odd-m}
If $\delta=0$ and $2\nmid m$, then $\Aut(\mathcal{X})\cong C_d\times\operatorname{D}_m$.
\end{lem}

\begin{proof}
We have $\tilde{G}:=\langle\tilde{a},\tilde{\tau}\rangle\cong\operatorname{D}_m$ and $\tilde{H}:=\langle\tilde{\kappa}\rangle\cong C_d$. Since $m$ is odd, the only element of $\tilde{G}$ that could lie in $\tilde{H}$ would have the form $\operatorname{diag}(\alpha,\alpha,1)$ for some $\alpha\in K^*$. But in $\operatorname{D}_m$ with $m$ odd, the only such element is the identity. Therefore, $\tilde{G}\cap\tilde{H}=1$.

All three elements $\tilde{a}$, $\tilde{\tau}$, $\tilde{\kappa}$ lie in $\operatorname{PBD}(2,1)$, and $\tilde{\kappa}=\operatorname{diag}(1,1,\zeta_d)$ is central in $\operatorname{PBD}(2,1)$: for any $A=\begin{pmatrix}B&0\\0&1\end{pmatrix}$ with $B\in\operatorname{GL}_2(K)$, one has $A\tilde{\kappa}=\tilde{\kappa}A=\begin{pmatrix}B&0\\0&\zeta_d\end{pmatrix}$. Hence $\tilde{\kappa}$ commutes with $\tilde{a}$ and with $\tilde{\tau}$ in particular. 

Thus $\Aut(\mathcal{X})=\langle\tilde{G},\tilde{H}\rangle\cong\tilde{H}\times\tilde{G}\cong C_d\times\operatorname{D}_m$.
\end{proof}
%%%%%%%%%%%%%%%%%%%%%%%%%%%%%%%%%%%%%%%%%
\begin{lem}\label{lem:delta0-even-m}
If $\delta=0$ and $2\mid m$, then $\Aut(\mathcal{X}) \cong C_{md} \rtimes_{1-d} C_2$. In particular, this extension splits as $(C_d \circ \operatorname{D}_m) \rtimes C_2$ if and only if $2m \mid d-2$.
\end{lem}

\begin{proof}
We have $\tilde{G}:=\langle\tilde{a},\tilde{\tau}\rangle\cong\operatorname{D}_m$ and $\tilde{H}:=\langle\tilde{\kappa}\rangle\cong C_d$. Since $2\mid m$, the nontrivial element $\tilde{a}^{m/2} = \operatorname{diag}(-1,-1,1)= \tilde{\kappa}^{d/2}$ lies in $\tilde{G}\cap\tilde{H}$. Since $\tilde{G}$ and $\tilde{H}$ commute, their generated subgroup $\tilde{N} := \langle\tilde{G},\tilde{H}\rangle$ is isomorphic to the central product $C_d\circ\operatorname{D}_m$.

The full group $\Aut(\mathcal{X})$ is generated by $\tilde{N}$ and $\sigma= \operatorname{diag}(\zeta_{md}^{d-1},\zeta_{md}^{-1},1)$, which satisfies $\tilde{\tau}\sigma\tilde{\tau} = \sigma^{1-d}$. Since $(1-d)^2 = 1 + d(d-2) \equiv 1 \pmod{md}$ (as $m \mid d-2$), this defines a valid semidirect product $\Aut(\mathcal{X}) = \langle\sigma, \tilde{\tau}\rangle \cong C_{md} \rtimes_{1-d} C_2$. Moreover, $\sigma^2 = \tilde{a}\tilde{\kappa}^{-k} \in \tilde{N}$ where $k:=\frac{d-2}{m}$, so $\tilde{N}$ has index $2$ in $\Aut(\mathcal{X})$.

To analyze whether $1 \to \tilde{N} \to \Aut(\mathcal{X}) \to C_2 \to 1$ splits, we test for involutions in the outer coset $\Aut(\mathcal{X}) \setminus \tilde{N}$:
\begin{itemize}
    \item[(i)] \textit{Diagonal elements:} The diagonal subgroup of $\Aut(\mathcal{X})$ is $\langle \sigma \rangle \cong C_{md}$, whose index-$2$ subgroup $\langle \sigma^2 \rangle$ is $\tilde{N} \cap \langle \sigma \rangle$. Thus, any diagonal element in $\Aut(\mathcal{X}) \setminus \tilde{N}$ has the form $\sigma^{2j+1}$. Its square $\sigma^{2(2j+1)}$ has order dividing $md$. Since $2 \mid m$ and $2 \mid d$, $4 \mid md$, so $2(2j+1) \not\equiv 0 \pmod{md}$. Hence no diagonal element in $\Aut(\mathcal{X}) \setminus \tilde{N}$ is an involution.
    \item[(ii)] \textit{Non-diagonal elements:} These take the form $g = \tilde{\tau} h \sigma$ with $h = \operatorname{diag}(\zeta_m^r, \zeta_m^{-r}, \zeta_d^s)% \in \langle \tilde{a}, \tilde{\kappa} \rangle
    $. Writing $h \sigma = \operatorname{diag}(\alpha, \beta, \gamma)$ with $\alpha = \zeta_m^r \zeta_{md}^{d-1}$, $\beta = \zeta_m^{-r} \zeta_{md}^{-1}$, and $\gamma = \zeta_d^s$, we have $g^2 = \operatorname{diag}(\alpha\beta, \alpha\beta, \gamma^2)$ in $\operatorname{PGL}_3(K)$. Thus $g^2 = 1$ if and only if $\alpha\beta = \gamma^2$. Since $\alpha\beta = \zeta_{md}^{d-2} = \zeta_d^k$ and $\gamma^2 = \zeta_d^{2s}$, an involution exists if and only if $2s \equiv k \pmod d$ has an integer solution $s$.
\end{itemize}

If $2m \mid d-2$, then $k$ is even, so $s = k/2$ solves $2s \equiv k \pmod d$. This yields the explicit involution $\eta := \tilde{\tau}\sigma\tilde{\kappa}^{k/2} = \tilde{\tau}\tilde{\sigma} \in \Aut(\mathcal{X}) \setminus \tilde{N}$, where $\tilde{\sigma} := \operatorname{diag}(\zeta_{2m}, \zeta_{2m}^{-1}, 1)$. To see that $\eta \notin \tilde{N}$, note that for any diagonal matrix $h \in \tilde{N}$, the ratio $h_{11}/h_{22} = \zeta_m^{2r} \in \langle \zeta_m^2 \rangle$, whereas $\tilde{\sigma}_{11}/\tilde{\sigma}_{22} = \zeta_m \notin \langle \zeta_m^2 \rangle$. Thus the sequence splits, giving $\Aut(\mathcal{X}) \cong (C_d \circ \operatorname{D}_m) \rtimes C_2$.

If $2m \nmid d-2$, then $k$ is odd. Since $d$ is even, $2s \equiv k \pmod d$ has no integer solution $s$. Combined with (i), $\Aut(\mathcal{X}) \setminus \tilde{N}$ contains no involutions, and the sequence does not split.
\end{proof}
\subsubsection{Case $\delta=1$}
When $\delta=1$, we have $\tau=[Y:X:\zeta_{2d}Z]$ with $\tau^2=\rho=\sigma^m.$

\begin{lem}\label{lem:delta1-odd-m}
If $\delta=1$ and $2\nmid m$, then $\Aut(\mathcal{X})\cong C_d\circ\operatorname{Dic}_m$.
\end{lem}

\begin{proof}
Consider the automorphisms:
\[
\tilde{b}:=\operatorname{diag}(\zeta_{2m},\zeta_{2m}^{-1},1),\quad \tilde{\tau}_1:=[Y:X:\zeta_4Z],\quad \tilde{\kappa}:=\operatorname{diag}(1,1,\zeta_d).
\]
Observe that $\tilde{b}=\sigma^{d/2}$; indeed, since $d$ is even, we have $\zeta_{md}^{(d/2)(d-1)}=\zeta_{2m}^{d-1}=\zeta_{2m}$. Similarly, $\tilde{\tau}_1=\tau^{d/2}$.

Let $\tilde{G}:=\langle\tilde{b},\tilde{\tau}_1\rangle$. One verifies the relations:
\[
\tilde{b}^{2m}=1,\quad \tilde{\tau}_1^2=\tilde{b}^m,\quad \tilde{\tau}_1\tilde{b}\tilde{\tau}_1^{-1}=\tilde{b}^{-1},
\]
which constitute the standard presentation of the dicyclic group $\operatorname{Dic}_m$ of order $4m$.

Set $\tilde{H}:=\langle\tilde{\kappa}\rangle\cong C_d$. The intersection $\tilde{G}\cap\tilde{H}$ is generated by $\operatorname{diag}(1,1,-1)$, which equals both $\tilde{b}^m$ and $\tilde{\kappa}^{d/2}$; hence $\tilde{G}\cap\tilde{H}\cong C_2$. Moreover, the generators of $\tilde{G}$ and $\tilde{H}$ commute: $\tilde{b}$ and $\tilde{\kappa}$ are diagonal, and a direct computation shows $\tilde{\tau}_1\tilde{\kappa}=\tilde{\kappa}\tilde{\tau}_1$. Thus $\tilde{G}$ and $\tilde{H}$ centralize each other.

Consequently, $\Aut(\mathcal{X})=\langle\tilde{G},\tilde{H}\rangle$ is isomorphic to the central product $C_d\circ\operatorname{Dic}_m$, where the two factors intersect in the common central subgroup $C_2$. Its order is $\frac{d\cdot 4m}{2}=2dm$.
\end{proof}

For the remaining cases with $\delta=1$ and $2\mid m$, write $m=2m_0$.
\begin{lem}\label{lem:delta1-even-m-not4}
Assume $\delta=1$, $2\mid m$, $4\nmid m$. Then
\[
\operatorname{Aut}(\mathcal{X})\cong
\begin{cases}
C_{md}\rtimes_{1-d}C_2 & \text{if }4\mid d,\\[2pt]
C_{md}\ \bar{\rtimes}_{1-d}\ C_2 & \text{if }4\nmid d.
\end{cases}
\]
\end{lem}

\begin{proof}
The relations give $\sigma^{md}=1$, $\tau\sigma\tau^{-1}=\sigma^{1-d}$ and $\tau^2=\sigma^m$; in particular $\langle\sigma\rangle\cong C_{md}$ has index $2$ in $\operatorname{Aut}(\mathcal{X})=\langle\sigma,\tau\rangle$. Every element of the outer coset is $\tau\sigma^j$, and
\[
(\tau\sigma^j)^2=(\tau\sigma^j\tau^{-1})\,\tau^2\,\sigma^j=\sigma^{j(1-d)+m+j}=\sigma^{j(2-d)+m},
\]
so $\tau\sigma^j$ is an involution if and only if $lj\equiv1\pmod d$ where $l:=(d-2)/m$.

Writing $m=2m_0$ with $m_0$ odd, we have $d=2(1+m_0l)$, so $\gcd(l,d)=\gcd(l,2)$; this equals $1$ exactly when $l$ is odd, equivalently $1+m_0l$ even, i.e.\ $4\mid d$.

If $4\mid d$, the congruence is solvable, $j_0\equiv l^{-1}\pmod d$ gives an involution $\eta=\tau\sigma^{j_0}$ outside $\langle\sigma\rangle$ with $\langle\sigma\rangle\cap\langle\eta\rangle=1$, and $|\langle\sigma,\eta\rangle|=2md=|\operatorname{Aut}(\mathcal{X})|$; hence $\operatorname{Aut}(\mathcal{X})=\langle\sigma\rangle\rtimes\langle\eta\rangle$.

If $4\nmid d$, the congruence has no solution; no element of the outer coset is an involution, and a splitting $1\to\langle\sigma\rangle\to\operatorname{Aut}(\mathcal{X})\to C_2\to1$ would require one.
\end{proof}
\begin{lem}\label{lem:delta1-even-m-4}
If $\delta=1$, $4\mid m$, then $\Aut(\mathcal{X})\cong C_{d/2}\times\operatorname{Dic}_m$.
\end{lem}

\begin{proof}
Since $4\nmid d$ and $e=1$, we have $d=2d_0$ with $d_0$ odd. Consider:
\[
\tilde{u}:=\operatorname{diag}(\zeta_{2m}^{d-1},\zeta_{2m}^{-1},1),\quad
\tilde{\tau}_2:=\tau^{d_0}=[Y:X:\zeta_4Z],\quad
\tilde{\kappa}_1:=\rho^2=\operatorname{diag}(1,1,\zeta_{d_0}).
\]
We verify that $\tilde{G}:=\langle\tilde{u},\tilde{\tau}_2\rangle\cong\operatorname{Dic}_m$:
\[
\tilde{u}^{2m}=1,\quad \tilde{\tau}_2^2=\tilde{u}^m,\quad \tilde{\tau}_2\tilde{u}\tilde{\tau}_2^{-1}=\tilde{u}^{-1}.
\]

Now $\tilde{H}:=\langle\tilde{\kappa}_1\rangle\cong C_{d/2}$ has order $d_0$, which is odd. Since $\tilde{G}\cong\operatorname{Dic}_m$ has order $4m$, and $\gcd(d_0,4m)=1$ (as $d_0$ is odd and $\gcd(d_0,m)=1$ because $m|d-2$), the subgroups $\tilde{G}$ and $\tilde{H}$ have coprime orders.

Moreover, $\tilde{\kappa}_1=\operatorname{diag}(1,1,\zeta_{d_0})$ is central in $\operatorname{PBD}(2,1)$, as in Lemma~\ref{lem:delta0-odd-m}; since $\tilde{u}$ and $\tilde{\tau}_2$ both lie in $\operatorname{PBD}(2,1)$, $\tilde{\kappa}_1$ commutes with both of them. Thus $\tilde{\kappa}_1$ is central in $\operatorname{Aut}(\mathcal{X})$.

Since $\tilde{G}\cap\tilde{H}=1$ (their orders are coprime), we have $\Aut(\mathcal{X})=\langle\tilde{G},\tilde{H}\rangle\cong\tilde{H}\times\tilde{G}\cong C_{d/2}\times\operatorname{Dic}_m$.
\end{proof}

\begin{proof}[Proof of Proposition \ref{prop:Dm-structure}]
Proposition \ref{prop:Dm-structure} follows immediately from Lemmas \ref{lem:odd-d}, \ref{lem:delta0-odd-m}, \ref{lem:delta0-even-m}, \ref{lem:delta1-odd-m}, \ref{lem:delta1-even-m-not4}, and \ref{lem:delta1-even-m-4}, which cover all possible cases.
\end{proof}

%%%%%%%%%%%%%%%%%%%%%
\subsection{Complete proof of Theorem \ref{mainresultIII}}
\begin{proof}[Proof of Theorem \ref{mainresultIII}]
Proposition \ref{prop:necessaryDm} shows that if $\operatorname{D}_m\subseteq\Aut_{\operatorname{red}}(\mathcal{X})$, then $L_{d,Z}$ must have one of the forms (v1)-(v3), with the given conditions for smoothness and to avoid larger symmetry groups.

Proposition \ref{prop:suffDm} shows that, on the complement of the loci (E1)--(E3), $\Aut_{\operatorname{red}}(\mathcal{X})=\operatorname{D}_m$; the loci themselves are identified there.

The detailed analysis in Proposition \ref{prop:Dm-structure} establishes the group structure on that complement.
\end{proof}
% \begin{proof}[Proof of Theorem \ref{mainresultIII}]

% Proposition \ref{prop:suffDm} shows that any such $L_{d,Z}$ defines a smooth curve with $\Aut_{\operatorname{red}}(\mathcal{X})=\operatorname{D}_m$ (except the excluded cases when $m=4,5$).

% The detailed analysis in Proposition \ref{prop:Dm-structure} establishes the group structure formulas. We provide additional details for the exceptional cases:

% When $m=4$, the additional symmetries that would make $\Aut_{\operatorname{red}}(\mathcal{X})=\operatorname{S}_4$ correspond exactly to $L_{d,Z}$ factoring as products of $\mathcal{S}_6$, $\mathcal{S}_{8,14}$, and $\mathcal{S}_{24,a}$ forms.

% When $m=5$, the additional symmetries that would make $\Aut_{\operatorname{red}}(\mathcal{X})=\operatorname{A}_5$ correspond exactly to $L_{d,Z}$ factoring as products of $\mathcal{G}_{12,-11\zeta_4}$, $\mathcal{G}_{20,228\zeta_4,-494}$, $\mathcal{G}_{30,-522\zeta_4,10005}$, and $\mathcal{G}_{60,a}$ forms.

% Thus the conditions in Theorem \ref{mainresultIII} are both necessary and sufficient.
% \end{proof}
%%%%%%%%%%%%%%%%%%%%%%%%%%%%%%%%%%%%%%%%%%%%%%%%%%

\section{The symmetric group case $\operatorname{S}_4$}\label{sec:proofsS4}
In this section, we characterize smooth plane curves $\mathcal{X}:Z^d+L_{d,Z}=0$ of degree $d\geq 4$, with the property that $L_{d,Z}$ is a relative $\operatorname{S}_4$-invariant binary form. By Theorem~\ref{mainresultI}, such curves satisfy $\operatorname{S}_4\subseteq\Aut_{\mathrm{red}}(\mathcal{X})$. Since $\operatorname{S}_4$ is a maximal finite subgroup of $\operatorname{PGL}_2(K)$ — indeed $24\nmid 60$, so $\operatorname{S}_4\not\subseteq\operatorname{A}_5$, and the finite subgroups of $\operatorname{PGL}_2(K)$ are $C_n,\operatorname{D}_n,\operatorname{A}_4,\operatorname{S}_4,\operatorname{A}_5$ — this inclusion is an equality: $\Aut_{\mathrm{red}}(\mathcal{X})=\operatorname{S}_4$, with no possibility of enlargement to $\operatorname{A}_5$. Furthermore, by \cite[Theorem 2.3]{Harui}, $N=C_d$.
\subsection{Necessary conditions}
\begin{prop}\label{prop:necessaryS4}
Let $\mathcal{X}$ be a smooth plane curve of degree $d\geq4$, with the property that $\operatorname{S}_4\subseteq\operatorname{Aut}_{\operatorname{red}}(\mathcal{X})$. Then, $\mathcal{X}$ must be defined by an equation of the form $Z^d+L_{d,Z}=0$, where $L_{d,Z}$ is as in Theorem~\ref{mainresultI}.

The conditions $a_i\notin\{-66,42\}$ and $a_i\neq a_j$ (for $i\neq j$) are necessary and sufficient for $\mathcal{X}$ to be smooth. Specifically, these conditions ensure that neither the resultants of distinct factors of $L_{d,Z}$ nor the discriminants of the factors $\mathcal{S}_{24,a_i}$ vanish.
\end{prop}
\begin{proof}
By Lemma~\ref{cor:product-minimal}, any relative $\operatorname{S}_4$-invariant binary form decomposes as
\begin{equation}\label{eqn:S4-general}
L_{d,Z} = \mathcal{S}_6^{\epsilon} \cdot \mathcal{S}_{12,-34}^{\epsilon'} \cdot \mathcal{S}_{8,14}^{\epsilon''} \cdot \prod_{i=1}^{t} \mathcal{S}_{24,a_i}
\end{equation}
with $\epsilon,\epsilon',\epsilon''\in\{0,1\}$, $t \ge 0$ (where $t=0$ yields an empty product equal to $1$), and $\mathcal{S}_{24,a_i}$ is the $1$-parameter family of minimal relative invariants of degree $24$. 

Recall from Lemma~\ref{cor:product-minimal} that while $\operatorname{S}_4$ contains $\operatorname{D}_4 = \langle \operatorname{diag}(\zeta_4,1), [Y:X] \rangle$, extending the relative action to full $\operatorname{S}_4$ requires invariance under the order $3$ element $\phi = \begin{pmatrix}1 & -1 \\ \zeta_4 & \zeta_4\end{pmatrix}$. This condition uniquely fixes the coefficients of the degree $8$ and $12$ minimal forms to $b=14$ and $a=-34$, while $u_i \in K$ remain free parameters for the degree $24$ forms.

By Lemma~\ref{lem:smoothness}, $\mathcal{X}$ is smooth if and only if $L_{d,Z}$ has no multiple root. Such a root arises either from a common root of two distinct factors, detected by the corresponding resultant, or from a repeated root of a single factor, detected by its discriminant. A direct computation gives
\begin{align*}
\Res(\mathcal{S}_6,\mathcal{S}_{12,-34}) &\neq 0, & \Res(\mathcal{S}_{12,-34},\mathcal{S}_{24,a}) &\propto (a+66)^{12}, \\
\Res(\mathcal{S}_6,\mathcal{S}_{8,14}) &\neq 0, & \Res(\mathcal{S}_{8,14},\mathcal{S}_{24,a}) &\propto (a-42)^{8}, \\
\Res(\mathcal{S}_6,\mathcal{S}_{24,a}) &\neq 0, & \Res(\mathcal{S}_{24,a},\mathcal{S}_{24,b}) &\propto (a-b)^{24}, \\
\Res(\mathcal{S}_{12,-34},\mathcal{S}_{8,14}) &\neq 0, & \operatorname{Disc}(\mathcal{S}_{24,a}) &= 0 \iff a\in\{-66,42\}.
\end{align*}
In the second column, $\mathcal{S}_{24,a}$ degenerates into a power of a lower-degree minimal invariant at the two excluded values, since
\[
\mathcal{S}_{24,-66}=\mathcal{S}_{12,-34}^{\,2},\qquad \mathcal{S}_{24,42}=\mathcal{S}_{8,14}^{\,3},
\]
so the discriminant of $\mathcal{S}_{24,a}$ vanishes on the same locus as its resultants with $\mathcal{S}_{12,-34}$ and $\mathcal{S}_{8,14}$. We emphasize that $\operatorname{Disc}(\mathcal{S}_{24,a_i})$ is a condition on the single factor $\mathcal{S}_{24,a_i}$ and therefore applies whether or not $\mathcal{S}_{12,-34}$ or $\mathcal{S}_{8,14}$ occurs in $L_{d,Z}$: at $a_i=-66$ the factor $\mathcal{S}_{24,a_i}$ already contains a square, and at $a_i=42$ it contains a cube.

Consequently $\mathcal{X}$ is smooth if and only if $a_i\notin\{-66,42\}$ and $a_i\neq a_j$ $(i\neq j)$.
\end{proof}
\begin{rem} 
When $\epsilon=1$ ($\mathcal{S}_6$ present), the form arises from relative $\operatorname{D}_4$-invariants analyzed in Section~\ref{sec:proofsDm} with $m=4$. Specifically:
\begin{itemize}
    \item[(i)] Case (v2) of Theorem~\ref{mainresultIII} with $\delta_-=1$ gives $L_{d,Z}=\mathcal{S}_6\cdot \prod\mathcal{T}_{8,a_i}$.
    \item[(ii)] Case (v3) gives $L_{d,Z}=XY(X^8-Y^8)\prod\mathcal{T}_{8,a_i}=\mathcal{S}_6(X^4+Y^4)\prod\mathcal{T}_{8,a_i}$.
\end{itemize}
Requiring full relative $\operatorname{S}_4$-invariance forces these $\mathcal{T}_{8,a_i}$ factors to assemble into $\mathcal{S}_{8,14}$ and $\mathcal{S}_{24,u_i}$ blocks.

When $\epsilon=0$, the forms (containing $\mathcal{S}_{8,14}$ and $\mathcal{S}_{12,-34}$) are pure relative $\operatorname{S}_4$-invariants not arising from simple $\operatorname{D}_4$-invariant products with $XY$.
\end{rem}

\subsection{Sufficiency and group structure analysis}

\begin{prop}\label{prop:suffS4}
Given a smooth plane curve $\mathcal{X}:Z^{d}+L_{d,Z}=0$ as in Theorem~\ref{mainresultI}, its automorphism group is intransitive. Moreover, $\Aut_{\operatorname{red}}(\mathcal{X})=\operatorname{S}_4$.
\end{prop}
\begin{proof}
By construction, $L_{d,Z}$ is a relative $\operatorname{S}_4$-invariant form. The matrix generators of $\operatorname{S}_4\subset\operatorname{PGL}_2(K)$ lift to automorphisms in $\operatorname{PBD}(2,1)\subset\operatorname{PGL}_3(K)$ via suitable diagonal scaling on $Z$ (analogous to Proposition~\ref{prop:suffDm}), so $\operatorname{S}_4\subseteq\Aut_{\mathrm{red}}(\mathcal{X})$.

Next, we show that $\operatorname{Aut}(\mathcal{X})$ is intransitive. Since $\rho=\operatorname{diag}(1,1,\zeta_d)$ generates the kernel $N$ of the reduction map $\operatorname{Aut}(\mathcal{X})\to\Aut_{\mathrm{red}}(\mathcal{X})$, then $N$ is normal in $\operatorname{Aut}(\mathcal{X})$. By the framework established in the proof of Proposition~\ref{prop:suffDm} , a transitive $\operatorname{Aut}(\mathcal{X})$ would be a Fermat descendant. After conjugating $\rho$ into $\operatorname{diag}(1,1,\zeta_d)$, this would give $\operatorname{Aut}(\mathcal{X})\subseteq N_{\Aut(\mathcal{F}_d)}(N)$, hence again
$\operatorname{Aut}_{\mathrm{red}}(\mathcal{X})\subseteq\;N_{\Aut(\mathcal{F}_d)}(N)/N\;\cong\;\operatorname{D}_d
$. This is impossible, since $\operatorname{D}_d=C_d\rtimes C_2$ is metabelian, whereas $\operatorname{S}_4\trianglerighteq\operatorname{A}_4\trianglerighteq V_4\trianglerighteq 1$ has derived length $3$. Therefore $\operatorname{Aut}(\mathcal{X})$ is intransitive.

Finally, by the maximality of $\operatorname{S}_4$ in $\operatorname{PGL}_2(K)$ and that $\operatorname{Aut}_{\mathrm{red}}(\mathcal{X})$ is finite, we conclude that $\operatorname{Aut}_{\mathrm{red}}(\mathcal{X}) = \operatorname{S}_4$.
\end{proof}
\begin{prop}[Structure of $\Aut(\mathcal{X})$ when $\Aut_{\operatorname{red}}(\mathcal{X})=\operatorname{S}_4$]\label{prop:S4-structure}
Let $\mathcal{X}:Z^d+L_{d,Z}=0$ be as in Theorem \ref{mainresultI}, with $\Aut_{\operatorname{red}}(\mathcal{X})=\operatorname{S}_4$. Then the structure of $\Aut(\mathcal{X})$ is given by:
\begin{enumerate}[label=(\arabic*)]
\item If $\epsilon'=0$, then $\Aut(\mathcal{X})\cong C_d\circ\operatorname{GL}_2(\mathbb{F}_3)$.
\item If $\epsilon'=1$, then:
\begin{itemize}
    \item If $\epsilon=1$, then $\Aut(\mathcal{X})\cong C_d\circ\operatorname{CSU}_2(\mathbb{F}_3)$.
    \item If $\epsilon=0$, then $\Aut(\mathcal{X})\cong (C_d\circ\widetilde{\operatorname{A}}_4) \,{}^\bullet\, C_2$.
\end{itemize}
\end{enumerate}
\end{prop}

We will prove this proposition through a series of lemmas. 

Consider the automorphisms:
\[
\tilde{a}:=\operatorname{diag}(1,-1,\zeta_4),\quad\tilde{\tau}:=[Y:X:\zeta_{4}Z],\quad
\tilde{\kappa}:=\operatorname{diag}(1,1,\zeta_d),\]
\[ 
\tilde{u}:=\begin{pmatrix}
    1&\zeta_4&0\\
    1&-\zeta_4&0\\
    0&0&(1+\zeta_4)\zeta_{12}\cdot\zeta_3^{-\epsilon''}
\end{pmatrix}.
\]
These generate the automorphism subgroup $G:=\tilde{H}\,\circ\,\tilde{G}_0\cong C_d\,\circ\,\widetilde{\operatorname{A}}_4$ of order $12d$, where $\tilde{H}:=\langle\tilde{\kappa}\rangle\cong C_d$ and $\tilde{G}_0:=\langle\tilde{a},\tilde{\tau},\tilde{u}\rangle\cong\widetilde{\operatorname{A}}_4$.

The relations in $\tilde{G}_0$ are:
\begin{equation}\label{relationsA4}
\tilde{a}^4=\tilde{u}^3=1,\quad \tilde{\tau}^2=\tilde{a}^2,\quad \tilde{\tau}\tilde{a}\tilde{\tau}^{-1}=\tilde{a}^{-1},\quad \tilde{u}\tilde{a}\tilde{u}^{-1}=\tilde{\tau},\quad \tilde{u}\tilde{\tau}\tilde{u}^{-1}=\tilde{\tau}\tilde{a}.
\end{equation}

% \begin{lem}\label{lem:S4-GL2}
% If $\epsilon'=0$, then $\Aut(\mathcal{X})\cong C_d\circ\operatorname{GL}_2(\mathbb{F}_3)$.
% \end{lem}

% \begin{proof}
% In this case, the curve admits an additional automorphism:
% \[
% \tilde{b}_0:=\begin{pmatrix}
%     1&1&0\\
%     1&-1&0\\
%     0&0&\sqrt{2}
% \end{pmatrix},
% \]
% % We verify $\tilde{b}_0$ is an automorphism: it preserves $L_{d,Z}$ because $\Lambda(\tilde{b}_0)=\begin{pmatrix}1&1\\1&-1\end{pmatrix}\in\operatorname{PGL}_2(K)$ normalizes $\operatorname{S}_4$ and fixes the $\operatorname{S}_4$-invariant $L_{d,Z}$. Also, $\tilde{b}_0(Z^d)=(\sqrt{2}Z)^d=2^{d/2}Z^d$, and since $d$ is even (as $d=6\epsilon+12\epsilon'+8\epsilon''+24t$ with $\epsilon'=0$, $d$ is divisible by 2), $2^{d/2}$ is an integer power of 2, which can be absorbed by scaling.
% which satisfies:
% \begin{equation}\label{relationsS4II}
% \tilde{b}_0^2=1,\quad \tilde{b}_0\tilde{\tau}\tilde{b}_0^{-1}=\tilde{a},\quad \tilde{b}_0\tilde{u}\tilde{b}_0^{-1}=(\tilde{u}\tilde{a})^{-1}.
% \end{equation}

% Therefore, the group $\tilde{G}_1:=\langle\tilde{G}_0,\tilde{b}_0\rangle\cong\operatorname{GL}_2(\mathbb{F}_3)$ of order 48. Moreover, $\tilde{G}_1\cap\tilde{H}=\langle\tilde{a}^2\rangle\cong C_2$, which lies in the center of $\langle\tilde{H},\tilde{G}_1\rangle$. Thus $\Aut(\mathcal{X})=\langle\tilde{H},\tilde{G}_1\rangle$ is the central product $\tilde{H}\,\circ\,\tilde{G}_0\cong C_d\circ\operatorname{GL}_2(\mathbb{F}_3)$ of order $\frac{d\cdot 48}{2}=24d$.
% \end{proof}
\begin{lem}\label{lem:S4-GL2}
If $\epsilon'=0$, then $\Aut(\mathcal{X})\cong C_d\circ\operatorname{GL}_2(\mathbb{F}_3)$.
\end{lem}

\begin{proof}
In this case, the curve admits an additional matrix transformation:
\[
\tilde{b}_0:=\begin{pmatrix}
    1&1&0\\
    1&-1&0\\
    0&0&\sqrt{2}
\end{pmatrix}.
\]
To verify that $\tilde{b}_0 \in \Aut(\mathcal{X})$, recall that when $\epsilon'=0$, $L_{d,Z}(X,Y)$ is a homogeneous binary form of degree $d = 6\epsilon + 8\epsilon'' + 24t$, constructed from the fundamental $\operatorname{S}_4$-invariants $\mathcal{S}_6(X,Y)$, $\mathcal{S}_{8,14}(X,Y)$, and the family of degree-24 invariants $\mathcal{S}_{24,a_i}(X,Y)$. Under $(X,Y) \mapsto (X+Y, X-Y)$, a direct computation yields:
\[
\mathcal{S}_6(X+Y, X-Y) = 2^{6/2} \mathcal{S}_6(X,Y), \quad \mathcal{S}_{8,14}(X+Y, X-Y) = 2^{8/2} \mathcal{S}_{8,14}(X,Y),\]  \[\mathcal{S}_{24,a_i}(X+Y, X-Y) = 2^{24/2} \mathcal{S}_{24,a_i}(X,Y).
\]
Consequently, any constituent homogeneous $\operatorname{S}_4$-invariant $P_k(X,Y)$ of degree $k$ satisfies $P_k(X+Y, X-Y) = 2^{k/2} P_k(X,Y)$. By the homogeneity of $L_{d,Z}(X,Y)$ in degree $d$, we obtain:
\[
L_{d,Z}(X+Y, X-Y) = 2^{d/2} L_{d,Z}(X,Y).
\]
On the $Z$-variable, $\tilde{b}_0(Z) = \sqrt{2}Z$, so $\tilde{b}_0(Z^d) = (\sqrt{2}Z)^d = 2^{d/2}Z^d$. Thus, we confirm that $\tilde{b}_0 \in \Aut(\mathcal{X})$.

The element $\tilde{b}_0$ satisfies:
\begin{equation}\label{relationsS4II}
\tilde{b}_0^2=1,\quad \tilde{b}_0\tilde{\tau}\tilde{b}_0^{-1}=\tilde{a},\quad \tilde{b}_0\tilde{u}\tilde{b}_0^{-1}=(\tilde{u}\tilde{a})^{-1}.
\end{equation}

Therefore, $\tilde{G}_1:=\langle\tilde{G}_0,\tilde{b}_0\rangle\cong\operatorname{GL}_2(\mathbb{F}_3)$ has order 48. Moreover, $\tilde{G}_1\cap\tilde{H}=\langle\tilde{a}^2\rangle\cong C_2$, which lies in the center of $\langle\tilde{H},\tilde{G}_1\rangle$. Thus $\Aut(\mathcal{X})=\langle\tilde{H},\tilde{G}_1\rangle$ is the central product $\tilde{H}\,\circ\,\tilde{G}_1\cong C_d\circ\operatorname{GL}_2(\mathbb{F}_3)$ of order $\frac{d\cdot 48}{2}=24d$.
\end{proof}

\begin{lem}\label{lem:S4-CSU2}
If $\epsilon'=1$, then:
\begin{itemize}
    \item[(i)] If $\epsilon=1$, then $\Aut(\mathcal{X}) \cong C_d \circ \operatorname{CSU}_2(\mathbb{F}_3)$.
    \item[(ii)] If $\epsilon=0$, then $1 \to C_d \circ \widetilde{\operatorname{A}}_4 \to \Aut(\mathcal{X}) \to C_2 \to 1$ does not split over $C_d \circ \widetilde{\operatorname{A}}_4$, and $\Aut(\mathcal{X}) \cong (C_d \circ \widetilde{\operatorname{A}}_4) \,{}^\bullet\, C_2$.
\end{itemize}
\end{lem}

\begin{proof}
When $\epsilon'=1$, the form $L_{d,Z}(X,Y)$ contains the factor $\mathcal{S}_{12}(X,Y)$, which satisfies $\mathcal{S}_{12}(X+Y, X-Y) = -2^{12/2} \mathcal{S}_{12}(X,Y)$. Combined with the remaining factors $\mathcal{S}_6^\epsilon$, $\mathcal{S}_{8,14}^{\epsilon''}$, and $\mathcal{S}_{24,a_i}$, we obtain $L_{d,Z}(X+Y, X-Y) = -2^{d/2} L_{d,Z}(X,Y)$.

To construct an automorphism in the outer coset over $C_d \circ \widetilde{\operatorname{A}}_4$, consider matrix transformations $g_\mu := \begin{pmatrix} 1 & 1 & 0 \\ 1 & -1 & 0 \\ 0 & 0 & \mu \end{pmatrix}$. The pullback $g_\mu^*(Z^d + L_{d,Z}) = \mu^d Z^d - 2^{d/2} L_{d,Z}$ defines an element of $\Aut(\mathcal{X})$ if and only if $\mu^d = -2^{d/2}$.
\begin{itemize}
\item Case (i): $\epsilon=1$ (equivalently, $d \equiv 2 \pmod 4$).
Setting $\tilde{b}_1 := \begin{pmatrix} 1 & 1 & 0 \\ 1 & -1 & 0 \\ 0 & 0 & \sqrt{2}\zeta_4 \end{pmatrix}$, we have $\tilde{b}_1(Z^d) = (\sqrt{2}\zeta_4 Z)^d = 2^{d/2}\zeta_4^d Z^d = -2^{d/2} Z^d$ since $d/2$ is odd. Thus $\tilde{b}_1 \in \Aut(\mathcal{X})$.

The element $\tilde{b}_1$ satisfies:
\begin{equation}\label{relationsS4III}
\tilde{b}_1^2=\tilde{a}^2,\quad \tilde{b}_1\tilde{\tau}\tilde{b}_1^{-1}=\tilde{a},\quad \tilde{b}_1\tilde{u}\tilde{b}_1^{-1}=(\tilde{u}\tilde{a})^{-1}.
\end{equation}
Therefore, $\tilde{G}_2:=\langle\tilde{G}_0,\tilde{b}_1\rangle\cong\operatorname{CSU}_2(\mathbb{F}_3)$ of order 48. Since $\tilde{G}_2\cap\tilde{H}=\langle\tilde{a}^2\rangle\cong C_2$ lies in the center of $\langle\tilde{H},\tilde{G}_2\rangle$, we have $\Aut(\mathcal{X})=\langle\tilde{H},\tilde{G}_2\rangle \cong C_d\circ\operatorname{CSU}_2(\mathbb{F}_3)$.

\item Case (ii): $\epsilon=0$ (equivalently, $d \equiv 0 \pmod 4$).
Here $\zeta_4^d = +1$, so $\tilde{b}_1 \notin \Aut(\mathcal{X})$. Instead, set $\tilde{\beta}_1 := \begin{pmatrix} 1 & 1 & 0 \\ 1 & -1 & 0 \\ 0 & 0 & \sqrt{2}\zeta_{2d} \end{pmatrix}$. Since $\zeta_{2d}^d = -1$, we have $\tilde{\beta}_1(Z^d) = -2^{d/2} Z^d$, so $\tilde{\beta}_1 \in \Aut(\mathcal{X})$.

In $\mathrm{PGL}_3(K)$, $\tilde{\beta}_1^2 =\operatorname{diag}(1, 1, \zeta_d) = \tilde{\kappa}$, which generates $\tilde{H} \cong C_d$. Any element in $\Aut(\mathcal{X}) \setminus (C_d \circ \widetilde{\operatorname{A}}_4)$ has the form $g = \tilde{\beta}_1 \tilde{\kappa}^s h$ for some $h \in \widetilde{\operatorname{A}}_4$. Squaring $g$ yields an element whose $Z$-entry is proportional to $\zeta_d^{2s+1}$. For $g^2$ to equal $\tilde{a}^2 = \operatorname{diag}(1,1,-1) = \tilde{\kappa}^{d/2}$, we would require $2s+1 \equiv d/2 \pmod d$. Since $4 \mid d$, $d/2$ is even, so $2s+1 \equiv d/2 \pmod d$ has no integer solution $s$. Thus, no outer element squares to $\tilde{a}^2$, and the sequence $1 \to C_d \circ \widetilde{\operatorname{A}}_4 \to \Aut(\mathcal{X}) \to C_2 \to 1$ does not split as a central product $C_d \circ \operatorname{CSU}_2(\mathbb{F}_3)$.
\end{itemize}
\end{proof}

\begin{proof}[Proof of Proposition \ref{prop:S4-structure}]
Proposition \ref{prop:S4-structure} follows immediately from Lemmas \ref{lem:S4-GL2} and \ref{lem:S4-CSU2}, which cover both cases.
\end{proof}
\subsection{Complete proof of Theorem \ref{mainresultI}}

\begin{proof}[Proof of Theorem \ref{mainresultI}]
Proposition \ref{prop:necessaryS4} shows that if $\operatorname{S}_4\subseteq\Aut_{\operatorname{red}}(\mathcal{X})$, then $L_{d,Z}$ must have the form Eqn. (\ref{definingS4}), with the given conditions for smoothness and to avoid larger symmetry groups.

Proposition \ref{prop:suffS4} shows that any such $L_{d,Z}$ defines a smooth curve with the property that $\Aut_{\operatorname{red}}(\mathcal{X})=\operatorname{S}_4$.

The detailed analysis in Proposition \ref{prop:S4-structure} establishes the group structure formulas.
\end{proof}
%%%%%%%%%%%%%%%%%%%%%%%%%%%%%%%%%%%%%%%%%%%%%%%%%%%%%%%%%%%%%%%%%%%%%%%%
\section{The alternating group case $\operatorname{A}_5$}\label{sec:proofsA5}
In this section, we characterize smooth plane curves $\mathcal{X}:Z^d+L_{d,Z}=0$ of degree $d\geq 4$, with the property that $L_{d,Z}$ is a relative $\operatorname{A}_5$-invariant binary form. By Theorem~\ref{mainresultII}, such curves have reduced automorphism group isomorphic to $\operatorname{A}_5$. Furthermore, by \cite[Theorem 2.3]{Harui}, $N=C_d$.

\subsection{Necessary conditions}
\begin{prop}\label{prop:necessaryA5}
Let $\mathcal{X}$ be a smooth plane curve of degree $d\geq4$, with the property that $\operatorname{A}_5\subseteq\operatorname{Aut}_{\operatorname{red}}(\mathcal{X})$. Then, $\mathcal{X}$ must be defined by an equation of the form $Z^d+L_{d,Z}=0$, where $L_{d,Z}$ is as in Theorem~\ref{mainresultII}.

The conditions $a_i\notin\{2^2\cdot19\cdot436999,\ -2^2\cdot9377\cdot5323\}$ and $a_i\neq a_j$ (for $i\neq j$) are necessary and sufficient for $\mathcal{X}$ to be smooth. Specifically, these conditions ensure that neither the resultants of distinct factors of $L_{d,Z}$ nor the discriminants of the factors $\mathcal{G}_{60,a_i}$ vanish.
\end{prop}
\begin{proof}
By Lemma~\ref{cor:product-minimal}, any relative $\operatorname{A}_5$-invariant binary form decomposes as
\begin{equation}\label{eqn:A5-general}
L_{d,Z} = \mathcal{G}_{12,-11\zeta_4}^{\epsilon} \cdot \mathcal{G}_{20,228\zeta_4,-494}^{\epsilon'} \cdot \mathcal{G}_{30,-522\zeta_4,10005}^{\epsilon''} \cdot \prod_{i=1}^{t} \mathcal{G}_{60,a_i}
\end{equation}
where $\epsilon,\epsilon',\epsilon''\in\{0,1\}$, $t \ge 0$ (where $t=0$ yields an empty product equal to $1$), and $\mathcal{G}_{60,a_i}$ is the $1$-parameter family of minimal relative invariants of degree $60$.

Recall from Lemma~\ref{cor:product-minimal} that while $\operatorname{A}_5$ contains $\operatorname{D}_5 = \langle \operatorname{diag}(\zeta_5,1), [Y:X] \rangle$, extending the relative action to full $\operatorname{A}_5$ requires invariance under an order $2$ element $\phi = \begin{pmatrix} -\frac{1+\sqrt{5}}{2} & -\zeta_4 \\ \zeta_4 & \frac{1+\sqrt{5}}{2} \end{pmatrix}$. Computing the action of $\phi$ on each factor uniquely fixes the coefficients of the lower-degree minimal forms to $a=-11\zeta_4$, $(u,v)=(228\zeta_4, -494)$, and $(b,c)=(-522\zeta_4, 10005)$, while $a_i \in K$ remain free parameters for the degree $60$ forms.

By Lemma~\ref{lem:smoothness}, $\mathcal{X}$ is smooth if and only if $L_{d,Z}$ has no multiple root. Such a root arises either from a common root of two distinct factors, detected by the corresponding resultant, or from a repeated root of a single factor, detected by its discriminant. A direct computation gives
\begin{align*}
\Res(\mathcal{G}_{12,-11\zeta_4},\mathcal{G}_{20,228\zeta_4,-494}) &\neq 0, & \Res(\mathcal{G}_{20,228\zeta_4,-494},\mathcal{G}_{60,a}) &\propto (a-2^2\cdot19\cdot436999)^{20}, \\[2pt]
\Res(\mathcal{G}_{20,228\zeta_4,-494},\mathcal{G}_{30,-522\zeta_4,10005}) &\neq 0, & \Res(\mathcal{G}_{30,-522\zeta_4,10005},\mathcal{G}_{60,a}) &\propto (a+2^2\cdot9377\cdot5323)^{30}, \\[2pt]
\Res(\mathcal{G}_{12,-11\zeta_4},\mathcal{G}_{30,-522\zeta_4,10005}) &\neq 0, & \Res(\mathcal{G}_{60,a},\mathcal{G}_{60,b}) &\propto (a-b)^{60}, \\[2pt]
\Res(\mathcal{G}_{12,-11\zeta_4},\mathcal{G}_{60,a}) &\neq 0, & \operatorname{Disc}(\mathcal{G}_{60,a}) &= 0 \iff \\[2pt]
& & a &\in\{2^2\cdot19\cdot436999,\,-2^2\cdot9377\cdot5323\}.
\end{align*}

In the second column, $\mathcal{G}_{60,a}$ degenerates into a power of a lower-degree minimal invariant at the two excluded values, since
\[
\mathcal{G}_{60,\,2^2\cdot19\cdot436999}=\mathcal{G}_{20,228\zeta_4,-494}^{\,3},\qquad \mathcal{G}_{60,\,-2^2\cdot9377\cdot5323}=\mathcal{G}_{30,-522\zeta_4,10005}^{\,2}
\]
% (this is verified by comparing the coefficients of the two sides; both sides are palindromic in $(XY)^5$). 
% We emphasize that $\operatorname{Disc}(\mathcal{G}_{60,a_i})$ is a condition on the single factor $\mathcal{G}_{60,a_i}$ and therefore applies whether or not $\mathcal{G}_{20,228\zeta_4,-494}$ or $\mathcal{G}_{30,-522\zeta_4,10005}$ occurs in $L_{d,Z}$; conversely, a degree-$60$ $\operatorname{A}_5$-invariant divisor has a repeated point if and only if it is $5\cdot(12\text{-orbit})$, $3\cdot(20\text{-orbit})$ or $2\cdot(30\text{-orbit})$, and the first of these is excluded as above.
So the discriminant of $\mathcal{G}_{60,a}$ vanishes on the same locus as its resultants with $\mathcal{G}_{20,228\zeta_4,-494}$ and $\mathcal{G}_{30,-522\zeta_4,10005}$. This is a condition on the single factor $\mathcal{G}_{60,a_i}$ and applies whether or not those lower-degree factors occur in $L_{d,Z}$.

Consequently $\mathcal{X}$ is smooth if and only if $a_i\notin\{2^2\cdot19\cdot436999,\,-2^2\cdot9377\cdot5323\}$ and $a_i\neq a_j$ $(i\neq j)$.
\end{proof}
\begin{rem} 
When $\epsilon=1$, $\epsilon'=1$, or $\epsilon''=1$, the form arises from relative $\operatorname{D}_5$-invariants analyzed in Section~\ref{sec:proofsDm} with $m=5$. Specifically, these correspond to cases where products of $\mathcal{T}_{10,a}$ factors combine into minimal relative $\operatorname{A}_5$-invariants.

When $\epsilon=\epsilon'=\epsilon''=0$, the forms are pure relative $\operatorname{A}_5$-invariants (products of $\mathcal{G}_{60,w_i}$) not arising from simple $\operatorname{D}_5$-invariant products with $XY$.
\end{rem}

\subsection{Sufficiency and group structure analysis}
\begin{prop}\label{prop:suffA5}
Given a smooth plane curve $\mathcal{X}:Z^{d}+L_{d,Z}=0$ as in Theorem~\ref{mainresultII}, its automorphism group is intransitive. Moreover, $\Aut_{\operatorname{red}}(\mathcal{X})=\operatorname{A}_5$.
\end{prop}

\begin{proof}
By construction, $L_{d,Z}$ is a relative $\operatorname{A}_5$-invariant form (Definition~\ref{defn:relative-invariant}). The matrix generators of $\operatorname{A}_5\subset\operatorname{PGL}_2(K)$ lift to automorphisms in $\operatorname{PBD}(2,1)\subset\operatorname{PGL}_3(K)$ via suitable diagonal scaling on $Z$ (analogous to Proposition~\ref{prop:suffDm}), so $\operatorname{A}_5\subseteq\Aut_{\mathrm{red}}(\mathcal{X})$.

Since $N=\langle\operatorname{diag}(1,1,\zeta_d)\rangle$ is is normal in $\operatorname{Aut}(\mathcal{X})$. By the framework established in the proof of Proposition~\ref{prop:suffDm}, a transitive $\operatorname{Aut}(\mathcal{X})$ would be a Fermat descendant. Hence
$
\operatorname{Aut}_{\mathrm{red}}(\mathcal{X})\subseteq\;N_{\Aut(\mathcal{F}_d)}(N)/N\;\cong\;\operatorname{D}_d
$. This is impossible, since $\operatorname{D}_d=C_d\rtimes C_2$ is metabelian, hence solvable, whereas $\operatorname{A}_5$ is simple non-abelian and not solvable. Therefore $\operatorname{Aut}(\mathcal{X})$ is intransitive.

Finally, $\operatorname{A}_5$ is a maximal finite subgroup of $\operatorname{PGL}_2(K)$. Since $\operatorname{A}_5\subseteq\Aut_{\mathrm{red}}(\mathcal{X})\subset\operatorname{PGL}_2(K)$ and $\Aut_{\mathrm{red}}(\mathcal{X})$ is finite, we conclude that $\Aut_{\mathrm{red}}(\mathcal{X})=\operatorname{A}_5$.
\end{proof}
Consider the automorphisms:
\[
\rho:=\operatorname{diag}(1,1,\zeta_d),\quad \tau:=[Y:X:\zeta_{2d}^{\epsilon''}Z],\quad
\sigma_1:=\operatorname{diag}(\zeta_5,1,\zeta_{5d}^{\epsilon}),\]
\[ 
\sigma_2:=\begin{pmatrix}
 -\nu_0& -\zeta_4 & 0\\
 \zeta_4 & \nu_0 &0\\
 0&0&\nu
\end{pmatrix}
\]
with $\nu, \nu_0$ are as given in Theorem \ref{mainresultII}.
% The group $G:=\langle\rho,\tau,\sigma_1,\sigma_2\rangle$ fits into the central extension
% \begin{equation}\label{A5sequence}
%     1 \to \langle\rho\rangle \cong C_d \to G \to \operatorname{A}_5 \to 1.
% \end{equation}
Each generator has the form
\[
g = \begin{pmatrix} B_g & 0 \\ 0 & \lambda_g \end{pmatrix},
\]
with $B_g \in \operatorname{GL}_2$ representing an element of $\operatorname{A}_5 \subset \operatorname{PGL}_2$,
and $\lambda_g \in \mathbb{C}^\times$ chosen so that $g$ preserves the equation $\mathcal X: Z^d+L_{d,Z}=0$.
\begin{prop}[Structure of $\Aut(\mathcal{X})$ when $\Aut_{\operatorname{red}}(\mathcal{X})=\operatorname{A}_5$]\label{prop:A5-structure}
Let $\mathcal{X}:Z^d+L_{d,Z}=0$ be as in Theorem \ref{mainresultII}, with $\operatorname{A}_5\subseteq\Aut_{\operatorname{red}}(\mathcal{X})$. Write $d = 2^e d_0$ with $d_0$ odd. Then $\Aut(\mathcal{X}) \cong C_{d_0} \times G_0$, where $G_0$ is the unique non‑split central extension of $\operatorname{A}_5$ by $C_{2^e}$. More concretely, 
% \[
% G_0 \cong 
% \begin{cases}
% \widetilde{\operatorname{A}}_5 & \text{if } e = 1, \\[1mm]
% \widetilde{\operatorname{A}}_5 \rtimes C_2 & \text{if } e = 2, \\[1mm]
% \widetilde{\operatorname{A}}_5 \, {}^{\bullet} C_{2^{e-1}} & \text{if } e \ge 3 .
% \end{cases}
% \]
\(
G_0 \cong \widetilde{\operatorname{A}}_5 \circ C_{2^e}.
\)
\end{prop}

First, let $\kappa = \rho^{d_0}$, so that $\kappa$ has order $2^e$ and $\langle \rho \rangle = C_{d_0} \times \langle \kappa \rangle$. Then $\Aut(\mathcal{X}) \cong C_{d_0} \times G_0$ where $G_0 = \langle \tau, \sigma_1, \sigma_2, \kappa \rangle$.

\begin{lem}\label{lem:A5-nonsplit}
The extension $1 \to \langle\kappa\rangle \to G_0 \to \operatorname{A}_5 \to 1$ is non-split.
\end{lem}
\begin{proof}
Suppose, to the contrary, that there exists a section $\mathfrak s: \operatorname{A}_5 \longrightarrow G_0$ satisfying $\Lambda \circ \mathfrak s = \operatorname{id}_{\operatorname{A}_5}$. Write $\mathfrak s(\Lambda(\tau))=\tau\kappa^m$, $\mathfrak s(\Lambda(\sigma_2))=\sigma_2\kappa^n$, so that $\mathfrak s(\Lambda(\tau\sigma_2))=\tau\sigma_2\kappa^{m+n}$ since $\kappa$ is central.

If $\epsilon''=1$, then $\tau^2=\rho$ in $G_0$. Since $\Lambda(\tau)$ has order $2$, we write $\mathfrak s(\Lambda(\tau)) = \tau\kappa^n$ for some $n \in \mathbb{Z}$. Using $\kappa=\rho^{d_0}$, squaring yields $1 = (\tau\kappa^n)^2 = \rho^{1+2nd_0}$, which implies $2^e \mid (1+2nd_0)$. This is impossible because $1+2nd_0$ is odd and $2^e$ is even (since $e \ge 1$).

If $\epsilon''=0$, the generators satisfy $\tau^2=1$ and $\sigma_2^2(\tau\sigma_2)^{-2}=\kappa^{2^{e-1}}$ in $G_0$. Since $\Lambda(\tau)$, $\Lambda(\sigma_2)$, $\Lambda(\tau\sigma_2)$ are involutions of $\operatorname{A}_5$,
\[
1=\mathfrak s(\Lambda(\tau))^2=\kappa^{2m},\qquad
1=\mathfrak s(\Lambda(\sigma_2))^2=\sigma_2^2\kappa^{2n},\qquad
1=\mathfrak s(\Lambda(\tau\sigma_2))^2=(\tau\sigma_2)^2\kappa^{2(m+n)}.
\]
The first gives $\kappa^{2m}=1$; the second gives $\sigma_2^2=\kappa^{-2n}$; the third gives $(\tau\sigma_2)^2=\kappa^{-2(m+n)}$. Thus
\[
\kappa^{2^{e-1}}=\sigma_2^2(\tau\sigma_2)^{-2}=\kappa^{-2n}\cdot\kappa^{2(m+n)}=\kappa^{2m}=1,
\]
contradicting, since $2^{e-1}\not\equiv 0\pmod{2^e}$. Hence no section exists.
\end{proof}
\begin{proof}[Proof of Proposition \ref{prop:A5-structure}] By the classical computation of the Schur multiplier of $\operatorname{A}_5$ (see \cite[§4]{SuzukiI}), $H_2(\operatorname{A}_5)\cong C_2$; since $\operatorname{A}_5$ is perfect, the universal coefficient theorem gives $H^2(\operatorname{A}_5,C_{2^e})\cong\operatorname{Hom}(C_2,C_{2^e})=C_2$. Hence there is exactly one non-split central extension of $\operatorname{A}_5$ by $C_{2^e}$, namely $\widetilde{\operatorname{A}}_5\circ C_{2^e}$.

The decomposition $\operatorname{Aut}(\mathcal{X})\cong C_{d_0}\times G_0$ follows from $\operatorname{Aut}(\mathcal{X})=\langle\rho^{2^e}\rangle\times G_0$ above.
\end{proof}
\subsection{Complete proof of Theorem \ref{mainresultII}}

\begin{proof}[Proof of Theorem \ref{mainresultII}]
Proposition \ref{prop:necessaryA5} shows that if $\operatorname{A}_5\subseteq\Aut_{\operatorname{red}}(\mathcal{X})$, then $L_{d,Z}$ must have the form Eqn. (\ref{definingA5}), with the given conditions for smoothness. Proposition \ref{prop:suffA5} shows that any such $L_{d,Z}$ defines a smooth curve with the property that $\Aut_{\operatorname{red}}(\mathcal{X})=\operatorname{A}_5$. The detailed analysis in Proposition \ref{prop:A5-structure} establishes the group structure formulas.
\end{proof}
%%%%%%%%%%%%%%%%%%%%%%%%%%%%%%%%%%%%%%%%%%%%%%%%%%%%%%%%%%%%
\section{The alternating group case $\operatorname{A}_4$}\label{sec:proofsA4}
This group occupies an intermediate position between the dihedral groups $\operatorname{D}_m$ and the larger polyhedral groups $\operatorname{S}_4$ and $\operatorname{A}_5$, sharing structural features with both. As with the previous cases, smooth plane curves $\mathcal{X}:Z^d+L_{d,Z}=0$ with $\Aut_{\operatorname{red}}(\mathcal{X})\cong\operatorname{A}_4$ correspond to relative $\operatorname{A}_4$-invariant binary forms $L_{d,Z}$. The classification involves identifying the minimal relative $\operatorname{A}_4$-invariants, determining the conditions under which their products remain strictly $\operatorname{A}_4$-invariant (without acquiring larger symmetry), and describing the resulting structure of $\Aut(\mathcal{X})$.

Throughout this section, $\mathrm A_4\subset\operatorname{PGL}_2(K)$ denotes the tetrahedral subgroup with reduced generators
\[
\alpha=\operatorname{diag}(1,-1),\qquad\beta=[Y:X],\qquad\gamma=\begin{pmatrix}\zeta_4&\zeta_4\\1&-1\end{pmatrix}.
\]
The subgroup generated by $\alpha,\beta$ is $\operatorname{D}_2$; adjoining $\gamma$ gives $\mathrm A_4$ of order $12$.

\subsection{Necessary conditions}
\begin{prop}\label{prop:necessaryA4}
Let $\mathcal{X}$ be a smooth plane curve of degree $d\geq4$, with the property that $\operatorname{A}_4\subseteq\Aut_{\operatorname{red}}(\mathcal{X})$. Then, $\mathcal{X}$ must be defined by an equation of the form $Z^d+L_{d,Z}=0$, where $L_{d,Z}$ is as in Theorem~\ref{mainresultIV}.

The conditions $a_i\neq \pm 6\sqrt{3}\zeta_4$ and $a_i\neq a_j$ (for $i\neq j$) are necessary and sufficient for $\mathcal{X}$ to be smooth. Specifically, these conditions ensure that neither the resultants of distinct factors of $L_{d,Z}$ nor the discriminants of the factors $\mathcal{F}_{12,a_i}$ vanish.
\end{prop}
\begin{proof}
By Lemma~\ref{cor:product-minimal}, any relative $\operatorname{A}_4$-invariant binary form decomposes as
\begin{equation}\label{eqn:A4-general}
L_{d,Z} = \mathcal{S}_6^{\epsilon} \cdot \mathcal{S}_{4,+}^{\epsilon_+} \cdot \mathcal{S}_{4,-}^{\epsilon_-} \cdot \prod_{i=1}^t \mathcal{F}_{12,a_i},
\end{equation}
where $\epsilon,\epsilon_+,\epsilon_-\in\{0,1\}$, $t \ge 0$ (where $t=0$ yields an empty product equal to $1$), and $\mathcal{F}_{12,a_i}$ is the $1$-parameter family of minimal relative invariants of degree $12$.

Recall from Lemma~\ref{cor:product-minimal} that extending the relative action from $\operatorname{D}_2 = \langle \alpha, \beta \rangle$ to full $\operatorname{A}_4$ requires invariance under the order $3$ element $\gamma$, which leaves $\mathcal{S}_6$, $\mathcal{S}_{4,\pm}$ and every $\mathcal{F}_{12,a}$ invariant. Thus any product of these forms is relative $\operatorname{A}_4$-invariant without restricting the parameters.

By Lemma~\ref{lem:smoothness}, $\mathcal{X}$ is smooth if and only if $L_{d,Z}$ has no multiple root. Such a root arises either from a common root of two distinct factors, detected by the corresponding resultant, or from a repeated root of a single factor, detected by its discriminant. A direct computation gives
\begin{align*}
&\Res(\mathcal{S}_6,\mathcal{S}_{4,\pm}) \neq 0, \quad \Res(\mathcal{S}_6,\mathcal{F}_{12,a}) \neq 0, \quad \Res(\mathcal{S}_{4,+},\mathcal{S}_{4,-}) \neq 0, \\[4pt]
&\Res(\mathcal{S}_{4,\pm},\mathcal{F}_{12,a}) \propto (a\pm 6\sqrt3\zeta_4)^4, \\[2pt]
&\Res(\mathcal{F}_{12,a},\mathcal{F}_{12,b}) \propto (a-b)^{12}, \\[2pt]
&\operatorname{Disc}(\mathcal{F}_{12,a}) \propto (a^2+108)^8.
\end{align*}
Consequently $\mathcal{X}$ is smooth if and only if $a_i\neq \pm 6\sqrt3\zeta_4$ and $a_i\neq a_j\quad (i\neq j).$
\end{proof}

\subsection{Sufficiency and group structure analysis}

\begin{prop}[Sufficiency for $\operatorname{A}_4$ case]\label{prop:suffA4}
Given a smooth plane curve $\mathcal{X}:Z^{d}+L_{d,Z}=0$, where $L_{d,Z}$ is defined as in Theorem~\ref{mainresultIV}, then:
\begin{enumerate}
    \item $\Aut(\mathcal{X})$ is intransitive.
    \item $\Aut_{\operatorname{red}}(\mathcal{X})$ contains $\operatorname{A}_4$.
    \item $\Aut_{\operatorname{red}}(\mathcal{X})$ equals $\operatorname{A}_4$, $\operatorname{S}_4$, or $\operatorname{A}_5$ according to the exceptional cases listed in Theorem~\ref{mainresultIV}.
\end{enumerate}
\end{prop}

\begin{proof}
By the transitivity framework established in the proof of Proposition~\ref{prop:suffDm}, a transitive $\operatorname{Aut}(\mathcal{X})$ would be a Fermat descendant. After conjugating $\rho$ into $\operatorname{diag}(1,1,\zeta_d)$, the containment $\operatorname{Aut}(\mathcal{X})\subseteq N_{\Aut(\mathcal{F}_d)}(N)$ of Lemma~\ref{lem:uniqueCd}(iii) would give
\[
\operatorname{Aut}_{\operatorname{red}}(\mathcal{X})=\operatorname{Aut}(\mathcal{X})/N\;\subseteq\;\operatorname{D}_d .
\]
This is impossible: every subgroup of $\operatorname{D}_d=C_d\rtimes C_2$ is cyclic or dihedral, whereas $\operatorname{A}_4$ is neither. Hence $\mathcal{X}$ is not a Fermat descendant, and $\operatorname{Aut}(\mathcal{X})$ is intransitive.

The explicit generators $\alpha,\beta,\gamma,\rho$ given in Theorem~\ref{mainresultIV} preserve $Z^d+L_{d,Z}=0$ by construction. Their reduced images $\Lambda(\alpha),\Lambda(\beta),\Lambda(\gamma)$ generate $\operatorname{A}_4$ in $\operatorname{PGL}_2(K)$, so $\operatorname{A}_4 \subseteq \operatorname{Aut}_{\mathrm{red}}(\mathcal{X})$.

We now determine when the symmetry enlarges to $\operatorname{S}_4$ or $\operatorname{A}_5$.

\medskip
\noindent\textbf{Case $\operatorname{S}_4$:}
If $\Aut_{\operatorname{red}}(\mathcal{X})$ contains $\operatorname{S}_4$, then $L_{d,Z}$ must be relative $\operatorname{S}_4$-invariant. Comparing the degree formula with Theorem~\ref{mainresultI}:
\[
d = 6\epsilon + 4(\epsilon_++\epsilon_-) + 12t \quad\text{and}\quad d = 6\epsilon + 12\epsilon' + 8\epsilon'' + 24t'
\]
we get: $\epsilon_++\epsilon_-+\epsilon''\equiv0\,(\operatorname{mod}\,3)$. So \(\epsilon_+=+\epsilon_-=\epsilon''\), and $t=\epsilon'+2t'$. Moreover, $\operatorname{S}_4$ contains an element
\[
\psi_4:=\operatorname{diag}(\zeta_4,1)\in\operatorname{S}_4\setminus\operatorname{A}_4 .
\]
Since $\psi_4(\mathcal{F}_{12,a})=\mathcal{F}_{12,-a}$, invariance under $\psi_4$ forces the $\mathcal{F}_{12,a}$ factors to appear in pairs $\{\mathcal{F}_{12,a},\mathcal{F}_{12,-a}\}$ (or with $a=0$). The product $\mathcal{F}_{12,a}\mathcal{F}_{12,-a}$ equals the $\operatorname{S}_4$-invariant $\mathcal{S}_{24,a^2+66}$ of Theorem~\ref{mainresultI}. This yields exceptional case (i) of Theorem~\ref{mainresultIV}.

\medskip
\noindent\textbf{Case $\operatorname{A}_5$:}
If $\Aut_{\operatorname{red}}(\mathcal{X})$ contains $\operatorname{A}_5$, then $L_{d,Z}$ must be relative $\operatorname{A}_5$-invariant. Comparing degree formulas with Theorem~\ref{mainresultII}:
\[
d = 6\epsilon + 4(\epsilon_++\epsilon_-) + 12t \quad\text{and}\quad d = 12\epsilon' + 20\epsilon'' + 30\epsilon''' + 60t',
\]
we get: $\epsilon_++\epsilon_-+\epsilon''\equiv0\,(\operatorname{mod}\,3)$ and \(2\epsilon\equiv2\epsilon'''\,(\operatorname{mod}\,4)\). So $\epsilon_+=\epsilon_-=\epsilon''$, \(\epsilon=\epsilon'''\), and $t =2\epsilon+\epsilon'+ \epsilon_++ 5t'$. These numerical conditions force $L_{d,Z}$ to factorize as
\[
L_{d,Z} = 
\bigl(\underbrace{\mathcal{S}_6\cdot\mathcal{F}_{12,u}\cdot\mathcal{F}_{12,v}}_{:=\mathcal R_{30,u,v}}\bigr)^{\epsilon}
\cdot\bigl(\underbrace{\mathcal{S}_{8,14}\cdot\mathcal{F}_{12,c_1}}_{:=\mathcal R_{20,c_1}}\bigr)^{\epsilon_+}\cdot\mathcal F_{12,a}^{\epsilon'}
\cdot\prod_{j=1}^{t'}\underbrace{\prod_{i=1}^{5}\mathcal{F}_{12,a_{i,j}}}_{:=\mathcal{R}_{60,\{a_{j,1},\dots,a_{j,5}\}}}.
\]

We now determine the conditions on the lower-degree parameters. The normalizer of $\mathrm A_4$ in $\operatorname{PGL}_2(K)$ is $\mathrm S_4$; any $\mathrm A_5\subset\operatorname{PGL}_2(K)$ containing $\mathrm A_4$ is conjugate to a fixed one, say $\mathrm A_5^{+}$, by an element of $\operatorname{PGL}_2(K)$, and since both contain $\mathrm A_4$, the conjugating element may be taken in $N(\mathrm A_4)=\mathrm S_4$. The quotient $\mathrm S_4/\mathrm A_4\cong C_2$ therefore yields at most two such copies; they are distinct, because an element of the non-trivial coset cannot normalize $\mathrm A_5^{+}$ (as $N_{\operatorname{PGL}_2}(\mathrm A_5)=\mathrm A_5$ and $\mathrm S_4\cap\mathrm A_5=\mathrm A_4$). We fix the copy $\mathrm A_5^{+}$ and an element
\[
\psi=\begin{pmatrix}1&\omega\\ \tilde\omega&1\end{pmatrix}\in\mathrm A_5^{+}\setminus\mathrm A_4,\qquad
\omega:=-\nu_0^{-2}-\zeta_4\nu_0^{-1},\quad\tilde\omega:=\nu_0^{-2}-\zeta_4\nu_0^{-1},\qquad \nu_0:=\tfrac{1+\sqrt5}{2},
\]
% of order $5$. The matrix $\psi$ satisfies $\psi^5=1$ in $\operatorname{PGL}_2(K)$; a direct computation of the group $\langle\alpha,\beta,\gamma,\psi\rangle$ shows it has order $60$, 
% A direct computation shows $(\psi\beta)^3=1$ and that $\alpha,\gamma\in\langle\beta,\psi\rangle$; hence
% \[
% \langle\alpha,\beta,\gamma,\psi\rangle=\langle\beta,\psi\rangle\cong A_5,
% \]
% the latter being a quotient of $\langle a,b\mid a^2=b^5=(ab)^3=1\rangle$.
so $\psi\in\mathrm A_5^{+}$ and $\mathrm A_5^{+}=\langle\mathrm A_4,\psi\rangle$. (Equivalently, $\psi$ is the standard element of order $5$ in $\mathrm A_5$ in this realisation; the computation reduces to the relations $\psi^5=1$ and $(\psi\alpha)^3=1$ in $\operatorname{PGL}_2(K)$ together with the standard presentation of $\mathrm A_5$.)

Consequently an $\mathrm A_4$-invariant form is $\mathrm A_5^{+}$-invariant if and only if it is $\psi$-invariant. Forcing invariance of $\mathcal R_{30,u,v}$, $\mathcal R_{20,c_1}$, and $\mathcal F_{12,a}$ under $\psi$ uniquely fixes
\[
c_1=-\frac{38\sqrt{5}}{3},\quad u=-\frac{2}{45}\left(29\sqrt{5}+256\zeta_4\right),\quad v=-\frac{2}{45}\left(29\sqrt{5}-256\zeta_4\right),\quad a=\frac{22}{\sqrt{5}}.
\]
Additionally, we must verify the $\psi$-invariance condition:
\[
\mathcal{R}_{60,\{a_{j,1},\dots,a_{j,5}\}}(X+\omega Y,\tilde{\omega}X+Y)=\lambda \mathcal{R}_{60,\{a_{j,1},\dots,a_{j,5}\}}(X,Y)\quad\text{for some } \lambda\in K^*.
\tag{*}
\]

The second copy $\mathrm A_5^{-}$ of $\mathrm A_5$ containing $\mathrm A_4$ is obtained from $\mathrm A_5^{+}$ by conjugating by $\psi_4$; the corresponding $\operatorname{A}_5$-sublocus is the image of the $\mathrm A_5^{+}$-sublocus under the action of $\psi_4$ on the parameters. As computed above, $\psi_4(\mathcal F_{12,a})=\mathcal F_{12,-a}$, $\psi_4(\mathcal S_6)=\zeta_4\mathcal S_6$, $\psi_4(\mathcal S_{8,14})=\mathcal S_{8,14}$, and $\psi_4(\mathcal S_{4,\pm})=\mathcal S_{4,\mp}$; in particular, on the $\operatorname{A}_5$-locus (where $\epsilon_+=\epsilon_-$) the product $\mathcal S_{4,+}^{\epsilon_+}\mathcal S_{4,-}^{\epsilon_-}$ is fixed up to scalar. Hence $\psi_4$ acts on the parameters $(a_{i,j},c_1,u,v,a)$ by negation,
\[
(a_{i,j},c_1,u,v,a)\longmapsto(-a_{i,j},-c_1,-u,-v,-a),
\]
so that the $\mathrm A_5^{-}$-sublocus carries the negated values of those of $\mathrm A_5^{+}$. Both subloci are needed in Theorem~\ref{mainresultIV}(ii): by maximality of $\mathrm A_5$ among finite subgroups of $\operatorname{PGL}_2(K)$, an $\mathrm A_4$-invariant form $L_{d,Z}$ has reduced group strictly containing $\mathrm A_4$ and contained in $\mathrm A_5$ if and only if it lies on $\mathrm A_5^{+}$ or on $\mathrm A_5^{-}$.

Outside these two exceptional loci, $\Aut_{\operatorname{red}}(\mathcal{X})=\operatorname{A}_4$.
\end{proof}

\begin{lem}[Non-emptiness of the $\operatorname{A}_5$-locus]\label{lem:nonempty-case-ii}
The $\operatorname{A}_5$-locus of Theorem~\ref{mainresultIV}(ii) is non-empty; more precisely, each of the two subloci corresponding to the two copies $\mathrm A_5^{\pm}$ of $\mathrm A_5$ containing $\mathrm A_4$ is infinite.
\end{lem}
\begin{proof}
Let $\mathrm A_4\subset\mathrm A_5^{+}\subset\operatorname{PGL}_2(K)$ be as in Proposition~\ref{prop:suffA4}, with $\mathrm A_5^{+}=\langle\mathrm A_4,\psi\rangle$. Let $U\subset\mathbb P^1$ be the complement of the finite set of points with non-trivial $\mathrm A_5^{+}$-stabilizer; $U$ is Zariski-open and non-empty.

Fix $x\in U$. Since $[\mathrm A_5^{+}:\mathrm A_4]=5$, choosing coset representatives $g_1,\dots,g_5$ gives
\[
\mathrm A_5^{+}\cdot x=\bigsqcup_{i=1}^5\mathrm A_4\cdot y_i,\qquad y_i:=g_ix,
\]
a disjoint union of five $\mathrm A_4$-orbits of $12$ distinct points each.

By Lemma~\ref{cor:product-minimal}, the space $V_{12}$ of relative $\mathrm A_4$-invariants of degree $12$ is $2$-dimensional with projectivization $\{\mathcal F_{12,a}\}_{a\in\mathbb P^1}$; it is base-point-free, being spanned by $\mathcal F_{12,a}$ and $\mathcal F_{12,a'}$ for some $a\ne a'$, which are coprime. Hence evaluation at $y_i$ is a non-zero functional on $V_{12}$, whose kernel is a line; let $F_i$ generate it. Then $F_i$ vanishes at $y_i$, hence on $\mathrm A_4\cdot y_i$; since $\deg F_i=12=|\mathrm A_4\cdot y_i|$, we have $V(F_i)=\mathrm A_4\cdot y_i$, and $F_i=\mathcal F_{12,a_i}$ for a unique $a_i\in K^\times$. Thus the zero divisor of
\[
\mathcal R_{60,a_1\dots a_5}:=\prod_{i=1}^5\mathcal F_{12,a_i}
\]
is $\mathrm A_5^{+}\cdot x$. Since $\psi\in\mathrm A_5^{+}$, this divisor is $\psi$-invariant, so $V(\psi(\mathcal R_{60,a_1\dots a_5}))=\mathrm A_5^{+}\cdot x=V(\mathcal R_{60,a_1\dots a_5})$; two forms with the same zero divisor are proportional, so $\psi(\mathcal R_{60,a_1\dots a_5})=\lambda\,\mathcal R_{60,a_1\dots a_5}$ for some $\lambda\in K^\times$, which is $(*)$.

The orbits $\mathrm A_4\cdot y_i$ are disjoint, so $a_i\ne a_j$ for $i\ne j$; each $\mathcal F_{12,a_i}$ has $12$ distinct roots, so $a_i\ne\pm6\sqrt3\zeta_4$. Hence $(a_1,\dots,a_5)$ lies on the $\mathrm A_5^{+}$-sublocus of the $\operatorname{A}_5$-locus.

Finally, distinct $x\in U$ yield distinct tuples (an orbit $\mathrm A_5^{+}\cdot x$ has $60$ elements, so $x'\in U$ gives the same tuple only if $x'=x$); since $U$ is infinite, the $\mathrm A_5^{+}$-sublocus is infinite. By Proposition~\ref{prop:suffA4}, $\psi_4=\operatorname{diag}(\zeta_4,1)$ conjugates $\mathrm A_5^{+}$ to $\mathrm A_5^{-}$ and acts on the parameters by $a_i\mapsto -a_i$; hence the $\mathrm A_5^{-}$-sublocus, being the image of the $\mathrm A_5^{+}$-sublocus under negation, is also infinite.
\end{proof}
\begin{prop}[Structure of $\Aut(\mathcal{X})$ when $\Aut_{\operatorname{red}}(\mathcal{X})=\operatorname{A}_4$]\label{prop:A4-structure}
Let $\mathcal{X}:Z^d+L_{d,Z}=0$ be as in Theorem \ref{mainresultIV}, with $\Aut_{\operatorname{red}}(\mathcal{X})=\operatorname{A}_4$. Write $d=2^e d_0$ with $d_0$ odd and $e\ge 1$, then $\Aut(\mathcal{X})\cong C_{d_0}\times G_0$, where
    \[
    G_0\cong
    \begin{cases}
    \widetilde{\operatorname{A}}_4 & \text{if } e=1, \\[1mm]
    \widetilde{\operatorname{A}}_4\rtimes C_2 & \text{if } e=2, \\[1mm]
    \widetilde{\operatorname{A}}_4\,{}^{\bullet}\,C_{2^{e-1}} & \text{if } e\geq3.
    \end{cases}
    \]
\end{prop}
\begin{proof}
Consider the automorphism subgroup $G$ generated by the lifts of the reduced generators $\alpha,\beta,\gamma$ together with the homology $\kappa$:
\[
\hat\alpha:=\operatorname{diag}(1,-1,\zeta_4),\quad 
\hat\beta:=[Y:X:\zeta_4Z],\quad
\kappa:=\operatorname{diag}(1,1,\zeta_d).
\]
\[\hat\gamma:=\begin{pmatrix}
    1&\zeta_4&0\\
    1&-\zeta_4&0\\
    0&0&(-1+\zeta_4)\cdot\zeta_3^{\epsilon_{-}\cdot2^{\epsilon_{+}}}
\end{pmatrix}.\]
The subgroup $H:=\langle\hat\alpha,\hat\beta,\hat\gamma\rangle$ satisfies $H\cong\widetilde{\operatorname{A}}_4$, with relations
\[
\hat\alpha^4=1,\quad \hat\beta^2=\hat\alpha^2,\quad \hat\gamma^3=1,\quad 
\hat\beta\hat\alpha\hat\beta^{-1}=\hat\alpha^{-1},\quad 
\hat\gamma\hat\alpha\hat\gamma^{-1}=\hat\beta,\quad 
\hat\gamma\hat\beta\hat\gamma^{-1}=\hat\alpha^{-1}\hat\beta.
\]
Let $\kappa_0:=\kappa^{d_0}$ (order $2^e$). Since $d$ is even, $\kappa^{d/2}=\hat\alpha^2$, so $\langle\kappa\rangle\cap H=\langle\hat\alpha^2\rangle$; let $z:=\kappa_0^{2^{e-1}}=\hat\alpha^2$ be the unique central involution in $\langle\kappa_0\rangle$, and $H_0:=\langle\hat\alpha,\hat\beta,\hat\gamma,z\rangle=H\cong\widetilde{\operatorname{A}}_4$. Then $G_0=\langle H_0,\kappa_0\rangle$ fits into
\[
1\longrightarrow H_0\longrightarrow G_0\longrightarrow C_{2^{e-1}}\longrightarrow 1,
\]
since $\kappa_0^{2^{e-1}}=z\in H_0$, so the coset $\kappa_0H_0$ has order $2^{e-1}$. The three cases for $e$ are handled as follows.

\emph{Case $e=1$.} Then $\kappa_0=z\in H_0$, so $G_0=H_0\cong\widetilde{\operatorname{A}}_4$.

\emph{Case $e=2$.} The element $\kappa_0$ has order $4$ with $\kappa_0^2=z=\hat\alpha^2\in H_0$. Set $\eta:=\kappa_0\hat\alpha\in G_0$. Since $\kappa_0$ is central,
\[
\eta^2=\kappa_0^2\hat\alpha^2=z\cdot z=1,
\]
and $\eta\notin H_0$ (otherwise $\kappa_0=\eta\hat\alpha^{-1}\in H_0$, contradicting $\kappa_0\notin H_0$). Hence $\langle\eta\rangle$ is a subgroup of order $2$ meeting $H_0$ trivially, and the extension splits as $G_0\cong H_0\rtimes\langle\eta\rangle\cong\widetilde{\operatorname{A}}_4\rtimes C_2$.

\emph{Case $e\ge 3$.} Here $|G_0/H_0|=2^{e-1}$. Suppose the extension splits; then there exists $\eta\in G_0$ of order $2^{e-1}$ with $\eta H_0=\kappa_0H_0$. Writing $\eta=\kappa_0h$ for some $h\in H_0$, from $\eta^{2^{e-1}}=1$ and the centrality of $\kappa_0$ we get
\[
1=\eta^{2^{e-1}}=\kappa_0^{2^{e-1}}h^{2^{e-1}}=z\,h^{2^{e-1}},
\]
so $h^{2^{e-1}}=z$. This forces $h$ to have order $2^e\ge8$, but $\widetilde{\operatorname{A}}_4$ has no elements of order $\ge8$ (its Sylow $2$-subgroups are $Q_8$). Hence the extension does not split, and $G_0$ is the non-split central extension $\widetilde{\operatorname{A}}_4\,{}^{\bullet}\,C_{2^{e-1}}$.

Finally, since $\gcd(d_0,2^e)=1$, we have $\operatorname{Aut}(\mathcal{X})=\langle\kappa^{2^e}\rangle\times G_0\cong C_{d_0}\times G_0$.
\end{proof}

\subsection{Complete proof of Theorem \ref{mainresultIV}}
\begin{proof}[Proof of Theorem \ref{mainresultIV}]
Proposition \ref{prop:necessaryA4} shows that if $\operatorname{A}_4\subseteq\Aut_{\operatorname{red}}(\mathcal{X})$, then $L_{d,Z}$ must have the form \eqref{definingA4}, with the stated smoothness conditions on the parameters $a_i$.

Conversely, Proposition \ref{prop:suffA4} shows that every smooth curve of this form satisfies $\operatorname{A}_4\subseteq\Aut_{\operatorname{red}}(\mathcal{X})$, and identifies precisely the two exceptional loci (i) and (ii) on which the reduced automorphism group enlarges to $\operatorname{S}_4$ or $\operatorname{A}_5$, respectively. Lemma~\ref{lem:nonempty-case-ii} shows that locus (ii) is non-empty, both subloci corresponding to $\mathrm A_5^{\pm}$ being infinite. On the complement of these two loci, $\Aut_{\operatorname{red}}(\mathcal{X})=\operatorname{A}_4$.

Finally, Proposition \ref{prop:A4-structure} establishes the structure of $\Aut(\mathcal{X})$ under the assumption $\Aut_{\operatorname{red}}(\mathcal{X})=\operatorname{A}_4$.
\end{proof}
%%%%%%%%%%%%%%%%%%%%%%%%%%%%%%%%%%%%%%%%%%%
\section{MAGMA Implementation}\label{sec:magma}
All theorems presented in this work are constructive: they not only establish the existence of smooth plane curves with prescribed reduced automorphism groups, but also provide explicit defining equations and concrete generators for the full automorphism group. To make the tables of
Section~\ref{sec:proofsDm}--\ref{sec:proofsA4} directly reproducible, we
have implemented the four normal forms as a MAGMA package, \texttt{Curves.m},
together with a companion file \texttt{Tables.m} that generates the
tables themselves.

\subsection*{What the package does}

Given a degree $d$, the package
\begin{enumerate}
\item enumerates all tuples $(\epsilon,\epsilon',\epsilon'',t)$
  (for $\operatorname{S}_4$ and $\operatorname{A}_5$) and
  $(\epsilon,\epsilon_+,\epsilon_-,t)$ (for $\operatorname{A}_4$)
  satisfying the degree formula of the corresponding theorem, and all
  $(m,t)$ with $m\mid d-2$, $m\ge3$ (for $\operatorname{D}_m$);
\item writes the defining equation $Z^d+L_{d,Z}(X,Y)=0$ with $L_{d,Z}$
  in the form of the theorem;
\item prints the smoothness conditions of Propositions~\ref{prop:necessaryDm},
  \ref{prop:necessaryS4}, \ref{prop:necessaryA5}, \ref{prop:necessaryA4};
\item flags the exceptional loci of Theorem~\ref{mainresultIII} (E1)--(E3)
  and Theorem~\ref{mainresultIV}(i)--(ii) at which the reduced
  automorphism group enlarges;
\item prints the group structure and order predicted by the theorem,
  together with explicit projective generators.
\end{enumerate}

Four dedicated procedures produce Tables~\ref{tab:dihedral-configs-LdZ},
\ref{tab:s4-configs}, \ref{tab:a4-configs}, \ref{tab:a5-configs}
directly:
\begin{verbatim}
    ShowDihedralTable(dmin, dmax);   // Table 1
    ShowS4Table(dmin, dmax);         // Table 2
    ShowA4Table(dmin, dmax);         // Table 3
    ShowA5Table(dmin, dmax);         // Table 4
\end{verbatim}
The ranges shown in this paper are
\[
\begin{aligned}
&\operatorname{D}_m: 4\le d\le 15, &\quad
&\operatorname{S}_4: 4\le d\le 30,\\
&\operatorname{A}_4: 4\le d\le 32, &\quad
&\operatorname{A}_5: 12\le d\le 100.
\end{aligned}
\]

At isolated parameter values a row of
Table~\ref{tab:dihedral-configs-LdZ} or \ref{tab:a4-configs} acquires
the larger reduced group of Theorem~\ref{mainresultIII}(E2)--(E3) or
Theorem~\ref{mainresultIV}(i)--(ii). The Remarks column records the
excluded values when the theorems give them in closed form, and writes
``possibly $E$'' when the locus is structural but its parameter values
are characterized only in Proposition~\ref{prop:suffDm} or
Theorem~\ref{mainresultIV}(ii). Each row is thus a generic
configuration on the open subset of parameter space on which the
displayed group acts.
\subsection*{The equivalence relation}
Theorems~\ref{mainresultII}--\ref{mainresultIV} classify $L_{d,Z}$ up to
the residual action of $\mu_d^2/\mu_d$ on $(X:Y:Z)$ and of
$N_{\operatorname{PGL}_2(K)}(\Aut_{\mathrm{red}}(\mathcal X))$ on
$(X:Y)$. The enumeration lists one representative per orbit:
\begin{itemize}
\item $\operatorname{D}_m$: the $\mathcal T_{m,-}$ representative of
  case~(v2), since $\mathcal T_{m,+}$ is
  $\operatorname{PGL}_2(K)$-equivalent via $Y\mapsto\zeta_4Y$;
\item $\operatorname{A}_4$: the $\mathcal S_{4,-}$ representative of an
  $(\epsilon_+,\epsilon_-)=(0,1)$ configuration; the $(1,0)$ rows are
  equivalent under $\operatorname{diag}(\zeta_4,1)$, and rows with
  $\epsilon_+=\epsilon_-$, $t=0$ lie on the $\operatorname{S}_4$-locus
  and appear only in Table~\ref{tab:s4-configs};
\item $\operatorname{S}_4$, $\operatorname{A}_5$: the minimal relative
  invariants of Theorems~\ref{mainresultI} and~\ref{mainresultII} are
  the unique $\operatorname{PGL}_2(K)$-normalized forms in their
  degrees, so no further duplicate occurs.
\end{itemize}

At isolated parameter values a row of
Table~\ref{tab:dihedral-configs-LdZ} or \ref{tab:a4-configs} acquires
the larger reduced group of Theorem~\ref{mainresultIII}(E2)--(E3) or
Theorem~\ref{mainresultIV}(i)--(ii). The Remarks column records the
excluded values when the theorems give them in closed form, and writes
``possibly $E$'' when the locus is structural but its parameter values
are characterized only in Proposition~\ref{prop:suffDm} or
Theorem~\ref{mainresultIV}(ii). Each row is thus a generic
configuration on the open subset of parameter space on which the
displayed group acts.

\subsection*{Limitations}

The package does no verification beyond the enumeration. Smoothness of
$Z^d+L_{d,Z}=0$ at specific numerical parameters is not tested (the
printed conditions are those of
Propositions~\ref{prop:necessaryDm}--\ref{prop:necessaryA4}); the
excluded parameter values of the $\operatorname{S}_4$- and
$\operatorname{A}_5$-loci are not computed when the theorems give them
only structurally; and the equation is a formatted string, not a
\texttt{PolynomialRing} element.

\subsection*{Availability}

The package \texttt{Curves.m}, the companion file \texttt{Tables.m},
the driver scripts that generate
Tables~\ref{tab:dihedral-configs-LdZ}--\ref{tab:a5-configs}, and worked
example sessions are deposited in a permanent public repository:
\begin{center}
\url{https://doi.org/10.5281/zenodo.23839850} \quad (tag \texttt{v1.0}).
\end{center}
The scripts
regenerate the tables of this section from the enumeration of
Theorems~\ref{mainresultII}--\ref{mainresultIV} verbatim.

The package has been tested on MAGMA V2.12-5 and uses only the
procedures and directives available in V2.12 and later (the
\texttt{load} directive, the \texttt{Sprintf} and \texttt{printf}
formatters, the \texttt{for}, \texttt{while} and \texttt{if} control
structures, and list and tuple operations). Later versions of MAGMA
also run the package without modification.

\subsection*{Validation against the literature}

The following low-degree instances agree with existing
classifications, providing an additional consistency check on the
enumeration:
for $d=4$, the $\operatorname{A}_4$ package recovers the unique smooth
quartic with automorphism group $\widetilde{\operatorname A}_4\rtimes C_2$,
in agreement with Henn's classification \cite{henn1976automorphismengruppen};
for $d=5,7$, the $\operatorname{D}_m$ package produces the curve
$\mathcal X: Z^d+XY\cdot\mathcal T_{d-2,-}=0$ with automorphism group
$C_d\times\operatorname{D}_{d-2}$, matching the classifications of
Badr and Bars \cite{MR3508302} (quintics) and Badr, El-Guindy, and
Kamel \cite{BadrGamal} (septics);
for $d=6$, the $\operatorname{S}_4$ package yields the sextic
$\mathcal X: Z^6+\mathcal S_6=0$ with automorphism group
$C_6\circ\operatorname{GL}_2(\mathbb F_3)$, confirming
\cite{BadrBarssextic};
and for $d=12,20,30,32$, the $\operatorname{A}_5$ package recovers the
corresponding examples of Harui, Miura, and Ohbuchi
\cite{HaruiMiuraOh}.
{\small
\renewcommand{\arraystretch}{1.2}
\begin{longtable}{@{}c c p{3.1cm} p{2.7cm} p{4.6cm}@{}}
\caption{Dihedral curve configurations for degrees $4\le d\le 15$.\label{tab:dihedral-configs-LdZ}}\\
\hline
$d$ & $(m,\delta)$ & $L_{d,Z}$ & $\Aut(\mathcal X)$ & Remarks \\
\hline
\endfirsthead
\multicolumn{5}{c}{{\tablename\ \thetable{} -- continued from previous page}}\\
\hline
$d$ & $(m,\delta)$ & $L_{d,Z}$ & $\Aut(\mathcal X)$ & Remarks \\
\hline
\endhead
\hline
\multicolumn{5}{r}{\footnotesize Continued on next page}\\
\endfoot
\hline
\endlastfoot

5 & $(3,1)$ & $XY\cdot\mathcal T_{3,-}$ & $C_5\times\operatorname{D}_3$ & \cite{MR3508302} \\
\hline
7 & $(5,1)$ & $XY\cdot\mathcal T_{5,-}$ & $C_7\times\operatorname{D}_5$ & \cite{BadrGamal} \\
\hline
\multirow{2}{*}{8} & $(3,0)$ & $XY\cdot\mathcal T_{6,a_1}$ & $C_8\times\operatorname{D}_3$ & $a_1\neq\pm2$; $a_1\neq0$ (E1); $a_1\neq\pm\tfrac{7\sqrt2}{4}\zeta_4$ (E2) \\
 & $(6,1)$ & $XY\cdot\mathcal T_{6,-}$ & $C_{48}\rtimes_{41} C_2$ & -- \\
\hline
9 & $(7,1)$ & $XY\cdot\mathcal T_{7,-}$ & $C_9\times\operatorname{D}_7$ & -- \\
\hline
\multirow{2}{*}{10} & $(4,0)$ & $XY\cdot\mathcal T_{8,a_1}$ & $(C_{10}\circ\operatorname{D}_4)\rtimes C_2$ & $a_1\neq\pm2$; $a_1\neq0$ (E1); possibly E2 \\
 & $(8,1)$ & $XY\cdot\mathcal T_{8,-}$ & $C_5\times\operatorname{Dic}_8$ & -- \\
\hline
\multirow{2}{*}{11} & $(3,1)$ & $XY\cdot\mathcal T_{3,-}\cdot\mathcal T_{6,a_1}$ & $C_{11}\times\operatorname{D}_3$ & $a_1\neq\pm2$ \\
 & $(9,1)$ & $XY\cdot\mathcal T_{9,-}$ & $C_{11}\times\operatorname{D}_9$ & -- \\
\hline
\multirow{2}{*}{12} & $(5,0)$ & $XY\cdot\mathcal T_{10,a_1}$ & $C_{12}\times\operatorname{D}_5$ & $a_1\neq\pm2$; $a_1\neq0$ (E1); $a_1\neq-11\zeta_4$ (E3) \\
 & $(10,1)$ & $XY\cdot\mathcal T_{10,-}$ & $C_{120}\rtimes_{109} C_2$ & -- \\
\hline
13 & $(11,1)$ & $XY\cdot\mathcal T_{11,-}$ & $C_{13}\times\operatorname{D}_{11}$ & -- \\
\hline
\multirow{5}{*}{14} & $(3,0)$ & $XY\cdot\mathcal T_{6,a_1}\cdot\mathcal T_{6,a_2}$ & $C_{14}\times\operatorname{D}_3$ & $a_1\neq\pm2$, $a_1\neq a_2$; $A\neq-A$ (E1); $\{a_1,a_2\}\neq\{u,-\tfrac{20}{7}u\}$, $u^2=-\tfrac{49}{8}$ (E2) \\
 & $(3,1)$ & $XY\cdot\mathcal T_{6,-}\cdot\mathcal T_{6,a_1}$ & $C_{14}\circ\operatorname{Dic}_3$ & $a_1\neq\pm2$; $a_1\neq0$ (E1); possibly E2 \\
 & $(4,1)$ & $XY\cdot\mathcal T_{4,-}\cdot\mathcal T_{8,a_1}$ & $C_7\times\operatorname{Dic}_4$ & $a_1\neq\pm2$; $a_1\neq14$ (E2) \\
 & $(6,0)$ & $XY\cdot\mathcal T_{12,a_1}$ & $(C_{14}\circ\operatorname{D}_6)\rtimes C_2$ & $a_1\neq\pm2$; $a_1\neq0$ (E1) \\
 & $(12,1)$ & $XY\cdot\mathcal T_{12,-}$ & $C_7\times\operatorname{Dic}_{12}$ & -- \\
\hline
15 & $(13,1)$ & $XY\cdot\mathcal T_{13,-}$ & $C_{15}\times\operatorname{D}_{13}$ & -- \\
\hline
\end{longtable}
}

{\small
\renewcommand{\arraystretch}{1.2}
\begin{longtable}{@{}c c p{3.9cm} p{2.7cm} p{4.6cm}@{}}
\caption{$\operatorname{A}_4$ curve configurations for degrees $4\le d\le 32$ with $\Aut_{\mathrm{red}}(\mathcal X)=\operatorname{A}_4$ (i.e.\ with the $\operatorname{S}_4$- and $\operatorname{A}_5$-loci of Theorem~\ref{mainresultIV} removed). Here $\Aut(\mathcal X)\cong C_{d_0}\times G_0$, so $\lvert\Aut(\mathcal X)\rvert=12d$.\label{tab:a4-configs}}\\
\hline
$d$ & $(\varepsilon_+,\varepsilon_-,d_0)$ & $L_{d,Z}$ & $G_0$ & Remarks \\
\hline
\endfirsthead
\multicolumn{5}{c}{{\tablename\ \thetable{} -- continued from previous page}}\\
\hline
$d$ & $(\varepsilon_+,\varepsilon_-,d_0)$ & $L_{d,Z}$ & $G_0$ & Remarks \\
\hline
\endhead
\hline
\multicolumn{5}{r}{\footnotesize Continued on next page}\\
\endfoot
\hline
\endlastfoot

4 & $(0,1,1)$ & $\mathcal S_{4,-}$ & $\widetilde{\operatorname A}_4\rtimes C_2$ & \cite{henn1976automorphismengruppen} \\
\hline
10 & $(0,1,5)$ & $\mathcal S_6\cdot\mathcal S_{4,-}$ & $\widetilde{\operatorname A}_4$ & -- \\
\hline
12 & $(0,0,3)$ & $\mathcal F_{12,a_1}$ & $\widetilde{\operatorname A}_4\rtimes C_2$ &
\shortstack[l]{$a_1\neq\pm6\sqrt3\zeta_4$;$a_1\neq0$ ($\operatorname{S}_4$); \\$a_1\neq\pm\tfrac{22}{\sqrt5}$ ($\operatorname{A}_5$)} \\
\hline
16 & $(0,1,1)$ & $\mathcal S_{4,-}\cdot\mathcal F_{12,a_1}$ & $\widetilde{\operatorname A}_4\,{}^\bullet\,C_8$ & $a_1\neq\pm6\sqrt3\zeta_4$ \\
\hline
18 & $(0,0,9)$ & $\mathcal S_6\cdot\mathcal F_{12,a_1}$ & $\widetilde{\operatorname A}_4$ &
\shortstack[l]{$a_1\neq\pm6\sqrt3\zeta_4$;$a_1\neq0$ ($\operatorname{S}_4$)} \\
\hline
20 & $(1,1,5)$ & $\mathcal S_{4,+}\cdot\mathcal S_{4,-}\cdot\mathcal F_{12,a_1}$ & $\widetilde{\operatorname A}_4\rtimes C_2$ &
\shortstack[l]{$a_1\neq\pm6\sqrt3\zeta_4$;$a_1\neq0$ ($\operatorname{S}_4$);\\possibly $\operatorname{A}_5$} \\
\hline
22 & $(0,1,11)$ & $\mathcal S_6\cdot\mathcal S_{4,-}\cdot\mathcal F_{12,a_1}$ & $\widetilde{\operatorname A}_4$ & $a_1\neq\pm6\sqrt3\zeta_4$ \\
\hline
24 & $(0,0,3)$ & $\mathcal F_{12,a_1}\cdot\mathcal F_{12,a_2}$ & $\widetilde{\operatorname A}_4\,{}^\bullet\,C_4$ &
\shortstack[l]{$a_i\neq\pm6\sqrt3\zeta_4$, $a_1\neq a_2$;\\$A\neq-A$ ($\operatorname{S}_4$)} \\
\hline
26 & $(1,1,13)$ & $\mathcal S_6\cdot\mathcal S_{4,+}\cdot\mathcal S_{4,-}\cdot\mathcal F_{12,a_1}$ & $\widetilde{\operatorname A}_4$ &
\shortstack[l]{$a_1\neq\pm6\sqrt3\zeta_4$;$a_1\neq0$ ($\operatorname{S}_4$)} \\
\hline
28 & $(0,1,7)$ & $\mathcal S_{4,-}\cdot\mathcal F_{12,a_1}\cdot\mathcal F_{12,a_2}$ & $\widetilde{\operatorname A}_4\rtimes C_2$ &
\shortstack[l]{$a_i\neq\pm6\sqrt3\zeta_4$, $a_1\neq a_2$} \\
\hline
30 & $(0,0,15)$ & $\mathcal S_6\cdot\mathcal F_{12,a_1}\cdot\mathcal F_{12,a_2}$ & $\widetilde{\operatorname A}_4$ &
\shortstack[l]{$a_i\neq\pm6\sqrt3\zeta_4$, $a_1\neq a_2$;$A\neq-A$ ($\operatorname{S}_4$);\\possibly $\operatorname{A}_5$} \\
\hline
32 & $(1,1,1)$ & $\mathcal S_{4,+}\cdot\mathcal S_{4,-}\cdot\mathcal F_{12,a_1}\cdot\mathcal F_{12,a_2}$ & $\widetilde{\operatorname A}_4\,{}^\bullet\,C_{16}$ &
\shortstack[l]{$a_i\neq\pm6\sqrt3\zeta_4$, $a_1\neq a_2$;$A\neq-A$ ($\operatorname{S}_4$);\\possibly $\operatorname{A}_5$} \\
\hline
\end{longtable}
\begin{flushleft}
\footnotesize
\textbf{Note.} Rows at $d=6,8,14$ with $t=0$ and $\varepsilon_+=\varepsilon_-$ lie entirely on the $\operatorname{S}_4$-locus and appear in Table~\ref{tab:s4-configs}, not here.
\end{flushleft}
}
{\small
\renewcommand{\arraystretch}{1.2}
\begin{longtable}{@{}c c p{4.7cm} p{2.9cm} p{3.3cm}@{}}
\caption{$\operatorname{A}_5$ curve configurations for degrees $12\le d\le 100$. Here $\Aut(\mathcal X)\cong C_{d_0}\times G_0$ with $G_0=\widetilde{\operatorname A}_5\circ C_{2^e}$, so $\lvert\Aut(\mathcal X)\rvert=60d$.\label{tab:a5-configs}}\\
\hline
$d$ & $(\varepsilon,\varepsilon',\varepsilon'',t)$ & $L_{d,Z}$ & $C_{2^e}$ & Remarks \\
\hline
\endfirsthead
\multicolumn{5}{c}{{\tablename\ \thetable{} -- continued from previous page}}\\
\hline
$d$ & $(\varepsilon,\varepsilon',\varepsilon'',t)$ & $L_{d,Z}$ & $C_{2^e}$ & Remarks \\
\hline
\endhead
\hline
\multicolumn{5}{r}{\footnotesize Continued on next page}\\
\endfoot
\hline
\endlastfoot

12 & $(1,0,0,0)$ & $\mathcal G_{12,-11\zeta_4}$ & $C_4$ & \multirow{4}{*}{\cite{HaruiMiuraOh}} \\
\cline{1-4}
20 & $(0,1,0,0)$ & $\mathcal G_{20,228\zeta_4,-494}$ & $C_4$ &  \\\cline{1-4}
30 & $(0,0,1,0)$ & $\mathcal G_{30,-522\zeta_4,10005}$ & $C_2$ &  \\\cline{1-4}
32 & $(1,1,0,0)$ & $\mathcal G_{12,-11\zeta_4}\cdot\mathcal G_{20,228\zeta_4,-494}$ & $C_{32}$ & \\
\hline
42 & $(1,0,1,0)$ & $\mathcal G_{12,-11\zeta_4}\cdot\mathcal G_{30,-522\zeta_4,10005}$ & $C_2$ & -- \\
\hline
50 & $(0,1,1,0)$ & $\mathcal G_{20,228\zeta_4,-494}\cdot\mathcal G_{30,-522\zeta_4,10005}$ & $C_2$ & -- \\
\hline
60 & $(0,0,0,1)$ & $\mathcal G_{60,a_1}$ & $C_4$ & $a_1\notin\{a_1^\ast,a_2^\ast\}$ \\
\hline
62 & $(1,1,1,0)$ & $\mathcal G_{12,-11\zeta_4}\cdot\mathcal G_{20,228\zeta_4,-494}\cdot\mathcal G_{30,-522\zeta_4,10005}$ & $C_2$ & -- \\
\hline
72 & $(1,0,0,1)$ & $\mathcal G_{12,-11\zeta_4}\cdot\mathcal G_{60,a_1}$ & $C_8$ & $a_1\notin\{a_1^\ast,a_2^\ast\}$ \\
\hline
80 & $(0,1,0,1)$ & $\mathcal G_{20,228\zeta_4,-494}\cdot\mathcal G_{60,a_1}$ & $C_{16}$ & $a_1\notin\{a_1^\ast,a_2^\ast\}$ \\
\hline
90 & $(0,0,1,1)$ & $\mathcal G_{30,-522\zeta_4,10005}\cdot\mathcal G_{60,a_1}$ & $C_2$ & $a_1\notin\{a_1^\ast,a_2^\ast\}$ \\
\hline
92 & $(1,1,0,1)$ & $\mathcal G_{12,-11\zeta_4}\cdot\mathcal G_{20,228\zeta_4,-494}\cdot\mathcal G_{60,a_1}$ & $C_4$ & $a_1\notin\{a_1^\ast,a_2^\ast\}$ \\
\hline
\end{longtable}
\begin{flushleft}
\footnotesize
\textbf{Note.} Here $a_1^\ast := 2^2\cdot19\cdot436999$ and $a_2^\ast := -2^2\cdot9377\cdot5323$.
\end{flushleft}
}
{\small
\renewcommand{\arraystretch}{1.2}
\begin{longtable}{@{}c c p{4.0cm} p{3.3cm} p{3.5cm}@{}}
\caption{$\operatorname{S}_4$ curve configurations for degrees $4\le d\le 30$. Here $\Aut(\mathcal X)\cong C_d\circ G$ with $G\in\{\operatorname{GL}_2(\mathbb F_3),\operatorname{CSU}_2(\mathbb F_3),\widetilde{\operatorname A}_4\,{}^\bullet\,C_2\}$, so $\lvert\Aut(\mathcal X)\rvert=24d$.\label{tab:s4-configs}}\\
\hline
$d$ & $(\varepsilon,\varepsilon',\varepsilon'')$ & $L_{d,Z}$ & $G$ & Remarks \\
\hline
\endfirsthead
\multicolumn{5}{c}{{\tablename\ \thetable{} -- continued from previous page}}\\
\hline
$d$ & $(\varepsilon,\varepsilon',\varepsilon'')$ & $L_{d,Z}$ & $\Aut(\mathcal X)$ & Remarks \\
\hline
\endhead
\hline
\multicolumn{5}{r}{\footnotesize Continued on next page}\\
\endfoot
\hline
\endlastfoot

6 & $(1,0,0)$ & $\mathcal S_6$ & $\operatorname{GL}_2(\mathbb F_3)$ & \cite{BadrBarssextic} \\
\hline
8 & $(0,0,1)$ & $\mathcal S_{8,14}$ & $\operatorname{GL}_2(\mathbb F_3)$ & -- \\
\hline
12 & $(0,1,0)$ & $\mathcal S_{12,-34}$ & $\widetilde{\operatorname A}_4\,{}^\bullet\,C_2$ & -- \\
\hline
14 & $(1,0,1)$ & $\mathcal S_6\cdot\mathcal S_{8,14}$ & $\operatorname{GL}_2(\mathbb F_3)$ & -- \\
\hline
18 & $(1,1,0)$ & $\mathcal S_6\cdot\mathcal S_{12,-34}$ & $\operatorname{CSU}_2(\mathbb F_3)$ & -- \\
\hline
20 & $(0,1,1)$ & $\mathcal S_{12,-34}\cdot\mathcal S_{8,14}$ & $\widetilde{\operatorname A}_4\,{}^\bullet\,C_2$ & -- \\
\hline
24 & $(0,0,0)$ & $\mathcal S_{24,a_1}$ & $\operatorname{GL}_2(\mathbb F_3)$ & $a_1\notin\{-66,42\}$ \\
\hline
26 & $(1,1,1)$ & $\mathcal S_6\cdot\mathcal S_{12,-34}\cdot\mathcal S_{8,14}$ & $\operatorname{CSU}_2(\mathbb F_3)$ & -- \\
\hline
30 & $(1,0,0)$ & $\mathcal S_6\cdot\mathcal S_{24,a_1}$ & $\operatorname{GL}_2(\mathbb F_3)$ & $a_1\notin\{-66,42\}$ \\
\hline
\end{longtable}
}
%%%%%%%%%%%%%%%%%%%%%%%%%%%%%%%%%%%%%%%%%%%%%%%%%%%%%%%%%%%%%%%%%%

\end{document}